\documentclass[11pt,a4paper,utf8x]{article}
\usepackage{color}
\usepackage{fancyhdr}
\usepackage{tikz}
\usepackage[vmargin=100pt, hmargin=80pt]{geometry}
\usetikzlibrary{decorations.pathmorphing}
\usepackage[T1]{fontenc}
\usepackage{ucs}
\usepackage[utf8x]{inputenc}
\usepackage{amsfonts}
\usepackage{listings}
\usepackage{amsmath}
\usepackage[mathscr]{euscript}
\usepackage{amssymb}
\usepackage{amsfonts}
\usepackage{dsfont}
\usepackage{amsthm}
\usepackage{fancybox}
\usepackage{enumitem}
\usepackage{dsfont}
\usepackage{algpseudocode}
\usepackage{algorithm}
\usepackage{graphicx}
\usepackage{caption}
\usepackage{subcaption}
\usepackage{float}
\usepackage{authblk}

\DeclareMathOperator{\Tr}{Tr}
\DeclareMathOperator{\diag}{diag}

\DeclareMathOperator{\MSE}{MSE}
\newcommand*{\symm}{\textbf{${\mathcal{S}_d }$}}
\newcommand*{\asymm}{\textbf{${\mathcal{A}_d }$}}

\newcommand*{\genm}{\textbf{${\mathcal{M}_d }$}}

\newcommand*{\dpos}{\textbf{${\mathcal{S}_d^{+,*}}$}}
\newcommand*{\posm}{\textbf{${\mathcal{S}_d^+}$}}

\newcommand*{\rbi}{\overline R_{\infty}}
\newcommand*{\qbi}{\overline Q_{\infty}}

\newcommand*{\E}{\mathbb{E}}
\newcommand*{\R}{\mathbb{R}}
\newcommand*{\N}{\mathbb{N}}
\newcommand*{\cF}{\mathcal{F}}
\newcommand*{\cL}{\mathcal{L}}
\newcommand*{\Px}{\mathbb{P}}
\newcommand*{\tp}{^\top}
\newcommand*{\geb}{\geqslant}
\newcommand*{\leb}{\leqslant}

\newenvironment{alphafootnotes}
  {\par\edef\savedfootnotenumber{\number\value{footnote}}
   
   \setcounter{footnote}{0}}
  {\par\setcounter{footnote}{\savedfootnotenumber}}

\newtheorem{mytheo}{Theorem}[section]
\newtheorem{prop}{Proposition}[section]
\newtheorem{lemme}{Lemma}[section]
\newtheorem{coro}{Corollary}[section]
\newtheorem{rmrk}{Remark}[section]

\begin{document}
\title{Maximum Likelihood Estimation for Wishart processes}
\author[1]{Aur\'elien Alfonsi}
\author[2]{Ahmed Kebaier}
\author[1]{Cl\'ement Rey}
\affil[1]{CERMICS, \'Ecole des Ponts, UPE, Projet MathRisk ENPC-INRIA-UMLV, Champs-sur-Marne, France}
\affil[2]{Universit\'{e} Paris 13, Sorbonne Paris Cit\'{e}, LAGA, CNRS (UMR 7539)}
\maketitle
\begin{alphafootnotes}
\footnote{e-mails : alfonsi@cermics.enpc.fr, kebaier@math.univ-paris13.fr, reyc@cermics.enpc.fr. This research benefited from the support of the ``Chaire Risques Financiers'', Fondation du Risque and the Laboratories of Excellence B\'ezout http://bezout.univ-paris-est.fr/ and MME-DII http://labex-mme-dii.u-cergy.fr/ .}
\end{alphafootnotes}

\begin{abstract}
In the last decade, there has been a growing interest to use Wishart processes for modelling, especially for financial applications. However, there are still few studies on the estimation of its parameters. Here, we study the Maximum Likelihood Estimator (MLE) in order to estimate the drift parameters of a Wishart process. We obtain precise convergence rates and limits for this estimator in the ergodic case and in some nonergodic cases.  We check that the MLE achieves the optimal convergence rate in each case. Motivated by this study, we also present new results on the Laplace transform that extend the recent findings of Gnoatto and Grasselli~\cite{GnoattoGrasselli} and are of independent interest. 
\end{abstract}


\section{Introduction and preliminary results}

The goal of this paper is to study the maximum likelihood estimation of the parameters of Wishart processes. These processes have been introduced by Bru~\cite{Bru} and take values in the set of positive semidefinite matrices. Let  $d \in \N^*$ denote the dimension,  $\mathcal{M}_d$ be the set of real $d$-square matrices, $\posm$ (resp. $\dpos$) be the subset of positive semidefinite (resp. definite) matrices, $\symm$ (resp. $\asymm$) the subset of symmetric (resp. antisymmetric) matrices. Wishart processes are defined by the following SDE
\begin{eqnarray}\label{SDE_Intro}
\left\{
    \begin{array}{ll}
        dX_t=\left[\alpha a\tp  a+bX_t+X_tb\tp  \right]dt+ \sqrt{X}_tdW_t a+a\tp  dW_t \tp \sqrt{X}_t , \quad t>0 \\
        X_0=x \in \posm,
    \end{array}
\right.
\end{eqnarray}
where $\alpha \geqslant d-1$, $ a \in \mathcal{M}_d$, $b \in \mathcal{M}_d$ and   $(W_t)_{t \geqslant 0}$ denotes a $d$-square matrix made of independent Brownian motions. We recall that for $x \in \posm$, $\sqrt{x}$ is the unique matrix in $\posm$ such that $\sqrt{x}^2=x$. It is shown by Bru~\cite{Bru} and Cuchiero et al.~\cite{CFMT} in a more general affine setting that the SDE~\eqref{SDE_Intro} has a unique strong solution when $\alpha \geqslant d+1$ and a unique weak solution when  $\alpha \geqslant d-1$. Besides, we have $X_t\in \dpos$ for any $t\geb 0$ when $x\in \dpos$ and $\alpha \geb d+1$. In this paper, we will denote by $WIS_d(x,\alpha,b,a)$ the law of $(X_t,t\geb 0)$ and  $WIS_d(x,\alpha,b,a;t)$ the law of $X_t$. In dimension~$d=1$, Wishart processes are known as Cox-Ingersoll-Ross processes in the literature. It is worth recalling that the law of $X$ only depends on $a$ through $a\tp  a$ since we have
$$WIS_d(x,\alpha,b,a)\underset{law}{=} WIS_d(x,\alpha,b,\sqrt{a\tp a}),$$
see e.g. equation~(12) in~\cite{AA2013}. Therefore, the parameters to estimate are $\alpha$, $b$ and $a\tp a$.

Wishart processes have been originally considered by Bru~\cite{Bruthesis}  to model some biological data. Recently, they have been widely used in financial models in order to describe the evolution of the dependence between assets. Namely, Gourieroux and Sufana~\cite{GourierouxSufana} and Da Fonseca et al.~\cite{Dafonseca} have proposed a stochastic volatility model for  a basket of assets that assumes that the instantaneous covariance between the assets follows a Wishart process. This extends the well-known Heston model~\cite{Heston} to many assets. Wishart processes have also been used for interest rates models. Affine term structure models involving these processes have been proposed for example by Gourieroux and Sufana~\cite{GourierouxSufana2}, Gnoatto~\cite{Gnoatto} and Ahdida et al.~\cite{AAP}. For these models, the question of estimating the parameters of the underlying Wishart process may be important for practical purposes and should be possible thanks to the profusion of financial data. This issue has been considered by Da Fonseca et al.~\cite{DFGI} for the model presented in~\cite{Dafonseca}. However, there is no dedicated study on the Maximum Likelihood Estimator (MLE) for Wishart processes. For the Cox-Ingersoll-Ross process, the estimation of parameters has been studied earlier, motivated in particular by its use for interest rates (see Fournié and Talay~\cite{FournieTalay}). Later on, the MLE has been studied by Overbeck~\cite{Overbeck} including some nonergodic cases, and more recently by Ben Alaya and Kebaier~\cite{BK2011,Alaya_Kebaier_2013}. This paper completes the literature by studying the MLE for Wishart processes.

In this paper, we will follow the theory developed in the books by Lipster and Shiryaev~\cite{LS_Book} and Kutoyants~\cite{KutoyantsBook} and assume that we observe the full path $(X_t, t\in [0,T])$ up to time~$T>0$. This choice will be convenient from a mathematical point of view to study the convergence of the MLE. Of course, in practice it can be relevant to study precisely the estimation when we only observe the process on a discrete time-grid. This is left for further research, but we already observe in our numerical experiments that the discrete approximation of the MLE gives a satisfactory estimation of Wishart parameters (see Section~\ref{Sec_Num_Study}). It is worth noticing that once we observe the path $(X_t, t\in [0,T])$, the parameter~$a\tp a$ is known. In fact, we can calculate the quadratic covariation (see for example Lemma 2 in~\cite{AA2013}) and get for $i,j,k,l \in  \{1, \cdots,d \}$
\begin{align}
\label{eq:crochet_general_de_X}
\langle X_{i,j}, X_{k,l} \rangle_T =&\int_0^T (a\tp a )_{j,l}(X_s)_{i,k}+(a\tp a )_{j,k}(X_s)_{i,l} +(a\tp a )_{i,l}(X_s)_{j,k}+(a\tp a )_{i,k}(X_s)_{j,l}ds.
\end{align}
This leads to 
\begin{align}
\label{eq:estim_crochet_a}
(a\tp a)_{i,i}&=\frac{1}{4} \langle X_{i,i}\rangle_T  \Big( \int_0^T (X_s)_{i,i} ds \Big)^{-1}, \\
(a\tp a)_{i,j}&=\left(\frac{1}{2} \langle  X_{i,j}, X_{i,i}\rangle_T - (a\tp a)_{i,i} \int_0^T (X_s)_{i,j} ds \right)\Big( \int_0^T (X_s)_{i,i} ds \Big)^{-1}, \nonumber
\end{align}
for $1\leb i,j\leb d$ and $j\not=i$. We note that these quantities are well defined as soon as the path $(X_t, t\in [0,T])$ has a finite quadratic variation and is such that $X_t\in \dpos$ $dt$-a.e., which is satisfied by the paths of Wishart processes (see Proposition~4 in~\cite{Bru}). We will assume that $a\tp a \in \dpos$ and denote by $a\in \genm$ an invertible matrix that matches the observed value of $a\tp a $: $a$ can be for example the square root of $a\tp a$ or the Cholesky decomposition of $a\tp a $. Then, we know that 
$Y_t=(a\tp)^{-1}X_ta^{-1}$ follows the law  $WIS_d((a\tp)^{-1}xa^{-1},\alpha,(a\tp)^{-1}ba\tp,I_d)$, see e.g. equation~(13) in~\cite{AA2013}. It is therefore sufficient to focus on the estimation of the parameters $\alpha$ and $b$ when $a=I_d$, which we consider now.

We first present the MLE of $\theta=(b,\alpha)$, and we denote by $\Px_\theta$   the original probability measure under which $X$ satisfies
\begin{equation}\label{SDE_Intro2}
        dX_t=\left[\alpha I_d +bX_t+X_tb\tp  \right]dt+ \sqrt{X}_tdW_t + dW_t \tp \sqrt{X}_t .
\end{equation}
When no confusion is possible, we also denote $\Px$ this probability.
We consider $\alpha_0 \geb d+1$ and set $\theta_0=(\alpha_0,0)$. We will assume for the joint estimation of $\alpha$ and $b$ that
\begin{equation}\label{cond_estim_j}
\alpha \geb d+1 \text{ and } x\in \dpos.
\end{equation}
The latter assumption is not restrictive in practice since the condition $\alpha \geb d+1$ ensures that $X_t \in \dpos$ for any $t>0$. Due to this assumption, we know by Theorem~4.1 in Mayerhofer~\cite{MayerhoferHammamet} that
\begin{align*}
\frac{d {\mathbb{P}}_{\theta_0,T}}{d\mathbb{P}_{\theta,T}}
&:=\exp\left({\int_0^T \Tr[H_sdW_s]-\frac{1}{2}\int_0^T \Tr[H_sH_s\tp ]ds}\right), \text{ with }
H_t=\frac{\alpha_0-\alpha}{2}(\sqrt{X_t})^{-1} -b \sqrt{X_t}
\end{align*}
defines a probability measure under which $\tilde{W}_t=W_t-\int_0^t H\tp_sds$ is a $d \times d$-Brownian motion, where $\mathbb{P}_{\theta,T}$ is the restriction of $\mathbb{P}_{\theta}$ to the $\sigma$-algebra $\sigma(W_s,s\in[0,T])$. We have
$$dX_t=\alpha_0 I_d dt+ \sqrt{X}_td\tilde{W}_t + d\tilde{W}_t \tp \sqrt{X}_t, $$
and the likelihood is then defined 
by (see Lipster and Shiryaev~\cite{LS_Book}, Chapter 7)
\begin{equation}\label{def_likelihood}L_T^{\theta,\theta_0}=\frac{1}{\E\left[ \exp\left({\int_0^T \Tr[H_sdW_s]-\frac{1}{2}\int_0^T \Tr[H_sH_s\tp ]ds}\right)\bigg| \cF^X_T \right]},
\end{equation}
where  $(\cF^X_t)_{t\geb 0}$ denote the filtration generated by the process~$X$.
\begin{prop}\label{prop_MLE_bgen} For $X\in \dpos$, let $\cL_X:\symm \rightarrow\symm$ be the linear application defined by $\cL_X(Y)=XY+YX$. It is invertible, and the likelihood of $(X_t,t\in[0,T])$ is given by
\begin{multline}
L_T^{\theta,\theta_0}=\exp\Bigl(\frac{\alpha-\alpha_0}{4}\log\left(\frac{\det[X_T]}{\det[x]}\right)
-\frac{\alpha-\alpha_0}{4}\big( \frac{\alpha+\alpha_0}{2}-1-d \big)\int_0^T \Tr[X_s^{-1}]ds-\frac{\alpha T}{2}\Tr[b] \\ +\frac{1}{2} \int_0^T \Tr\left[\mathcal{L}_{X_t}^{-1}\left(bX_t+X_tb\tp \right) dX_t \right]  -\frac{1}{4}  \int_0^T \Tr\left[\mathcal{L}_{X_t}^{-1}\left(bX_t+X_tb\tp \right)  (bX_t+X_tb\tp) \right]dt  \Bigr). \label{vs_gen_bgen}
\end{multline}
\end{prop}
\noindent Lemmas~\ref{lemme-linear} and~\ref{lemme-autoadjoint} states some properties of $\cL_X$, and the proof of Proposition~\ref{prop_MLE_bgen} is given in Appendix~\ref{App_proof_mle_bgen}. In particular, we see from this proof that $\frac{d {\mathbb{P}}_{\theta_0,T}}{d\mathbb{P}_{\theta,T}}\in \cF^X_T$ if, and only if $b\in \symm$, in which case the likelihood has the following simpler form
\begin{multline}
L_T^{\theta,\theta_0}=\exp\Bigl(\frac{\alpha-\alpha_0}{4}\log\left(\frac{\det[X_T]}{\det[x]}\right)+\frac{\Tr[bX_T]-{\Tr[bx]}}{2}-\frac{1}{2}\int_0^T \Tr[b^2X_s]ds
\\
-\frac{\alpha-\alpha_0}{4}\big( \frac{\alpha+\alpha_0}{2}-1-d \big)\int_0^T \Tr[X_s^{-1}]ds-\frac{\alpha T}{2}\Tr[b] \Bigr), \label{vs_gen_bsym}
\end{multline}
since $\mathcal{L}_{X_t}^{-1}\left(bX_t+X_tb \right)=b$.

Now, we want to maximize the likelihood and observe that the quantity in the exponential~\eqref{vs_gen_bgen} is quadratic with respect to $(b,\alpha)$ and goes almost surely to $-\infty$ when $\|(b,\alpha)\|\rightarrow +\infty$. To do so, we first remark that $\Tr[b]=\Tr[\mathcal{L}_{X_t}^{-1}\left(bX_t+X_tb\tp \right)]$ by Lemma~\ref{lemme-linear}. Then, Cauchy-Schwarz inequality yields to
\begin{align}
|\Tr[\alpha b]|&=\left|\frac{1}{T} \int_0^T \Tr\left[\sqrt{2}  \mathcal{L}_{X_s}^{-1}\left(bX_s+X_sb\tp \right) \sqrt{X_s} \frac{\alpha}{\sqrt{2}} \sqrt{X_s^{-1}}    \right]ds\right| \nonumber \\
&\leb \frac{1}{T}\int_0^T \Tr\left[(\mathcal{L}_{X_s}^{-1}\left(bX_s+X_sb\tp \right))^2 X_s \right]ds+ \frac{\alpha^2}{4}\frac{1}{T} \int_0^T \Tr\left[  X_s^{-1}\right]ds \\
&=\frac{1}{2T}\int_0^T \Tr\left[\mathcal{L}_{X_s}^{-1}\left(bX_s+X_sb\tp \right)(bX_s+X_sb\tp) \right]ds+ \frac{\alpha^2}{4}\frac{1}{T} \int_0^T \Tr\left[  X_s^{-1}\right]ds , \nonumber 
\end{align}
and it is strict almost surely, which gives that the quadratic form in the exponential~\eqref{vs_gen_bgen} is negative definite. 
There is thus a unique global maximum of~\eqref{vs_gen_bgen} on $\R \times \genm$. We know from Lemma~\ref{lemme-autoadjoint} that $\mathcal{L}_{X_s}^{-1}$ is self-adjoint, and we get with straightforward calculations that the MLE  $\hat\theta_T=(\hat b_T,\hat\alpha_T)$ is characterized by the following equations:
\begin{align}\label{pt_crit_emvgen} \begin{cases}
 \frac{1}{4}\log\left(\frac{\det[X_T]}{\det[x]}\right)
-\frac{\hat\alpha_T-1-d}{4}\int_0^T \Tr[X_s^{-1}]ds-\frac{T}{2}\Tr[\hat b_T] =0,\\
\int_0^T\mathcal{L}_{X_s}^{-1}(dX_s)X_s- \int_0^T\mathcal{L}_{X_s}^{-1}(\hat b_TX_s+X_s\hat b_T\tp)X_sds  -\frac{\hat\alpha_T T}{2}I_d=0.
  \end{cases}
\end{align}
Unless in the ergodic case, we will not be able to obtain convergence results for this estimator. Instead, we will mostly work with the MLE estimator when $b$ is known to be symmetric. This enables us to work with more tractable formulas, even if the calculations are already quite involved in case. Analyzing the general case would require development of further arguments.  Besides, we can consider that Wishart processes with $b$ symmetric already form an interesting family of processes that may be rich enough in many applications. When $b\in \symm$, the unique global maximum $\hat\theta_T=(\hat b_T, \hat\alpha_T)$ of~\eqref{vs_gen_bsym} on $\R \times \symm$   is   characterized by the following equations:
\begin{align}\label{pt_crit_emv} \begin{cases}
 \frac{1}{4}\log\left(\frac{\det[X_T]}{\det[x]}\right)
-\frac{\hat\alpha_T-1-d}{4}\int_0^T \Tr[X_s^{-1}]ds-\frac{T}{2}\Tr[\hat b_T] =0,\\
\frac{X_T-{x}}{2}
-\frac{1}{2}\int_0^T(\hat b_TX_s+X_s\hat b_T)ds-\frac{\hat\alpha_T T}{2}I_d=0.
  \end{cases}
\end{align}
To get more explicit formulas, we have to invert this linear system. For $X \in \symm$ and $a\in \R$, we define the linear applications
\begin{align}\label{def_LXa}
\mathcal{L}_X: \begin{array}[t]{l}  \symm \rightarrow \symm \\ Y \mapsto YX+XY
\end{array} \text{ and } \mathcal{L}_{X,a}: \begin{array}[t]{l}  \symm \rightarrow \symm \\ Y \mapsto YX+XY -2a\Tr[Y]I_d.
\end{array}
\end{align}
We introduce the following shorthand notation
\begin{align}
\label{eq:def_quantities}
R_T:=\int_0^TX_sds,\quad &Q_T:=\left(\int_0^T\Tr[X_s^{-1}]ds\right)^{-1}, \quad  Z_T:=\log\left(\dfrac{\det[X_T]}{\det[x]}\right),
\end{align}
and note that $Q_T$ and $Z_T$ are defined only for $\alpha \geb d+1$ while $R_T$ is defined for $\alpha\geb d-1$ and belongs almost surely to~$\dpos$.\footnote{This is obvious when $\alpha>d-1$ since $X_t \in \dpos$ a.s. by Proposition~4 in~\cite{Bru}. For $\alpha=d-1$, we would have by contradiction the existence of $v_T \in \cF_T^X$ such that $\forall t \in [0, T], v_T \tp X_t v_T=0$. This is clearly not possible by using the connection with matrix-valued Ornstein-Uhlenbeck in this case, see eq.~(5.7) in~\cite{Bru}.}
By using the convexity property of the inverse, see e.g. Mond and Pecaric~\cite{MondPecaric}, we have when $\alpha \geb d+1$
\begin{equation}\label{ineg_Q_R}\Tr\left[\left( \frac{R_T}{T}\right)^{-1}\right]< \frac{Q_T^{-1}}{T}  , \ a.s.
\end{equation}
We get $\hat\alpha_T=1+d +Q_T\big( Z_T - 2T\Tr[\hat b_T ] \big)$ and $\mathcal{L}_{R_T,T^2Q_T}(\hat b_T)=X_T-{x}-T\left(Q_TZ_T
+1+d\right) I_d $. By~\eqref{ineg_Q_R} and Lemma~\ref{lemme-linear}, the latter equation can be inverted, which leads to
\begin{align}
\label{eq:estimator_couple}
\left\{
\begin{array}{l l l}
\hat\alpha_T&=&
1+d+Q_T\Big(Z_T - 2T\Tr\big[\mathcal{L}^{-1}_{R_T,T^2Q_T}\left({X_T-{x}}-T\left[Q_TZ_T
+1+d\right] I_d\right)   \big]\Big) \\\\
\hat b_T
&=&\mathcal{L}^{-1}_{R_T,T^2Q_T}\left({X_T-{x}}-T\left[Q_TZ_T+1+d\right] I_d\right) . 
\end{array}
\right.
\end{align}

The estimator of $\alpha$ when $\alpha \in [d-1,d+1)$ given by the MLE is no longer well defined. The same thing already occurs in dimension $d=1$ for the CIR process, see Ben Alaya and Kebaier~\cite{BK2011}. However, it is still possible to estimate the parameter $b\in\genm$ when $\alpha \geb d-1$ is known.  In this case, we denote $\theta=(b,\alpha)$ and $\theta_0=(0,\alpha)$ and get by repeating the same arguments that
\begin{multline*}
 L_T^{\theta,\theta_0}=\exp\Bigl(\frac{1}{2} \int_0^T \Tr\left[\mathcal{L}_{X_t}^{-1}\left(bX_t+X_tb\tp \right) dX_t \right]  \\
-\frac{1}{4}  \int_0^T \Tr\left[\mathcal{L}_{X_t}^{-1}\left(bX_t+X_tb\tp \right)  (bX_t+X_tb\tp) \right]dt  -\frac{\alpha T}{2}\Tr[b]\Bigr),
\end{multline*}
and the MLE is characterized by 
\begin{equation} \label{MLE_bseul_gen} \int_0^T\mathcal{L}_{X_s}^{-1}(dX_s)X_s- \int_0^T\mathcal{L}_{X_s}^{-1}(\hat b_TX_s+X_s\hat b_T\tp)X_sds  -\frac{\alpha T}{2}I_d=0.
\end{equation}
When $b$ is known a priori to be symmetric, the likelihood and the MLE are then given by
\begin{align}\label{likeli_bseul}
L_T^{\theta,\theta_0}&=\exp\Bigl( \frac{\Tr[bX_T]-{\Tr[bx]}}{2}-\frac{1}{2}\int_0^T \Tr[b^2X_s]ds -\frac{\alpha T}{2}\Tr[b] \Bigr),\\
\hat{b}_T&=\mathcal{L}^{-1}_{R_T}\left(X_T-x-\alpha T I_d\right). \label{MLE_bseul}
\end{align}

The goal of the paper is to study the convergence of the MLE under the original probability~$\Px_\theta$. To do so, we first consider the case where the Wishart process is ergodic.  By Lemma~\ref{lem_ergo}, this holds if $-(b+b\tp)\in \dpos$ when $b\in \genm$, and the ergodicity is equivalent to $-b \in \dpos$ when $b\in\symm$.   Then, we can use Birkhoff's ergodic theorem to determine the convergence of the MLE. Section~\ref{Sec_MLE_ergo} presents these results for~\eqref{eq:estimator_couple} when $\alpha\geb d+1$, for \eqref{vs_gen_bgen} when $\alpha>d+1$ and for both~\eqref{MLE_bseul} and~\eqref{MLE_bseul_gen} when $\alpha\geb d-1$. Section~\ref{Sec_nonergo} studies the convergence of the MLE in some nonergodic cases, namely when $b=\lambda_0 I_d$ with $\lambda_0\geb 0$ and when $b$ is known to be symmetric. More precisely, when $b=0$, we obtain convergence results  for~\eqref{eq:estimator_couple} when $\alpha\geb d+1$ and for~\eqref{MLE_bseul} when $\alpha\geb d-1$. When $\lambda_0>0$, we only obtain  convergence results for~\eqref{MLE_bseul} when $\alpha\geb d-1$. In all these cases, we analyse the convergence by the mean of Laplace transforms. Though limited to some nonergodic cases, we however recover and extend the recent convergence results obtained by Ben Alaya and Kebaier for the one-dimensional CIR process~\cite{BK2011,Alaya_Kebaier_2013}. In Section~\ref{Sec_LAN}, we check that the MLE achieves the optimal rate of convergence in the different cases by proving local asymptotic properties. Last, we study in Section~\ref{Sec_laplace} the Laplace transform of $(X_T,R_T)$. This study can be of independent interest and improves the recent results of Gnoatto and Grasselli~\cite{GnoattoGrasselli}.

\section[Statistical Inference:  the ergodic case]{Statistical Inference of the Wishart process:  the ergodic case}\label{Sec_MLE_ergo}

When $-(b+b\tp) \in \dpos$, the Wishart process $X_t$ converges  in law when $t\rightarrow +\infty$ to the stationary law $X_\infty\sim WIS_d(0,\alpha,0,\sqrt{2q_\infty};1/2)$  with $ q_\infty=\int_0^\infty e^{sb}e^{s b \tp}ds$  for any starting point $x \in \posm$ by Lemma~\ref{lem_ergo}. Therefore this is the unique stationary law which is thus extremal, and we know by Stroock (\cite{Stroock}, Theorem 7.4.8) that it is then ergodic, see also Pag\`es~\cite{Pages_ergo}, Annex A. We introduce the following quantity
$$\rbi:=\mathbb E_\theta (X_\infty).$$
 From the ergodic Birkhoff's theorem, we have 
\begin{equation}\label{Birk1}
\frac{R_T}{T}\overset{a.s.}{\longrightarrow}
\rbi ,\quad\mbox{ as
}T\rightarrow+\infty.
\end{equation}
Besides, when $\alpha \geb d+1$,  $\qbi=\frac{1}{\mathbb E_\theta(\Tr[X_\infty^{-1}])}$ is finite and satisfies 
\begin{equation}\label{ineg_qbi}\qbi \Tr[\rbi^{-1}]< 1,
\end{equation}
due to the convexity property of the inverse, see e.g. Mond and Pecaric~\cite{MondPecaric}.
Again, the ergodic Birkhoff's theorem gives
\begin{equation}\label{Birk2}
 TQ_T\overset{a.s.}{\longrightarrow}
\qbi=\frac{1}{\mathbb E_\theta(\Tr[X_\infty^{-1}])},\quad\mbox{ as
}T\rightarrow+\infty.
\end{equation}
This section is organized as follows. First, we study the MLE~\eqref{eq:estimator_couple} when $b$ is known to be symmetric in the cases  $\alpha> d+1$ and $\alpha=d+1$. Then, we focus on the MLE~\eqref{pt_crit_emvgen} when $b\in\genm$ and $\alpha>d+1$. The analysis follows the same steps and reuses some calculations made in the symmetric case. Last, we study the convergence of the MLE when $\alpha\geb d-1$ is known, in both symmetric and general cases. 

\subsection{The global MLE estimator of $\theta=( b,\alpha)$ when $b$ is known to be symmetric}
When $b\in \symm$, the ergodicity is by Lemma~\eqref{lem_ergo} equivalent to $-b\in \dpos$, which we assume in this subsection. We have $X_\infty\sim WIS_d(0,\alpha,0,\sqrt{-b^{-1}};1/2)$ and it is easy to get from~\eqref{SDE_Intro2} that $\alpha I_d +b \rbi +\rbi b =0$, which gives $\rbi=-\frac{\alpha}{2 } b^{-1} \in \dpos$. We will also show in the proof of Theorem~\ref{thm_MLE_global_ergo} that
\begin{equation}\label{val_qbi} \qbi:=\frac{\alpha-(1+d)}{2 \Tr[-b]}.
\end{equation}

We consider the convergence of the MLE given by~\eqref{eq:estimator_couple} when $\alpha\geb d+1$. We introduce the following martingales:
\begin{align}
\label{eq:def_Mt}
M_t &:= \int_0^t\sqrt{X_s}d{W}_s+\int_0^td{W}\tp _s
\sqrt{X}_s, \\
N_t&:=\int_0^t \Tr[(\sqrt{X_s})^{-1}d {W}_s].\label{eq:def_Nt}
\end{align}
We use the dynamics of $(X_t)_{t \geqslant 0}$  under  ${\mathbb P}_{\theta}$ and Itô's formula for  $(Z_t)_{t \geqslant 0}$ (see e.g. Bru~\cite{Bru}, equation~(2.6)) to get on the one hand
\begin{align} \label{eq_Z}
X_T&=x+\alpha T I_d+\mathcal{L}_{R_T}(b) + M_T, &
Z_T&=(\alpha-1-d) Q_T^{-1}+2\Tr[b]T +2 N_T.
\end{align}
On the other hand, we obtain from~\eqref{pt_crit_emv} and~\eqref{eq:def_quantities} that $X_T=x+ \hat \alpha_T T I_d+\mathcal{L}_{R_T}(\hat b_T)$ and  $Z_T=(\hat \alpha_T-1-d) Q_T^{-1}+2T\Tr[\hat b_T]$, which yields to
\begin{equation}\label{erreur_mle}
\left\{
\begin{array}{r c l}
\hat\alpha_T-\alpha &=& 2TQ_T\Tr[b-\hat b_T]+2Q_T N_T \\\\

\mathcal{L}_{R_T}(\hat b_T-b)&=&(\alpha-\hat \alpha_T)TI_d+ M_T= 2T^2Q_T\Tr[\hat b_T-b]I_d + M_T-2TQ_TN_TI_d.
\end{array}
\right.
\end{equation}

\begin{mytheo}\label{thm_MLE_global_ergo}
Assume that $-b\in S^{+,*}_d$ and  $\alpha > d+1$. Under $\Px_\theta$, $\left(\sqrt{T}( \hat b_T-b, \hat \alpha_T-\alpha)\right)$ converges in law when $T\rightarrow +\infty$ to the centered Gaussian vector $(\mathbf G, H)$ that takes values in $\symm \times \R$ and has the following Laplace transform: for $c,\lambda \in \symm \times \R$,
$$ \E_\theta \left[ \exp\left( \Tr[c \mathbf G] +\lambda H \right) \right]= \exp\left(  \frac{2 \qbi  \lambda^2}{1-\qbi \Tr[\rbi^{-1}]} -    \frac{2 \qbi \lambda}{1-\qbi \Tr[\rbi^{-1}]} \Tr[ c \rbi^{-1}]  + \Tr[  c  \mathcal{L}^{-1}_{\rbi ,\qbi}(c)]  \right).$$

\end{mytheo}

\begin{proof} 
By~\eqref{ineg_Q_R} and Lemma~\ref{lemme-linear}, we can rewrite the system~\eqref{erreur_mle} as follows
\[
\left\{\begin{array}{r c l}
\sqrt{T}(\hat\alpha_T-\alpha) &=&
2TQ_T\dfrac{N_T}{\sqrt{T}}- 2TQ_T\Tr\left[\mathcal{L}^{-1}_{\frac{R_T}{T},TQ_T}\left(\dfrac{M_T}{\sqrt{T}}-2TQ_TI_d\dfrac{N_T}{\sqrt{T}} \right) \right]\\\\
\sqrt{T}(\hat b_T-b)&=&\mathcal{L}^{-1}_{\frac{R_T}{T},TQ_T}\left(\dfrac{M_T}{\sqrt{T}}-2TQ_TI_d\dfrac{N_T}{\sqrt{T}} \right).
\end{array}
\right.
\]
Note that,  for $i,j,k,l\in\{1,\dots,d\}$ we have
\begin{eqnarray}\label{brackets} 
\langle M_{i,j},
M_{k,l}\rangle_t &=&\left[\delta_{jl}(R_t)_{i,k}+\delta_{jk}(R_t)_{i,l}+
\delta_{il}(R_t)_{j,k}+\delta_{ik}(R_t)_{j,l}\right], \nonumber\\
\langle M_{i,j}, N\rangle_t &=&2t\delta_{ij}\quad \mbox{ and }\quad \langle
 N\rangle_t=Q_t^{-1}, \label{crochetsMN}
\end{eqnarray}
where $\delta_{ij}$ stands for the Kronecker symbol.

So, it follows from the central limit theorem for martingales (see e.g., Kutoyants~\cite{KutoyantsBook}, Proposition 1.21), that  $(\frac{M_T}{\sqrt{T}},  \frac{N_T}{\sqrt{T}})$  converges in law under $\Px_\theta$  towards a centered Gaussian vector $(\tilde{\mathbf G}, \tilde{H})$ taking values in $\symm\times \R$ such that
\begin{align}
\E_\theta(\tilde{\mathbf{G}}_{i,j}\tilde{\mathbf{G}}_{k,l})&=\left[\delta_{jl}(\rbi)_{i,k}+\delta_{jk}(\rbi)_{i,l}+
\delta_{il}(\rbi)_{j,k}+\delta_{ik}(\rbi)_{j,l}\right], \label{loi_Gtilde}\\
\E_\theta(\tilde{\mathbf{G}}_{i,j}\tilde{H})&=2 \delta_{i,j}  \text{ and } \E_\theta(\tilde{H}^2)=  \qbi^{-1}. \nonumber
\end{align} 
From~\eqref{eq_Z} and~\eqref{Birk2}, we obtain~\eqref{val_qbi}.
From Lemma \ref{lemme-linear}, the function $(X,Y,a)\mapsto \mathcal{L}^{-1}_{X,a}(Y)$ is continuous, and we get by Slutsky's theorem that $(\sqrt{T}(\hat b_T-b),\sqrt{T}(\hat\alpha_T-\alpha)) $ converges in law to the Gaussian vector
$$(\mathbf G, H)=\left( \mathcal{L}^{-1}_{\rbi ,\qbi}\left(\tilde{\mathbf G}-2 \qbi \tilde{H} I_d \right), 2 \qbi \left( \tilde{H}-\Tr\left[ \mathcal{L}^{-1}_{\rbi ,\qbi}\left(\tilde{\mathbf G}-2 \qbi \tilde{H} I_d \right) \right] \right)\right). $$
We are interested to calculate the Laplace transform of this law. First, we calculate the Laplace transform of $(\tilde{\mathbf G},\tilde{H})$:
\begin{equation}\label{lap_tgth}
\forall c \in \symm, \lambda \in \R, \E_\theta \left[ \exp\left( \Tr[c \tilde{\mathbf G} ] +\lambda \tilde{H} \right) \right]=\exp\left(\frac{1}{2}\left( \lambda^2 \qbi^{-1} + 4 \lambda \Tr[c] + 4 \Tr[c^2 \rbi] \right) \right).
\end{equation}
We want to calculate for $c\in \symm$ and $\lambda \in \R$, 
\begin{align*}
 \E_\theta \left[ \exp\left( \Tr[c \mathbf G] +\lambda H \right) \right] &=\E_\theta \left[ \exp\left( \Tr[(c-2 \lambda  \qbi I_d) {\mathbf G}] + 2 \lambda \qbi \tilde{H} \right) \right].
\end{align*}
Due to~\eqref{ineg_qbi} and Lemma~\ref{lemme-linear}, we can introduce $\tilde{c}=\mathcal{L}^{-1}_{\rbi ,\qbi}(c-2 \lambda  \qbi I_d)$. We have
$$ \rbi \tilde{c}+\tilde{c} \rbi- 2 \qbi \Tr[\tilde{c}]I_d=c-2\lambda \qbi I_d,$$
and thus
\begin{align*}
\Tr[(c-2 \lambda  \qbi I_d) \mathbf G]&=\Tr[(\rbi \tilde{c}+\tilde{c} \rbi- 2 \qbi \Tr[\tilde{c}]I_d)  \mathbf G]\\
&=\Tr[\tilde{c} (\rbi \mathbf G+\mathbf G \rbi- 2 \qbi \Tr[\mathbf G]I_d)  ]=\Tr[\tilde{c} (\tilde{\mathbf G}-2 \qbi \tilde{H} I_d )].
\end{align*}
We therefore obtain from~\eqref{lap_tgth}
\begin{align*}
 &\E_\theta \left[ \exp\left( \Tr[c \mathbf G] +\lambda H \right) \right] \\
&=\E_\theta \left[ \exp\left( \Tr[ \tilde{c} (\tilde{\mathbf G}-2 \qbi \tilde{H} I_d )] + 2 \lambda \qbi \tilde{H} \right) \right] \\
&=\E_\theta \left[ \exp\left( \Tr[ \tilde{c} \tilde{\mathbf G}] + 2 \qbi (\lambda-\Tr[\tilde{c}]) \tilde{H} \right) \right] \\
&=\exp\left( 2 \left\{ (\lambda-\Tr[\tilde{c}])^2 \qbi + 2(\lambda-\Tr[\tilde{c}])\Tr[\tilde{c}]\qbi + \Tr[\tilde{c}^2 \rbi] \right\} \right).
\end{align*}
Since $2\Tr[\tilde{c}^2 \rbi]=\Tr[\tilde{c}(\tilde{c} \rbi+ \rbi \tilde{c})]=\Tr[\tilde{c}c]+2\qbi (\Tr[\tilde{c}]-\lambda) \Tr[\tilde{c}] $, we get
$$\E_\theta \left[ \exp\left( \Tr[c \mathbf G] +\lambda H \right) \right]=\exp\left( 2   \lambda  (\lambda-\Tr[\tilde{c}]) \qbi + \Tr[\tilde{c}c  ]  \right).$$
We now use that $\mathcal{L}^{-1}_{\rbi ,\qbi}(I_d)=\frac{1}{2(1-\qbi \Tr[\rbi^{-1}])}\rbi^{-1}$ to get
$\tilde{c}=\mathcal{L}^{-1}_{\rbi ,\qbi}(c)-  \lambda \frac{\qbi  \rbi^{-1}}{1-\qbi \Tr[\rbi^{-1}]}$. Since we have $\Tr[\mathcal{L}^{-1}_{\rbi ,\qbi}(c)]= \frac{\Tr[\rbi^{-1}c]}{2(1-\qbi \Tr[\rbi^{-1}])} $ by Lemma~\ref{lemme-linear}, this yields to the claimed result.
\end{proof}

When $\alpha=d+1$, the rate of convergence of the MLE of $\alpha$ is even better as stated by the following theorem. 
\begin{mytheo}\label{alphad}
 Assume $-b\in S^{+,*}_d $ and $\alpha=d+1$. Then, under $\Px_\theta$, $\left(\sqrt{T}(\hat
b_T-b),T(\hat \alpha_T-\alpha)\right)$ converges in law when $T\rightarrow + \infty$
to  $\left(\mathbf G    , -2\tau^{-1}_{_{-\Tr[b]}}\Tr [b]  \right)$, where $\tau_a=\inf\{t\geq 0,\;B_t=a\}$ with $(B_t)_{t\geq 0}$ a given one-dimensional standard Brownian motion and $\mathbf G$ is a Gaussian vector independent of~$B$ such that $\E_\theta \left[ \exp\left( \Tr[c \mathbf G]  \right) \right]= \exp\left(   \Tr[  c  \mathcal{L}^{-1}_{\rbi }(c)]  \right)$, $c\in \symm$.
\end{mytheo}

\begin{proof}
By~\eqref{ineg_Q_R} and Lemma~\ref{lemme-linear}, we can rewrite the system~\eqref{erreur_mle} as follows
\begin{equation}\label{mle_d+1}
\left\{\begin{array}{r c l}
T(\hat\alpha_T-\alpha) &=&
2T^2Q_T\left(\dfrac{N_T}{T}-\frac{1}{\sqrt{T}}
\Tr\left[\mathcal{L}^{-1}_{\frac{R_T}{T},TQ_T}\left(\dfrac{M_T}{\sqrt{T}}-2T^{3/2}Q_TI_d\dfrac{N_T}{T} \right)
\right]\right)\\\\
\sqrt{T}(\hat
b_T-b)&=&\mathcal{L}^{-1}_{\frac{R_T}{T},TQ_T}\left(\dfrac{M_T}{\sqrt{T}}-2T^{3/2}Q_TI_d\dfrac{N_T}{T}
\right).
\end{array}
\right.
\end{equation}
From~\eqref{eq_Z}, we have 
 $$\frac{N_T}{T}=\frac{1}{2T}\log\left(\dfrac{\det[X_T]}{\det[x]}\right)-\Tr[b].$$
As for $-b\in S^{+,*}_d$ the Wishart process $(X_t)_{t\geq 0}$ is stationary with invariant
limit distribution $X_{\infty}$ we easily deduce that $\frac{N_T}{T}$ converges in probability to $-\Tr[b]$ when $T\rightarrow\infty$.
Then, it follows from~\eqref{Birk1} that
\begin{equation}
\label{couple}
(T^{-1}R_T, T^{-1}N_T)\overset{\mathbb P_\theta}\rightarrow
(\rbi,-\Tr[b]),\quad{ as }\;T\rightarrow\infty.
\end{equation}
Hence, we only need to study the asymptotic behavior of the couple $(T^{-1/2}M_T,T^2Q_T)$.
According to Theorem~4.1 in Mayerhofer~\cite{MayerhoferHammamet}, we have  for $ \lambda\geb 0$ and $\Gamma\in
S_d$ 
\begin{equation}
\label{RM}
\mathbb
E_\theta \left[\exp\left(\frac{\lambda}{T}N_T-\frac{\lambda^2}{2T^2}Q_T^{-1}+\frac{1}{\sqrt{T}}\Tr[\Gamma M_T]-\frac{2}{T}\int_0^T\Tr[\Gamma^2
X_s]ds-\frac{2\lambda}{\sqrt{T}}\Tr[\Gamma]\right)\right]=1.
\end{equation}
Now, let us introduce the quantity
\begin{multline*}
A_T=\E_\theta\left[\exp\left(\lambda\frac{N_T}{T}+\lambda\Tr[b]\right)\exp\left(-\frac{\lambda^2}{2T^2}Q_T^{-1}+\frac{1}{\sqrt{T}}\Tr[\Gamma M_T]\right)\right.\\ \times \left.\exp\left(-\frac{2}{T}\int_0^T\Tr[\Gamma^2X_s]ds+2\Tr[\Gamma^2 \rbi]\right)\right].
\end{multline*}
Then, by \eqref{RM} we easily get
$A_T=\exp\left(\lambda\Tr[b]+2\Tr[\Gamma^2 \rbi]+\frac{2\lambda}{\sqrt{T}}\Tr[\Gamma]\right). $
We now write $A_T=\tilde{A}_T+\E_\theta\left[\exp\left(-\frac{\lambda^2}{2T^2}Q_T^{-1}+\frac{1}{\sqrt{T}}\Tr[\Gamma
M_T]\right)\right]$ with
\begin{align*}
\tilde{A}_{T}&=\E_\theta \left[ ( \exp\left(\xi_T \right)-1 )  \exp \left(-\frac{\lambda^2}{2T^2}Q_T^{-1}+\frac{1}{\sqrt{T}}\Tr[\Gamma M_T]\right)\right] \\
\xi_T&=\lambda\frac{N_T}{T}+\lambda\Tr[b] -\frac{2}{T}\int_0^T\Tr[\Gamma^2X_s]ds+2\Tr[\Gamma^2 \rbi ] .
\end{align*}
Cauchy-Schwarz inequality and $Q_T^{-1}>0$ give
$$|\tilde{A}_{T}|\leb \E_\theta^{1/2}[  \exp\left(2\xi_T \right)-2 \exp\left(\xi_T \right)+1  ] \E_\theta^{1/2}\left[\exp\left(\frac{2}{\sqrt{T}}\Tr[\Gamma
M_T]\right)\right]. $$
On the one hand,   Proposition~\ref{prop:lap_Wish} with $m=-b\in \dpos$ gives $$\E_\theta\left[\exp\left(\frac{2}{\sqrt{T}}\Tr[\Gamma M_T]\right)\right] \leb  \E_\theta\left[\exp\left( \frac{2}{T} \Tr[\Gamma^2 R_T]\right)\right]< \infty. $$   On the other hand, we have for any $r\geb 0$,
$$ \E_\theta [  \exp\left(r \xi_T \right)] \leb \E_\theta \left[  \exp\left( \frac{\lambda r }{T}  N_T\right)\right]  \exp(2r\Tr[\Gamma^2\rbi]).$$
From~\eqref{eq_Z}, we have  
$$\E_\theta \left[  \exp\left(  \frac{\lambda r }{T} N_T \right) \right]=\exp(- \lambda r  \Tr[b]) \E_\theta \left[\left(\frac{\det[X_T]}{\det[x]}\right)^{ \frac{\lambda r }{2T}} \right].$$ 
The sublinear growth of the coefficients of the Wishart SDE  and the convergence to a stationary law gives that $\E_\theta \left[\left(\frac{\det[X_T]}{\det[x]}\right)^{\tilde{\lambda}} \right]$ is uniformly bounded in $T>0, \ \tilde{\lambda}<1$ and therefore $\sup_{T>\frac{\lambda r }{2}}\E_\theta \left[\left(\frac{\det[X_T]}{\det[x]}\right)^{ \frac{\lambda r }{2T}} \right]<\infty$.
This gives the uniform integrability of the family $(\exp\left(2\xi_T \right),T>\lambda)$. Then, we deduce from~\eqref{couple} that $\E_\theta[  \exp\left(2\xi_T \right)-2 \exp\left(\xi_T \right)+1  ]\underset{T \rightarrow  + \infty }{\rightarrow } 0$ and thus $ \tilde{A}_{T}\underset{T \rightarrow  + \infty }{\rightarrow } 0$.

 Hence, we obtain
$$\lim_{T\rightarrow\infty}\E_\theta\left[\exp\left(-\frac{\lambda^2}{2T^2}Q_T^{-1}+\frac{1}{\sqrt{T}}\Tr[\Gamma
M_T]\right)\right] =\lim_{T\rightarrow\infty} A_T=\exp\left(\lambda\Tr[b]+2\Tr[\Gamma^2\rbi ]\right). $$
Therefore, we deduce by Lemma~\ref{lemme_matgauss} the following convergence in law
$$ \left(\dfrac { Q_T^{-1}}{T^2},\dfrac{M_T}{\sqrt{T}}\right) \Rightarrow \left(\tau_{_{-\Tr[b]}}, { \sqrt{\rbi} \tilde{\mathbf G}+ \tilde{\mathbf G}\tp  \sqrt{\rbi} }\right) \text{ as } T\rightarrow\infty,$$ where $\tilde{G}_{i,j}$ $1\leb i,j\leb d$ are independent standard normal variables.
Together with~\eqref{couple}, we obtain that \begin{equation}\label{cv_loi_3}
( T^{-1}R_T,T^2Q_T, T^{-1}N_T, T^{-1/2}M_T) \Rightarrow (\rbi,1/\tau_{_{-\Tr[b]}},-\Tr[b], { \sqrt{\rbi} \tilde{\mathbf G}+ \tilde{\mathbf G}\tp  \sqrt{\rbi} }),
\end{equation}
which gives the claim by~\eqref{mle_d+1} and  Lemma~\ref{lemme_matgauss}.
\end{proof}

\subsection{The global MLE estimator of $\theta=( b,\alpha)$ when $b\in\genm$}

We define the linear operators $\bar{\cL}_X,\bar{\cL}_{X,a}:\genm\rightarrow \genm$ by 
$$\bar{\cL}_X(Y)=\cL_X^{-1}(YX+XY\tp)X, \ \bar{\cL}_{X,a}(Y)=\bar{\cL}_X(Y)-a \Tr[Y]I_d.$$ From~\eqref{SDE_Intro2}, we get $Z_T=(\alpha-1-d) Q_T^{-1}+2\Tr[b]T +2 N_T$. This yields with~\eqref{pt_crit_emvgen} to
\begin{equation}\label{erreur_mle_gen}
\left\{
\begin{array}{r c l}
\hat\alpha_T-\alpha &=& 2TQ_T\Tr[b-\hat b_T]+2Q_T N_T \\\\
\int_0^T\bar{\cL}_{X_s}(\hat b_T-b) ds -T^2Q_T \Tr[\hat b_T-b] I_d &=&\int_0^T\cL_{X_s}^{-1}(dM_s)X_s-TQ_TN_TI_d.
\end{array}
\right.
\end{equation}
We now define $$\hat{\cL}_T(Y)=\frac{1}{T}\int_0^T\bar{\cL}_{X_s}(Y) ds -TQ_T \Tr[Y] I_d,$$ which is a linear operator on~$\genm$. By using the convexity of the inverse function, there exists $\gamma \in(0,1)$ that depends on $(X_s,s\in[0,T])$ such that $TQ_T=\frac{\gamma}{T} \int_0^T \frac{1}{\Tr[X_s^{-1}]}ds$. We get  $\hat{\cL}_T(Y)= \frac{1}{T}\int_0^T\bar{\cL}_{X_s,\frac{\gamma}{\Tr[X_s^{-1}]}}(Y)ds$. By Lemma~\ref{lemme-autoadjoint2},  $\hat{\cL}_T$ is self adjoint and positive. It is even positive definite since $\Tr[\hat{\cL}_T(Y)\tp Y]=0$ implies by using Lemmas~\ref{lemme-autoadjoint2} that $YX_s+X_sY\tp=0$ a.s. on $[0,T]$ under $\Px_\theta$, and therefore the quadratic variation of $\Tr[uYX_s]$ is equal to zero for any $u\in \symm$. This gives $\int_0^T\Tr[(uY+Y\tp u)X_s(uY+Y\tp u)]ds=0$ and thus $uY+ Y\tp u=0$ for all $u\in \symm$, which necessarily implies $Y=0$.  Thus, we rewrite~\eqref{erreur_mle_gen} as
\begin{equation}\label{erreur_mle_gen_2}
\left\{
\begin{array}{r c l}
\sqrt{T}(\hat\alpha_T-\alpha) &=& -2TQ_T\Tr[\sqrt{T}(\hat b_T-b)]+2TQ_T \frac{N_T}{\sqrt{T}}   \\
\sqrt{T}(\hat b_T-b)  &=&\hat{\cL}_T^{-1}\left( \frac{1}{\sqrt{T}}\int_0^T \cL_{X_s}^{-1}(dM_s)X_s  -TQ_T \frac{N_T}{\sqrt{T}}I_d \right).
\end{array}
\right.
\end{equation}
We will assume $-(b+b\tp)\in \dpos$ and know from  Lemma~\ref{lem_ergo} that $X_T$ converges in law under $\Px_\theta$ to the stationary law $X_\infty\sim WIS_d(0,\alpha,0,\sqrt{2q_\infty};1/2)$. We define \begin{equation}\label{ellhat}\hat{\cL}_\infty(Y)=\E_\theta[\bar{\cL}_{X_{\infty}}(Y)]-\qbi \Tr[Y]I_d.
\end{equation}
Note that for $Y\in\symm$, $\hat{\cL}_\infty(Y)=Y\rbi-\qbi\Tr[Y]I_d$.
From the convexity of the inverse function, $\qbi=\gamma \E_\theta[1/\Tr[X_\infty^{-1}]]$ with $\gamma\in(0,1)$, and thus $\hat{\cL}_\infty(Y)=\E_\theta[\bar{\cL}_{X_{\infty},\frac{\gamma}{\Tr[X_\infty^{-1}]}}(Y)]$ is a self-adjoint positive operator by Lemma~\ref{lemme-autoadjoint2}. It is even positive definite since $\Tr[\hat{\cL}_\infty(Y)\tp Y]=0$ implies by Lemma~\ref{lemme-autoadjoint2} that $YX_\infty+X_\infty Y\tp=0$ almost surely. Since the law $X_\infty$ has a positive density on $\dpos$, this gives $Yu+u Y\tp=0$ for any $u\in \symm$ and thus $Y=0$.

\begin{mytheo}\label{theo_b_gen}
 Assume $-(b+b\tp)\in S^{+,*}_d $ and $\alpha>d+1$. Then, under $\Px_\theta$, $\left(\sqrt{T}(\hat
b_T-b),T(\hat \alpha_T-\alpha)\right)$ converges in law when $T\rightarrow + \infty$
to the centered Gaussian vector $(\mathbf G, H)$ that takes values in $\genm \times \R$ and has the following Laplace transform: for $c,\lambda \in \genm \times \R$,
\begin{align}
\E_\theta[\exp(\Tr[c\tp {\mathbf{G}}]+\lambda {H})] \label{laplace_lim_bgen}
&=\exp\bigg(\frac{2 \qbi  \lambda^2}{1-\qbi \Tr[\rbi^{-1}]} -     \frac{2 \qbi \lambda}{1-\qbi \Tr[\rbi^{-1}]} \Tr[ c \rbi^{-1}]  \\
&+\frac{1}{4}\E_\theta[\Tr[\cL_{X_\infty }^{-1}( \hat{\cL}_\infty^{-1}(c)X_\infty + X_\infty  \hat{\cL}_\infty^{-1}(c)\tp ) ( \hat{\cL}_\infty^{-1}(c)X_\infty + X_\infty  \hat{\cL}_\infty^{-1}(c)\tp ) ]] \nonumber \\
&+\frac{\Tr[c \rbi^{-1}]\qbi}{1-\qbi\Tr[\rbi^{-1}]}\left( \frac{1}{2} \frac{\Tr[c \rbi^{-1}]}{1-\qbi\Tr[\rbi^{-1}]}-\Tr[\hat{\cL}_\infty^{-1}(c)]\right)   \bigg) \nonumber
\end{align}

\end{mytheo}
\begin{proof}

 From the ergodic Birkhoff's theorem,  $\hat{\cL}_T$ converges almost surely to $\hat{\cL}_\infty$, and thus  $\hat{\cL}_T^{-1}$ converges almost surely to $\hat{\cL}_\infty^{-1}$. We define the martingale $\hat{M}_T=\int_0^T\cL_{X_s}^{-1}(dM_s)X_s$. We have
\begin{align*}
d (\hat{M}_s)_{i,j} &=\Tr[e_{i,j}\tp\cL_{X_s}^{-1}(dM_s)X_s  ]=\frac{1}{2}\Tr[\cL_{X_s}^{-1}( e_{i,j}X_s+ X_s e_{i,j}\tp )dM_s] \\&=\frac{1}{2}\sum_{1\le k,l\le d} (\cL_{X_s}^{-1}( e_{i,j}X_s+ X_s e_{i,j}\tp ))_{k,l} (dM_s)_{k,l}.
\end{align*}
We get from~\eqref{brackets} that
\begin{align*}
\langle d (\hat{M}_s)_{i,j},d (\hat{M}_s)_{i',j'} \rangle &=\frac{1}{2} \Tr\left[\cL_{X_s}^{-1}( e_{i,j}X_s+ X_s e_{i,j}\tp ) ( e_{i',j'}X_s+ X_s e_{i',j'}\tp ) \right]ds,\\
\langle d (\hat{M}_s)_{i,j},d N_s \rangle &= \Tr[\cL_{X_s}^{-1}( e_{i,j}X_s+ X_s e_{i,j}\tp )]ds =\delta_{i,j}ds.
\end{align*}
From the central limit theorem for martingales,    $(\frac{\hat{M}_T}{\sqrt{T}},  \frac{N_T}{\sqrt{T}})$  converges in law under $\Px_\theta$  towards a centered Gaussian vector $(\hat{\mathbf G}, \hat{H})$ taking values in $\genm\times \R$ such that
\begin{align}
&\E_\theta(\hat{\mathbf{G}}_{i,j}\hat{\mathbf{G}}_{i',j'})=\frac{1}{2}
\E_\theta\left[  \Tr\left[ \cL_{X_\infty }^{-1}( e_{i,j}X_\infty + X_\infty  e_{i,j}\tp ) ( e_{i',j'}X_\infty + X_\infty  e_{i',j'}\tp ) \right] \right] \label{loi_Ghat}\\
&\E_\theta(\hat{\mathbf{G}}_{i,j}\hat{H})=\delta_{i,j}  \text{ and } \E_\theta(\hat{H}^2)=  \qbi^{-1}. \nonumber
\end{align}
We thus have the following Laplace transform for $c\in \genm$, $\lambda \in \R$,
\begin{align}\label{lapl_intermed_bgen} \E_\theta[\exp(\Tr[c\tp\hat{\mathbf{G}}]+\lambda \hat{H})]=\exp\left(\frac{1}{4}\E_\theta[\Tr[\cL_{X_\infty }^{-1}( cX_\infty + X_\infty  c\tp ) ( cX_\infty + X_\infty  c\tp ) ]] +\lambda \Tr[ c ] +\frac{1}{2}\lambda^2\qbi^{-1} \right).
\end{align}
By using Slutsky's Theorem, we get from~\eqref{erreur_mle_gen_2} that  $\left(\sqrt{T}( \hat b_T-b, \hat \alpha_T-\alpha)\right)$ converges in law under $\Px_\theta$  when $T\rightarrow +\infty$ to the centered Gaussian vector 
$$(\mathbf G, H)=\left(\hat{\cL}_\infty^{-1}(\hat{\mathbf G}- \qbi \hat{H} I_d ) , 2 \qbi \left(\hat{H}-\Tr[\hat{\cL}_\infty^{-1}(\hat{\mathbf G}- \qbi \hat{H} I_d )] \right) \right) .$$ 
Now, we use that $\hat{\cL}_\infty^{-1}(I_d)=\frac{1}{1-\qbi\Tr[\rbi^{-1}]}\rbi^{-1}$ and that $\hat{\cL}_\infty$ is self-adjoint to get for $c\in \genm$, 
$$\Tr[c \tp {\mathbf G}]=\Tr[\hat{\cL}_\infty^{-1}(c) \tp \hat{\mathbf G}]-\frac{\qbi\hat{H}\Tr[c\tp \rbi^{-1}]}{1-\qbi\Tr[\rbi^{-1}]}, \ H=\frac{2\qbi }{1-\qbi\Tr[\rbi^{-1}]} (\hat{H}-\Tr[\rbi^{-1} \hat{\mathbf G}]) ].$$
From~\eqref{lapl_intermed_bgen}, we obtain after some calculations~\eqref{laplace_lim_bgen}, using in particular that for $m\in \genm,s\in \symm$,
\begin{align}
&\E_\theta[\Tr[\cL_{X_\infty }^{-1}( (m+s)X_\infty + X_\infty  (m+s)\tp ) ( (m+s)X_\infty + X_\infty  (m+s)\tp )]] \label{calcul_trace_intermed_gen}\\
&=\E_\theta[\Tr[\cL_{X_\infty }^{-1}( mX_\infty + X_\infty  m\tp ) ( mX_\infty + X_\infty  m\tp )]]+2\Tr[s(m \rbi+ \rbi m\tp)]+2\Tr[s^2\rbi]\nonumber
\end{align}
and taking $m=\hat{\cL}_\infty^{-1}(c)$ and $s=-\lambda \frac{2 \qbi}{1-\qbi\Tr[\rbi^{-1}]} \rbi^{-1}$. 
\end{proof}

\begin{rmrk} It is interesting to compare Theorems~\ref{thm_MLE_global_ergo} and~\ref{theo_b_gen} and see that the asymptotic variance of $\sqrt{T}(\hat{\alpha}_T-\alpha)$ is the same in both cases. Instead, for the estimation of the symmetric part of $b$, we can check that the asymptotic variance is greater when we do not know a priori that $b$ is symmetric. For $c \in \symm$, we have  $c=\frac{1}{2}[\hat{\cL}_\infty^{-1}(c)\rbi +\rbi \hat{\cL}_\infty^{-1}(c) \tp ]-\qbi \Tr[\hat{\cL}_\infty^{-1}(c)]I_d$ and   $c=\mathcal{L}^{-1}_{\rbi ,\qbi}(c)\rbi +\rbi\mathcal{L}^{-1}_{\rbi ,\qbi}(c)-2\qbi\Tr[\mathcal{L}^{-1}_{\rbi ,\qbi}(c)]I_d.$ Multiplying by $\rbi^{-1}$, we get 
$$\Tr[\hat{\cL}_\infty^{-1}(c)]=\frac{\Tr[c \rbi^{-1}]}{1-\qbi\Tr[\rbi^{-1}]}=2\Tr[\mathcal{L}^{-1}_{\rbi ,\qbi}(c)],$$
and then $\hat{\cL}_\infty^{-1}(c)=2\mathcal{L}^{-1}_{\rbi ,\qbi}(c)+\Delta$ with $\Delta \rbi+\rbi \Delta\tp=0$. This gives from~\eqref{calcul_trace_intermed_gen}
\begin{align*}
&\frac{1}{4}\E_\theta[\Tr[\cL_{X_\infty }^{-1}( \hat{\cL}_\infty^{-1}(c)X_\infty + X_\infty  \hat{\cL}_\infty^{-1}(c)\tp ) ( \hat{\cL}_\infty^{-1}(c)X_\infty + X_\infty  \hat{\cL}_\infty^{-1}(c)\tp ) ]] \\
&+\frac{\Tr[c \rbi^{-1}]\qbi}{1-\qbi\Tr[\rbi^{-1}]}\left( \frac{1}{2} \frac{\Tr[c \rbi^{-1}]}{1-\qbi\Tr[\rbi^{-1}]}-\Tr[\hat{\cL}_\infty^{-1}(c)]\right)\\
=& 2\Tr[\mathcal{L}^{-1}_{\rbi ,\qbi}(c)^2\rbi] + \frac{1}{4}\E_\theta[\Tr[\cL_{X_\infty }^{-1}( \Delta X_\infty + X_\infty  \Delta \tp ) ( \Delta X_\infty + X_\infty  \Delta\tp ) ]] -2\qbi \Tr[\mathcal{L}^{-1}_{\rbi ,\qbi}(c)]^2\\
=& \Tr[\mathcal{L}^{-1}_{\rbi ,\qbi}(c) c] +\frac{1}{4}\E_\theta[\Tr[\cL_{X_\infty }^{-1}( \Delta X_\infty + X_\infty  \Delta \tp ) ( \Delta X_\infty + X_\infty  \Delta\tp ) ]] \ge \Tr[\mathcal{L}^{-1}_{\rbi ,\qbi}(c) c] ,
\end{align*}
since $\cL_{X_\infty }^{-1}$ is a self-adjoint positive operator. 
\end{rmrk}

\subsection{The MLE estimator of $b$}

When $\alpha \in [d-1,d+1)$, we are no longer able to study the convergence of the MLE of $\alpha$. It is however still possible to get the speed of convergence of the MLE of~$b$.

\begin{mytheo}\label{thm_bseul_ergo}
Assume that $-b\in S^{+,*}_d$ and  $\alpha \geb d-1$. For $T>0$,  we consider $\hat{b}_T$ defined by (\ref{MLE_bseul}). Then, under $\Px_\theta$, $\sqrt{T}(\hat{b}_T-b)$ converges in law to a centered Gaussian vector $\mathbf G$ on $\symm$ with the following Laplace transform $\E_\theta \left[ \exp\left( \Tr[c \mathbf G]  \right) \right]= \exp\left(   \Tr[  c  \mathcal{L}^{-1}_{\rbi }(c)]  \right)$, $c\in \symm$.

Assume that $-(b+b\tp)\in S^{+,*}_d$ and  $\alpha \geb d-1$. For $T>0$,  we consider $\hat{b}_T$ defined by (\ref{MLE_bseul_gen}). Then, under $\Px_\theta$, $\sqrt{T}(\hat{b}_T-b)$ converges in law to a centered Gaussian vector $\mathbf G$ on $\genm$ with the following Laplace transform: for  $c\in \genm$, $\E_\theta \left[ \exp\left( \Tr[c\tp \mathbf G]  \right) \right]$
$$= \exp\left(  \frac{1}{4}\E_\theta[\Tr[\cL_{X_\infty }^{-1}( \check{\cL}_\infty^{-1}(c)X_\infty + X_\infty  \check{\cL}_\infty^{-1}(c)\tp ) ( \check{\cL}_\infty^{-1}(c)X_\infty + X_\infty  \check{\cL}_\infty^{-1}(c)\tp ) ]]  \right),$$
with $\check{\cL}_\infty(c)=\E_\theta[\bar{\cL}_{X_{\infty}}(Y)]$.
\end{mytheo} 
\begin{proof}
We could prove the result for (\ref{MLE_bseul}) by using the explicit Laplace transform Proposition~\ref{prop:lap_Wish}. Here, we use the same arguments as before based on the ergodic property. From~\eqref{MLE_bseul}, we have 
$$\sqrt{T}(\hat{b}_T-b)= \mathcal{L}_{\frac{R_T}{T}}^{-1}\left(\frac{1}{\sqrt{T}}\left(X_T-x-bR_T-R_Tb -\alpha T I_d \right) \right)=\mathcal{L}_{\frac{R_T}{T}}^{-1}\left(\frac{M_T}{\sqrt{T}}\right).$$
 As in the proof of Theorem~\ref{thm_MLE_global_ergo},  $\frac{M_T}{\sqrt{T}}$ converges in law to the centered Gaussian vector $\tilde{\mathbf{G}}$ defined by~\eqref{loi_Gtilde}. Slutsky's theorem and~\eqref{Birk1} give then the convergence of $\sqrt{T}(\hat{b}_T-b)$ to $\mathbf G=\mathcal{L}_{\rbi}^{-1}\left(\tilde{\mathbf{G}}\right)$, whose Laplace transform is given by Lemma~\ref{lemme_matgauss}.

Similarly, we get from~\eqref{MLE_bseul_gen} that $\sqrt{T}(\hat{b}_T-b)= \check{\cL}_T^{-1}(\hat{M}_T/\sqrt{T})$ with $\check{\cL}_T(Y)=\frac{1}{T}\int_0^T\bar{\cL}_{X_s}(Y) ds$. We easily check that $\Tr[\check{\cL}_T(Y)Y]\ge\Tr[\hat{\cL}_T(Y)Y]$ and $\Tr[\check{\cL}_\infty(Y)Y]\ge\Tr[\hat{\cL}_\infty(Y)Y]$ for $Y\in \genm$. Therefore, $\check{\cL}_T$ and $\check{\cL}_\infty$ are positive definite and thus invertible. Besides, the ergodic theorem gives that $\check{\cL}_T^{-1}$ converges almost surely to $\check{\cL}_\infty$. The result follows from~\eqref{lapl_intermed_bgen} and the self-adjoint property of $\check{\cL}_\infty$.
\end{proof}

\section[Statistical Inference:  some nonergodic cases] {Statistical Inference of the Wishart process:  some nonergodic cases}\label{Sec_nonergo}

This section studies the convergence of the MLE in the case $b=b_0 I_d$ with $b_0\geb 0$. When $b_0=0$ and $\alpha \geb d+1$, we are able to describe the rate of convergence of the MLE of $( b,\alpha)$ given by~\eqref{eq:estimator_couple}, when $b$ is known to be symmetric. When $b_0>0$ and $\alpha\geb d-1$, we can also obtain the rate of convergence of the MLE of $b$ given by~\eqref{MLE_bseul}. Last, when $b$ is known a priori to be diagonal,  the MLE of $b$ has a simpler form and we can describe precisely its convergence.

\subsection{The global MLE of $\theta=( b,\alpha)$ when $b=0$}
The following result provides the asymptotic behavior of the estimator of the couple when $\alpha >d+1$ and $b=0$ in (\ref{SDE_Intro2}).

\begin{mytheo}\label{b0}
Assume that $b=0$ and  $\alpha>d+1$. Let $(\hat{b}_T, \hat{\alpha}_T)$ be the MLE defined by~(\ref{eq:estimator_couple}). Then, $( T(\hat b_T -b ),\sqrt{\log(T)}(\hat \alpha_T-\alpha))$ converges in law under $\Px_\theta$  when $T\rightarrow +\infty$ to $$\left(\mathcal{L}^{-1}_{R_1^0}\left( X_1^0-\alpha I_d  \right) ,2 \sqrt{\frac{\alpha -(d+1)}{d}}G \right),$$ where 
$X^0_t=\alpha t I_d + \int_0^t \sqrt{X^0_s} dW_s +dW_s\tp \sqrt{X^0_s}$ is a Wishart process with the same parameters but starting from~$0$, $R_t^0 =\int_0^t X_s^0ds$ and $G\sim \mathcal{N}(0,1)$ is an independent standard Normal variable. 
\end{mytheo}

\begin{proof}
From~\eqref{eq:estimator_couple} and~\eqref{erreur_mle}, we obtain
\[
\left\{\begin{array}{r c l}
\sqrt{\log(T)}(\hat\alpha_T-\alpha) &=&
-2 \frac{T \Tr[\hat b_T]}{\sqrt{\log(T)}} \log(T)Q_T  +2 \log(T)Q_T \frac{N_T}{\sqrt{\log(T)}}\\\\
T\hat b_T &=&\mathcal{L}^{-1}_{\frac{R_T}{T^2},Q_T}\left(\frac{X_T}{T}- \frac{x}{T}-(Q_TZ_T+1+d)I_d \right),
\end{array}
\right.
\]
and we are interested in studying the convergence in law of $\left(\frac{N_T}{\sqrt{\log(T)}},\frac{X_T}{T},\frac{R_T}{T^2} \right).$ By Theorem~4.1 in~\cite{MayerhoferHammamet}, for $\mu \geb 0$ and $T>1$, 
$$\frac{d \overline{\Px}}{d \Px_\theta}= \exp \left( \frac{\mu N_T}{\sqrt{\log(T)}}- \frac{\mu^2}{2 Q_T \log(T)} \right)$$
defines a change of probability and $(X_t)_{t\in [0,T]}$ is a Wishart process with degree $\alpha +\frac{\mu }{\sqrt{\log(T)}}$ under $\overline{\Px}$. Let $\lambda_1,\lambda_2 \in \dpos$ and
$$A_T=\E_\theta \left[ \exp \left( \frac{\mu N_T}{\sqrt{\log(T)}}- \frac{\mu^2}{2 Q_T \log(T)}  \right)\exp \left( - \Tr \left[\frac{\lambda_2}{T}X_T  \right] - \Tr \left[ \frac{\lambda_1}{T^2} R_T\right] \right) \right].$$
 By Proposition~\ref{prop:lap_Wish}, we have
\begin{align}
A_T
&=  
\overline{\E}\left[ \exp \left( - \Tr \left[\frac{\lambda_2}{T}X_T  \right] - \Tr \left[ \frac{\lambda_1}{T^2} R_T\right] \right)\right] \nonumber \\
&=
\frac{1}{ \det[V]^{ \frac{\alpha+\mu/\sqrt{\log(T)}}{2}}} \exp \Big(- \frac{1}{2T}\Tr \big[ V'V^{-1} x  \big] \Big)\underset{T \rightarrow +\infty}\rightarrow  \frac{1}{\det[V]^{ \frac{\alpha}{2}}}, \label{lim_A_T}
\end{align}
where 
\begin{equation}\label{defV_Vprime}
V=(\sqrt{2\lambda_1})^{-1}\sinh(\sqrt{2\lambda_1})2 \lambda_2 +  \cosh(\sqrt{2\lambda_1}), \ V'  =   2\cosh(\sqrt{2\lambda_1})\lambda_2 +  \sqrt{2\lambda_1}\sinh(\sqrt{2\lambda_1}).
\end{equation}
We note that this limit does not depend on $\mu$ and is the Laplace transform of $(X^0_1,R^0_1)$ by Proposition~\ref{prop:lap_Wish}.

We now use that $\frac{1}{ Q_T \log(T)}\underset{T \rightarrow +\infty}\rightarrow \frac{d}{\alpha -(d+1)} $ a.s., see Lemma~\ref{lem_yor} and we define
\begin{align*}
\tilde{A}_T&=\E_{\theta} \left[ \exp (\xi_T)\right] \text{ and } \xi_T= \frac{\mu N_T}{\sqrt{\log(T)}}- \frac{\mu^2 d}{2(\alpha -(d+1))} - \Tr \left[\frac{\lambda_2}{T}X_T  \right] - \Tr \left[ \frac{\lambda_1}{T^2} R_T \right]
\end{align*}
that is finite by using equation~\eqref{ui1} of Lemma~\ref{lem_yor} since $\xi_T\leb \frac{\mu N_T}{\sqrt{\log(T)}}$. We have
$$A_T=\tilde{A}_T+\E_{\theta} \left[ \left\{ \exp \left( \frac{\mu^2 d}{2(\alpha -(d+1))} - \frac{\mu^2}{2 Q_T \log(T)}\right)-1 \right\}\exp \left( \xi_T \right) \right].$$
 The Cauchy-Schwarz inequality 
gives $$(A_T-\tilde{A}_T)^2\leb \E_{\theta} \left[ \left\{ \exp \left( \frac{\mu^2 d}{2(\alpha -(d+1))} - \frac{\mu^2}{2 Q_T \log(T)}\right)-1 \right\}^2\right]\E_{\theta} \left[\exp \left( 2\xi_T \right) \right].$$ Since $Q_T \log(T)$ is positive for $T>1$ and converges a.s. to $\frac{\alpha -(d+1)}{d}$, the first expectation goes to~$0$ while the second one is bounded by using again~\eqref{ui1}. Therefore,
$A_T-\tilde{A}_T\underset{T \rightarrow +\infty}\rightarrow 0$, and we  get
$$\E_\theta \left[ \exp \left( \frac{\mu N_T}{\sqrt{\log(T)}} - \Tr \left[\frac{\lambda_2}{T}X_T  \right] - \Tr \left[ \frac{\lambda_1}{T^2} R_T\right] \right) \right] \underset{T \rightarrow +\infty}\rightarrow  \frac{\exp \left( \frac{\mu^2 d}{2(\alpha -(d+1))} \right) }{\det[V]^{ \frac{\alpha}{2}}}.$$
Thus, $\left( \frac{  N_T}{\sqrt{\log(T)}}, \frac{ X_T}{T}, \frac{R_T}{T^2} \right)$ converges in law to $(\sqrt{\frac{d}{\alpha -(d+1)}}G,X^0_1,R^0_1)$, where $G\sim \mathcal{N}(0,1)$ is independent of $X^0$. From~\eqref{eq_Z}, we have
\begin{align*}
Q_T Z_T +1+d=2\frac{1}{\sqrt{\log(T)}} \log(T) Q_T  \frac{N_T}{\sqrt{\log(T)}} +\alpha,
\end{align*}
and therefore $Q_T Z_T +1+d$ converges in probability to $\alpha$. Slutsky's theorem gives then the following convergence in law: as $T\rightarrow +\infty$,
\begin{equation}\label{cv_loi_1}\left( \frac{  N_T}{\sqrt{\log(T)}}, \frac{ X_T}{T}, \frac{R_T}{T^2} , Q_T Z_T +1+d , Q_T \log(T) \right) \Rightarrow \left(\sqrt{\frac{d}{\alpha -(d+1)}}G,X^0_1,R^0_1,\alpha, \frac{\alpha-(d+1)}{d} \right).
\end{equation}
This gives the claimed convergence for $(\hat{\alpha}_T,\hat{b}_T)$ due to the continuity property given in Lemma~\ref{lemme-linear}.
\end{proof}

\begin{mytheo}\label{b0_alphadplusun}
Assume that $b=0$ and  $\alpha=d+1$. Let $(\hat{b}_T, \hat{\alpha}_T)$ be the MLE defined by~(\ref{eq:estimator_couple}). Then, $( T(\hat b_T -b ),\log(T)(\hat \alpha_T-\alpha))$ converges in law under $\Px_\theta$  when $T\rightarrow +\infty$ to 
$$\left(\mathcal{L}^{-1}_{R_1^0}\left( X_1^0- \alpha I_d  
\right) ,\frac{4}{d\tau_1} \right),$$ where 
$X^0_t=\alpha t I_d + \int_0^t \sqrt{X^0_s} dW_s +dW_s\tp \sqrt{X^0_s}$ is a Wishart process with the same parameters but starting from~$0$, $R_t^0 =\int_0^t X_s^0ds$ and $\tau_1=\inf \{t\geb 0, B_t=1\}$ where $B$ is a standard Brownian motion independent from~$W$. 
\end{mytheo}

\begin{proof}
The proof follows the same line as the one of Theorem~\ref{b0}, but we now write 
$$ \log(T)(\hat\alpha_T-\alpha) = -2 \frac{T \Tr[\hat b_T]}{\log(T)} \log(T)^2Q_T  +2 \log(T)^2Q_T \frac{N_T}{\log(T)},$$
while we still have $T\hat b_T =\mathcal{L}^{-1}_{\frac{R_T}{T^2},Q_T}\left(\frac{X_T}{T}- \frac{x}{T}-(Q_TZ_T+1+d)I_d \right)$.
By Theorem 4.1 in~\cite{MayerhoferHammamet}, for $\mu \geb 0$ and $T>1$, $\frac{d \overline{\Px}}{d \Px}= \exp \left( \frac{\mu N_T}{\log(T)}- \frac{\mu^2}{2 Q_T \log(T)^2} \right)$ defines a change of probability, and we define for $\lambda_1,\lambda_2\in \dpos$,
$$A_T=\E_{\theta} \left[ \exp \left( \frac{\mu N_T}{{\log(T)}}- \frac{\mu^2}{2 Q_T \log(T)^2}  \right)\exp \left( - \Tr \left[\frac{\lambda_2}{T}X_T  \right] - \Tr \left[ \frac{\lambda_1}{T^2} R_T\right] \right) \right].$$
By Proposition~\ref{prop:lap_Wish}, we have
\begin{align*}
A_T 
&= 
\frac{1}{ \det[V]^{ \frac{\alpha+\mu/{\log(T)}}{2}}} \exp \Big(- \frac{1}{2T}\Tr \big[ V' V^{-1} x  \big] \Big)\underset{T \rightarrow +\infty}\rightarrow  \frac{1}{\det[V]^{ \frac{\alpha}{2}}},
\end{align*}
where $V$ and $V'$ are defined by~\eqref{defV_Vprime}.

We now use that $\frac{N_T}{\log(T)}\rightarrow \frac{d}{2}$ in probability, see Lemma~\ref{lem_yor}, and define
\begin{align*}
\tilde{A}_T&=\E_{\theta} \left[ \exp (\xi_T)\right],\ \xi_T=   \frac{\mu d}{2 } - \frac{\mu^2}{2 Q_T \log(T)^2} - \Tr \left[\frac{\lambda_2}{T}X_T  \right] - \Tr \left[ \frac{\lambda_1}{T^2} R_T \right]
\end{align*}
and have
$$A_T=\tilde{A}_T+\E_{\theta} \left[ \left\{ \exp \left( \frac{\mu N_T}{ \log(T) } - \frac{\mu d}{2}\right)-1 \right\}\exp \left( \xi_T \right) \right].$$
We note that $\exp (\xi_T)\leb \exp \left(\frac{\mu d}{2 } \right)$. By using Lemma~\ref{lem_yor} and the uniform integrability~\eqref{ui2}, we get that $A_T-\tilde{A}_T \underset{T \rightarrow +\infty}{\rightarrow} 0$ and therefore 
$$ \E_{\theta} \left[ \exp \left(- \frac{\mu^2}{2 Q_T \log(T)^2} - \Tr \left[\frac{\lambda_2}{T}X_T  \right] - \Tr \left[ \frac{\lambda_1}{T^2} R_T \right]\right) \right]\underset{T \rightarrow +\infty}\rightarrow \frac{\exp(-\mu d/2)}{\det[V]^{ \frac{\alpha}{2}}}. $$
Therefore, $\left(  \frac{ X_T}{T}, \frac{R_T}{T^2}, Q_T\log(T)^2 \right)$ converges in law to $\left(X^0_1,R^0_1, \left(\frac{2}{d} \right)^2 \frac{1}{\tau_1} \right)$, where $\tau_1$ is independent of $X^0$.  We observe that
$Q_TZ_T=\frac{1}{\log(T)} Q_T\log(T)^2 \frac{Z_T}{\log(T)} $. 
Lemma~\ref{lem_yor} and Slutsky's theorem gives 
\begin{equation}\label{cv_loi_2}\left( \frac{  N_T}{ \log(T)}, \frac{ X_T}{T}, \frac{R_T}{T^2} , Q_T Z_T +1+d , Q_T \log(T)^2 \right) \Rightarrow \left( \frac{d}{2}, X^0_1,R^0_1,d+1, \left(\frac{2}{d} \right)^2 \frac{1}{\tau_1} \right), 
\end{equation}
which gives the claim by using the formulas for $T\hat b_T$ and $\log(T)(\hat\alpha_T-\alpha)$. 
\end{proof}

\subsection{The MLE of $b$}

Until the end of this section we consider that $\alpha \geb d-1$ is known and study the speed of convergence of the estimator of $b$ defined by~\eqref{MLE_bseul}. 

\subsubsection{Case $b =0 $.}

\begin{mytheo}
\label{th:estim_b=0_b}
Assume that $b=0$ and  $\alpha \geb d-1$. For $T>0$, let $\hat{b}_T$ be defined by (\ref{MLE_bseul}). When $ T\rightarrow +\infty$, $T( \hat b_T-b)$ converges in law under $\Px_\theta$ to 
$\mathcal{L}^{-1}_{R_1^0}\left( X_1^0-\alpha I_d  \right)$,
where $(X_t^0)_{t \geqslant 0}$ is the solution to $X^0_t=\alpha t I_d + \int_0^t \sqrt{X^0_s} dW_s +dW_s\tp \sqrt{X^0_s}$  and $R_t^0 =\int_0^t X_s^0ds$.
\end{mytheo}

\begin{proof}
From~\eqref{MLE_bseul}, we have $T \hat{b}_T= \mathcal{L}^{-1}_{\frac{R_T}{T^2}}\left(\frac{X_T}{T} - \frac{x}{T}- \alpha I_d\right)$. Let $V$ and $V'$ be defined by~\eqref{defV_Vprime}. Similarly to~\eqref{lim_A_T}, we have by Proposition~\ref{prop:lap_Wish} for $\lambda_1,\lambda_2 \in \dpos$
$$\E_{\theta} \left[ \exp \left( - \Tr \left[\frac{\lambda_2}{T}X_T  \right] - \Tr \left[ \frac{\lambda_1}{T^2} R_T\right] \right) \right]= \frac{1}{ \det[V]^{ \frac{\alpha}{2}}} \exp \Big(- \frac{1}{2T}\Tr \big[ V'V^{-1} x  \big] \Big)\underset{T \rightarrow +\infty}\rightarrow  \frac{1}{\det[V]^{ \frac{\alpha}{2}}}.$$
This gives the convergence in law of $\left(\frac{ X_T}{T}, \frac{R_T}{T^2}\right)$ to $(X^0_1,R^0_1)$ and then the claimed result.
\end{proof}

\subsubsection{Case $b=b_0I_d, b_0>0$.}

In this case $b=b_0I_d$ with $b_0>0$. In order to identify the speed of convergence and the limit law, we use the Laplace transform approach. We have the following result,
\begin{mytheo}\label{theo_b0_pos}
Assume that $b=b_0 I_d$, $b_0 >0$, and  $\alpha \geb d-1$. For $T>0$ let $\hat{b}_T$ defined by (\ref{MLE_bseul}). When $T \rightarrow + \infty$, $\exp(b_0 T)( \hat b_T-b)$ converges in law under~$\Px_\theta$ to  $\mathcal{L}^{-1}_{X}\left(\sqrt{X}  \tilde{\mathbf{G}}+\tilde{\mathbf{G}} \sqrt{X} \right) $ where $X \sim WIS_d\left(\frac{x}{2b_0},\alpha,0,I_d;\frac{1}{4b_0^2} \right) $  and  $\tilde{\mathbf{G}}$ is an independent $d$-square matrix whose elements are independent standard Normal variables. 
\end{mytheo} 
The proof of this results relies on the explicit calculation of the Laplace transform of $(X_T,R_T)$ and is postponed to Subsection~\ref{subsec_lapl_MLE}. 

Obviously, the case $b=b_0 I_d$ is very particular. One would like to consider more general nonergodic cases or ideally to be able to state a general convergence results of $\hat{b}_T$ towards $b$ for any $b\in \symm$. Despite our efforts, we have not been able to get such a result. The reason why we can handle the ergodic case and the nonergodic case $b=b_0 I_d$ with $b_0\geb 0$ is that the convergence of all the matrix terms occurs at the same speed, namely $1/\sqrt{T}$ for the ergodic case, $1/T$ for $b=0$ and $e^{-b_0T}$ when $b_0>0$. In the other cases, there is no such a simple scalar rescaling. Heuristically, there may be different speeds of convergence that are difficult to disentangle because of the different matrix products. To get an idea of this, we present now the case of the estimation of $b$ when $b$ is known to be a diagonal matrix. In this case, we obtain different speed of convergence for each diagonal terms.

\subsubsection{The MLE of $b$ when $b$ is known a priori to be diagonal.}

We assume that $\alpha \geb d-1$ is known and that $b$ is a diagonal matrix, i.e. $b=\diag(b_1, \cdots , b_d)$. We want to estimate the diagonal elements by maximizing the likelihood. We denote $\theta_0=(0,\alpha)$. As in~\eqref{likeli_bseul}, we have 
\begin{equation*}
L_T^{\theta,\theta_0}=\exp\Bigl( \frac{\Tr[bX_T]-{\Tr[bx]}}{2}-\frac{1}{2}\int_0^T \Tr[b^2X_s]ds -\frac{\alpha T}{2}\Tr[b] \Bigr).
\end{equation*}
By differentiating this with respect to $b_i$, $1\leb i \leb d$, we get
\begin{align*}
\frac{\partial_{b_i} L_T^{\theta,\theta_0}}{L_T^{\theta,\theta_0}} =\frac{1}{2} \Big((X_T)_{i,i}-x_{i,i}- \alpha t I_d - 2b_i \int_0^T (X_s)_{i,i}ds  \Big),
\end{align*}
and therefore the MLE of $b$ is  given by
\begin{align}
(\hat{b}_T)_i = \frac{(X_T)_{i,i}-x_{i,i}- \alpha T}{2 (R_T)_{i,i}}.
\label{eq:estimator_diag_b}
\end{align}
We therefore obtain
\begin{align}
(\hat{b}_T)_i -b_i= \frac{(X_T)_{i,i}-x_{i,i}- \alpha T -2b_i (R_T)_{i,i} }{2 (R_T)_{i,i} }.
\label{eq:erreur_estim_b_diag}
\end{align}
Let us observe that this estimator is precisely the one obtained by Ben Alaya and Kebaier~\cite{Alaya_Kebaier_2013} for the CIR process. This is not very surprising since we know from~\eqref{SDE_Intro2}, \eqref{eq:crochet_general_de_X} and $b$ diagonal that there exists independent Brownian motions $\beta^i$, $1\leb i \leb d$ such that
$$d(X_t)_{i,i}=(\alpha+2b_i(X_t)_{i,i})dt+2 \sqrt{(X_t)_{i,i}}d\beta^i_t.$$
Thus, the diagonal elements follow independent CIR processes, and the observation of the non diagonal elements does not improve the ML estimation. We can obtain the asymptotic convergence by applying Theorem~1 in~\cite{BK2011}, up to a small correction in the nonergodic case which is given by our Theorem~\ref{theo_b0_pos} in dimension $d=1$. This yields to the following proposition. 
\begin{prop}Let $\alpha \geb d-1$ and $b$ a diagonal matrix. 
Let $\epsilon_t = \diag(\epsilon_t^1, \cdots, \epsilon_t^d)$ be a diagonal matrix with

\[
{\epsilon_t^i =}\left\{
\begin{array}{r c l}
t^{-\frac{1}{2}}  \quad & \text{if} & b_i<0 \\ 
t^{-1}  \quad &  \text{if} & b_i=0 \\
\exp(-b_i t)  \quad & \text{if} & b_i>0 
\end{array}
\right.
\]
Then, under $\Px_\theta$, $\epsilon^{-1}_T \diag(( \hat b_T)_1-b_1, \cdots,(\hat b_T)_d-b_d )$ converges in law to a diagonal matrix  $\mathbf D$ made with independent elements. Each diagonal element  $\mathbf D_i$ is distributed as follows:
\[
{\forall i \in \{1, \cdots,d \}, \quad   \mathbf D_i  \underset{law}{=} }\left\{
\begin{array}{r c l}
& \sqrt{\frac{-2 b_i}{\alpha}} \mathbf G    \quad & \mbox{if }  b_i<0 \\ 
&  \frac{X_1^0-\alpha }{2R_1^0}  \quad & \mbox{if }  b_i=0 \\
& \frac{\mathbf G  }{ \sqrt{X^{x_{i,i}/(2b_i)}_{1/(4b_i^2)}}},  \quad & \mbox{if } b_i>0 
\end{array}
\right.
\]
where $X^x_t=x+\alpha t + 2 \int_0^t \sqrt{X^x_s}dW_s$, $R_t^0 =\int_0^t X_s^0ds$, and $\mathbf G \sim \mathcal{N}(0, 1)$ is independent of~$X$.
\end{prop}

\section{Optimality of the  MLE}\label{Sec_LAN}
In parametric estimation theory, a fundamental role is played by the local asymptotic normality
 ({\bf LAN})  property  since the work of Le Cam \cite{Lecam}. This general concept developed 
by Le Cam  is extended later by 
Le Cam and Yang   \cite{LecYan} and  Jeganathan  \cite{Jeg}  to local asymptotic mixed normality ({\bf LAMN}) and  local asymptotic quadraticity ({\bf LAQ}) properties. These notions are mainly dedicated to study the asymptotic  efficiency of estimators of a given parametric model. 
The aim of this section is to  check the validity of either {\bf LAN},  {\bf LAMN} or  {\bf LAQ} properties for the global model in order to get the  asymptotic efficiency of our maximum likelihood estimators studied in the previous section. Here we prove these properties only for the global model $\theta=(b,\alpha)$ when $b$ is known to be symmetric. The same technique applies for all the other cases considered in this paper where we have been able to obtain the corresponding local asymptotic property. 

Let us consider the Wishart process $(X_t)_{t \geqslant 0} \in  \mathcal{S}_d^+$ with parameters  $\theta:=(b,\alpha)$, with  $\alpha\geq d+1$ and $b \in
\mathcal{S}_d$. 
\begin{eqnarray}\label{Wish}
\left\{
    \begin{array}{ll}
        dX_t=\left[\alpha I+bX_t+X_tb \right]dt+ \sqrt{X}_tdW_t+dW_t \tp \sqrt{X}_t , t>0 \\
        X_0 \in \mathcal{S}_d^+.
    \end{array}
\right.
\end{eqnarray}
We recall that $\mathbb P_{\theta}$ denotes the distributions induced by the solutions of \eqref{Wish} on 
canonical space $C(\mathbb R_+, \posm )$  with the natural filtration $\mathcal F^X_t:=\sigma(X_s, s\leq t)$ and  
$\mathbb P_{\theta,t}={\mathbb P_{\theta}}_{|\mathcal F^X_t} $ denotes the restriction of $\mathbb P_{\theta}$ on the filtration $\mathcal F_t$.

For $\tilde \alpha\geq d+1$ and $\tilde b\in\mathcal{S}_d^+$, we set
$\tilde \theta=(\tilde b, \tilde\alpha)$,
$$H_t=\frac{\tilde \alpha-\alpha}{2}(\sqrt{X_t})^{-1} +(\tilde b-b) \sqrt{X_t},$$ and  we introduce the log-likelihood  function  
\begin{eqnarray}\label{log-likelihood}
 \ell^{\theta}_T(\tilde \theta)&=&\log\left(\frac{d\mathbb P_{\tilde \theta,T}}{d\mathbb P_{\theta,T}}\right)
 ={\int_0^T \Tr[H_sd
    W_s]-\frac{1}{2}\int_0^T \Tr[H_sH_s^\top]ds}
\end{eqnarray}

The process $(\tilde{W}_t=W_t-\int_0^t H^\top_sds, t\leq T)$ is a $d \times d$-Brownian motion
under ${\mathbb{P}}_{\tilde\theta,T}$.
In the sequel, let us introduce the quantity $\delta_T:=(\delta_{1,T},\delta_{2,T})\in \R^2$ where  
for $i\in\{1,2\}$  the  localizing rates satisfy  $|\delta_{i,T}|\rightarrow 0$ when $T\rightarrow\infty$.  For all  $u:=(u_1,u_2)\in\R\times \mathcal{S}_d $, we define $\delta_{T}\cdot u
:=(\delta_{1,T}u_{1},\delta_{2,T} u_{2})\in \R\times \mathcal{S}_d$.
Now, we rewrite  \eqref{log-likelihood}  with $\tilde\theta=\theta+\delta_{T}\cdot u$ 
\begin{eqnarray*}
\ell^{\theta}_T(\theta+\delta_{T}\cdot u)&=&\int_0^T\Tr\left[\frac{\delta_{1,T}u_1(\sqrt{X_s})^{-1}}{2}d W_s
+\delta_{2,T}u_2\sqrt{X_s} d W_s\right]\\
&&- \frac{1}{2}\int_0^T \Tr\left[ \delta_{2,T}u_2X_s \delta_{2,T} u_2 \right]ds
- \frac{T}{2}\delta_{1,T}u_1 \Tr\left[\delta_{2,T}u_2 \right]\\
&&- \frac{1}{8}\int_0^T (\delta_{1,T}u_1)^2\Tr\left[(X_s)^{-1}\right]ds.
\end{eqnarray*}
Hence, by using the definitions~\eqref{eq:def_quantities}, \eqref{eq:def_Mt} and~\eqref{eq:def_Nt} of the martingales processes $(N_t)_{t\geq 0}$ and $(M_t)_{t\geq 0}$ and the processes $(R_t)_{t\geq 0}$ and $(Q_t)_{t\geq 0}$, it is easy to check that
\begin{eqnarray}
 \ell_T^{\theta}(\theta +\delta_{T}\cdot u)&=& \frac{1}{2}\Bigl(\delta_{1,T}u_1 N_T+\Tr\bigl[\delta_{2,T
}u_2M_T\bigr]\Bigr)  
 -\frac{1}{2} \Tr\left[ \delta_{2,T}u_2R_T \delta_{2,T} u_2 \right]  \nonumber\\
&&- \frac{T}{2}\delta_{1,T}u_1 \Tr\left[\delta_{2,T} u_2 \right]-
 \frac{1}{8}(\delta_{1,T}u_1)^2(Q_T)^{-1}  \nonumber\\
&=&\Lambda_T(u) - \frac{1}{2} \Gamma_T(u), \label{dev}
\end{eqnarray}
where 
$
 \Lambda_T(u)=\frac{1}{2}\Bigl(\delta_{1,T}u_1 N_T+\Tr\bigl[\delta_{2,T} u_2M_T\bigr]\Bigr)  
$
is a linear random function  with respect to $u\in\R\times\posm$  with quadratic variation 
\begin{equation*}
 \Gamma_T(u)= \delta_{2,T}^2 \Tr[u_2^2 R_T] 
+ T \delta_{1,T} \delta_{2,T}u_1 \Tr\left[u_2 \right]+ \frac{1}{4}  \delta_{1,T}^2u_1^2 Q_T^{-1}.
\end{equation*}

\subsection{Case $-b\in S^+_d$ and $\alpha\geb d+1$}
We first consider $\alpha>d+1$.
In this ergodic case, we set $\delta_{i,T}=T^{-1/2}$ for $i\in\{1,2\}$, and we get from~\eqref{Birk1} and~\eqref{Birk2}
\begin{equation}\label{raykov}
 \Gamma_T(u)\overset{a.s.}{\longrightarrow} \overline\Gamma_{\infty}(u):= \Tr\left[ u^2_2\overline R_{\infty} \right] 
 + u_1 \Tr\left[u_2 \right]+
 \frac{1}{4}u_1^2\Tr\left[(\overline Q_{\infty})^{-1}\right],\quad\mbox{ as
}T\rightarrow+\infty.
\end{equation}
This yields the validity of the so called Raykov type condition. 
Hence, according to Theorem~1 in \cite{Luc92},  relations \eqref{dev} and \eqref{raykov}  ensure the validity of the local asymptotic normality ({\bf LAN}) property, that is  under $\mathbb P_{\theta}$ we have 
\begin{equation}\label{lan}
\left(\Lambda_T(u) , \Gamma_T(u) \right)\Rightarrow \left(\overline\Gamma_{\infty}^{1/2}(u)Z,\overline\Gamma_{\infty}(u)\right),\;\; \mbox{as } T\rightarrow\infty,
\end{equation}
with $Z$ a standard normal real random variable. 
It is worth noting that the above convergence  can also be obtained using the proof of Theorem \ref{thm_MLE_global_ergo}.  In fact,  we have already proven that under $\mathbb P_{\theta}$  
\begin{equation}\label{limlaw}
\left(\frac{N_T}{\sqrt{T}},\frac{M_T}{\sqrt{T}}\right)\Rightarrow (\tilde{\mathbf G}, \tilde{H})
\end{equation}
where $(\tilde G, \tilde H)$ is a centered Gaussian vector taking values in $\R\times S^+_d$ such that
 \begin{align*}
 \E_\theta(\tilde{\mathbf{G}}_{i,j}\tilde{\mathbf{G}}_{k,l})&=\left[\delta_{jl}(\rbi)_{i,k}+\delta_{jk}(\rbi)_{i,l}+
 \delta_{il}(\rbi)_{j,k}+\delta_{ik}(\rbi)_{j,l}\right], \label{loi_Gtilde}\\
 \E_\theta(\tilde{\mathbf{G}}_{i,j}\tilde{H})&=2 \delta_{i,j}  \text{ and } \E_\theta(\tilde{H}^2)= 
 \qbi^{-1}. \nonumber
 \end{align*}
Therefore, {\bf LAN} property \eqref{lan} follows from relations  \eqref{raykov} and \eqref{limlaw}. 

We now consider the case  $\alpha=d+1$ and set  $\delta_{1,T}=T^{-1}$ and   $\delta_{2,T}=T^{-1/2}$. By using~\eqref{cv_loi_3}, we get that 
under $\mathbb P_{\theta}$,
\begin{equation*}
\left(\Lambda_T(u) , \Gamma_T(u) \right)\Rightarrow 
\left(-\frac{u_1}{2}\Tr[b]+\Tr\left[u_2 \sqrt{\rbi} \tilde{\mathbf G}\right] ,
\Tr[u_2^2\overline R_{\infty}]+\frac{1}{4}u_1^2\tau_{_{-\Tr[b]}}\right),\;\; \mbox{as } 
T\rightarrow\infty,
\end{equation*}
where $\tau_{_{-\Tr[b]}}$ is defined as in Theorem~\ref{alphad} and $ \tilde{G}$ is an independent matrix, whose elements $\tilde{G}_{i,j}$, $1\leb i,j\leb d$, are independent standard normal variables.
Hence, according to  Le Cam and Yang  \cite{LecYan} and  Jeganathan  \cite{Jeg} this last convergence yields the {\bf LAQ} property for this ergodic case.
\subsection{Case $b=0$ and $\alpha \geb d+1$}
We first assume $\alpha>d+1$. From~\eqref{Wish} with $b=0$ and~\eqref{eq:def_Mt}, we have $M_T=X_T-x-\alpha I_dT$. From~\eqref{cv_loi_1}, it follows that as $T\rightarrow\infty$ 
$$\left( \frac{  N_T}{\sqrt{\log(T)}}, \frac{ M_T}{T}, \frac{R_T}{T^2}, \frac{ Q^{-1}_T}{\log(T)} \right) \Rightarrow\left(\sqrt{\frac{d}{\alpha -(d+1)}}G,X^0_1-\alpha I_d,R^0_1,\frac{d}{\alpha-(d+1)} \right),$$
where $X^0_1$ and $R^0_1$ are defined as in Theorem~\ref{b0}.
Thus, in the same way as in the previous case if we set
 $\delta_{1,T}=\frac{1}{\sqrt{\log(T)}}$ and   $\delta_{2,T}=T^{-1}$, then $\left(\Lambda_T(u) , \Gamma_T(u) \right)$ converges in law under $\Px_\theta$ to
\begin{multline*}
\left(\frac{1}{2}\sqrt{\frac{d}{\alpha -(d+1)}}u_1G +\frac{1}{2}\Tr\left[u_2 (X_1^0-\alpha I_d)\right] ,
\Tr\left[u_2^2\overline R_1^0\right]+ \frac{u_1^2d}{4(\alpha-(d+1))}\right),\;\; \mbox{as } 
T\rightarrow\infty.
\end{multline*}
This ensures the validity of the  {\bf LAQ}  property in this non-ergodic case.

When $\alpha=d+1$, we use the notation of Theorem~\ref{b0_alphadplusun} and  get from~\eqref{cv_loi_2}
$$\left( \frac{  N_T}{\log(T)}, \frac{ M_T}{T}, \frac{R_T}{T^2}, \frac{ Q^{-1}_T}{\log(T)^2} \right) \Rightarrow\left( \frac{d}{2},X^0_1-\alpha I_d,R^0_1, \left(\frac{d}{2}\right)^2 \tau_1 \right).$$
With $\delta_{1,T}=\frac{1}{\log(T)}$ and   $\delta_{2,T}=T^{-1}$, we get that $\left(\Lambda_T(u) , \Gamma_T(u) \right)$ converges in law under $\Px_\theta$ to
\begin{align*}
\left(\frac{d}{4} u_1+\frac{1}{2}\Tr[u_2(X^0_1-\alpha I_d)] ,
\Tr\left[u_2^2\overline R_1^0\right]+ \frac{d^2 u_1^2}{8}   \tau_1\right),\;\; \mbox{as } 
T\rightarrow\infty.
\end{align*}
This gives again the {\bf LAQ}  property.
\section{The Laplace transform and its use to study the MLE}\label{Sec_laplace}

\subsection{The Laplace transform of $(X_T, R_T)$}

We present our main result on the joint Laplace transform of $(X_T,R_T)$, that can be of independent interest. This Laplace transform is given by Bru~\cite{Bru}, eq.~(4.7) when $b=0$ and has been recently studied and obtained explicitly by Gnoatto and Grasselli~\cite{GnoattoGrasselli}. Here, we present another proof that enables us to get the Laplace transform for any $\alpha\geb d-1$, as well as a more precise result concerning its set of convergence, see Remarks~\ref{rk_setcv} and~\ref{rk_ext_Aff} below for a further discussion.  

\begin{prop}\label{prop:lap_Wish} Let $\alpha\geb d-1$, $x\in \posm$, $b\in \symm$ and $X\sim WIS_d(x,\alpha,b,I_d)$.
Let $v,w \in \mathcal{S}_d$ be such that 
\begin{equation}\label{dom_lap}\exists m \in \symm, \ \frac{v}{2}-mb-bm-2m^2 \in \posm \text{ and } \frac{w}{2}+m \in \posm. 
\end{equation}
Then, we have for $t\geb0$
\begin{align}
\mathbb{E}& \Big[ \exp \Big( -\frac{1}{2}\Tr [wX_t]- \frac{1}{2} \Tr [v R_t\big]\Big) \Big]  \nonumber \\
	&= \frac{\exp \big(- \frac{\alpha}{2} \Tr[b] t\big)}{\det[V_{v,w}(t)]^{ \frac{\alpha}{2}}} \exp \Big(- \frac{1}{2}\Tr \big[ (V'_{v,w}(t)V_{v,w}(t)^{-1}+b  )x  \big] \Big),\label{laplace_Wis_jointe}
\end{align}
with
\begin{align*}
    	V_{v,w}(t)= \left( \sum_{k=0}^\infty t^{2k+1} \frac{\tilde{v}^{k}}{(2k+1)!} \right)\tilde{w}+\sum_{k=0}^\infty t^{2k} \frac{\tilde{v}^k}{(2k)!}, \quad	\tilde{v}=v+b^2, \quad \mbox{and} \quad		\tilde{w}=w-b.
\end{align*}
If besides $\tilde{v}=v+b^2\in \dpos$, we have $V_{v,w}(t)= (\sqrt{\tilde{v}})^{-1} \sinh(\sqrt{\tilde{v}} t)\tilde{w}+\cosh(\sqrt{\tilde{v}} t)$ and then $V'_{v,w}(t)=  \cosh(\sqrt{\tilde{v}} t)\tilde{w}+\sinh(\sqrt{\tilde{v}} t)\sqrt{\tilde{v}}$.
\end{prop} 

Before proving this result, we recall the following fact:
\begin{equation}\label{trace_posit}
\forall x,y\in \posm, \ \Tr[xy]\geqslant 0,
\end{equation}
which is clear once we have observed that $\Tr[xy]=\Tr[\sqrt{x}y\sqrt{x}]$ and $\sqrt{x}y\sqrt{x} \in \posm$.  We also recall a result on matrix Riccati equations, see Dieci and Eirola~\cite{DieciEirola} Proposition 1.1.
\begin{lemme}\label{lem_riccati} Let $\tilde{b}\in \symm$ and $\tilde{\delta}\in \posm$.  Let~$\xi$ denote the solution of the following matrix Riccati differential equation
\begin{equation}\label{Ric_xi}\xi'+2\xi^2=\tilde{b}\xi+ \xi\tilde{b} + \tilde{\delta}, \ \xi(0)\in \symm. 
\end{equation}
If $\xi(0) \in \posm$, the solution~$\xi$ is well-defined for any $t\geqslant0$ and satisfies $\xi(t)\in \posm$. 
\end{lemme}

\begin{proof}[Proof of Proposition~\ref{prop:lap_Wish}]
Let $T>0$ be given. We first assume $w,v \in \posm$, which ensures that $\E\left[\exp\left(-\frac{1}{2} \Tr[w X_T + v R_T ] \right)\right]< \infty$. We consider the martingale
$$ M_t=\E\left[\exp\left(-\frac{1}{2} \Tr[w X_T + v R_T] \right) \bigg|\mathcal{F}_t\right], \ t\in[0,T].$$
Due to the affine structure, we are looking for smooth functions $\beta:\R_+ \rightarrow \R$, $\gamma,\delta:\R_+ \rightarrow \symm$ such that
$$M_t=\exp \left(\beta(T-t)+\Tr[\gamma(T-t)X_t]+\Tr[\delta(T-t)R_t] \right).$$
We necessarily have $\beta(0)=0$, $\gamma(0)=-w/2$ and $\delta(0)=-v/2$. Itô's formula gives
\begin{align*}
\frac{dM_t}{M_t}=&\Big\{-\beta'(T-t)-\Tr[\gamma'(T-t) X_t]- \Tr[\delta'(T-t)R_t ] + \Tr[\gamma(T-t)(\alpha I_d+bX_t+X_tb)] \\
&+ \Tr[\delta(T-t)X_t]+ 2 \Tr[\gamma(T-t)^2 X_t]  \Big\}dt + \Tr[\gamma(T-t)(\sqrt{X_t}dW_t +dW_t\tp \sqrt{X_t})]. \\
\end{align*}
Since $M$ is a martingale, the drift term should vanish almost surely. The drift term being a (deterministic) affine function of $(X_t,R_t)$, we obtain the following system of differential equations:
\begin{align}
&\delta'=0,  \label{eq_delta} \\
&-\gamma' + \gamma b +b \gamma + 2 \gamma^2 + \delta =0,  \label{eq_gamma}\\
&-\beta' + \alpha  \Tr[\gamma]=0.\label{eq_beta}
\end{align}
The first equation gives $\delta(t)=-v/2$. The second equation is a matrix Riccati differential equation. We now consider $\xi=m-\gamma$ with $m$ satisfying~\eqref{dom_lap}. It solves~\eqref{Ric_xi} with $\tilde{b}=b+2m$, $\tilde{\delta}=-bm-mb-2m^2+v/2$ and $\xi(0)=m+w/2$. We know then by Lemma~\ref{lem_riccati} that $\xi$ is well  defined for any $t\geb 0$ and stays in $\posm$. In particular,  $\gamma$  is well defined for any $t\geb 0$.

We set $\tilde{\gamma}=\gamma + \frac{1}{2}b$. We have $\gamma^2=\tilde{\gamma}^2-\frac{1}{2}(b \tilde{\gamma} +\tilde{\gamma}b) +\frac{1}{4}b^2$ and thus $\tilde{\gamma}$ solves the following  matrix Riccati differential equation: 
\begin{align*}
 \tilde{\gamma}'= 2\tilde{\gamma}^2 - \frac{1}{2} \tilde{v}, \ \tilde{\gamma}(0)=-\frac{1}{2}\tilde{w},
\text{ with } \tilde{v}=v+b^2 \text{ and } \tilde{w}=w-b.
\end{align*}
We set $M(t)=\left[\begin{array}{cc} M_1(t) &M_2(t) \\M_3(t) &M_4(t)
  \end{array}
  \right] = \exp \left( t \left[\begin{array}{cc} 0 & -\tilde{v}/2 \\ -2I_d &0
  \end{array}
  \right]\right) \in \mathcal{M}_{2d}$ and get by Levin~\cite{Levin} that  
$$ \tilde{\gamma}(t)=\left[M_2(t) -\frac{1}{2}M_1(t)\tilde{w} \right] \left[M_4(t) -\frac{1}{2}M_3(t)\tilde{w} \right]^{-1}.$$
We check that the matrix $M_4(t) -\frac{1}{2}M_3(t)\tilde{w}$ is indeed invertible. In fact, let $$\tau=\inf \left\{ t\geb 0, \det \left[ M_4(t) -\frac{1}{2}M_3(t)\tilde{w} \right]=0\right\}.$$ We have $\tau>0$ and for $t\in[0,\tau)$,
$\frac{d}{dt}[M_4(t) -\frac{1}{2}M_3(t)\tilde{w}]=-2\left[M_2(t) -\frac{1}{2}M_1(t)\tilde{w} \right]$ and thus $$\frac{d}{dt} \det \left[M_4(t) -\frac{1}{2}M_3(t)\tilde{w}\right]=-2 \det \left[M_4(t) -\frac{1}{2}M_3(t)\tilde{w}\right] \Tr[\tilde{\gamma}(t)].$$ This gives $\det \left[M_4(t) -\frac{1}{2}M_3(t)\tilde{w}\right]=\exp(-2 \int_0^t \Tr[\tilde{\gamma}(s)]ds )>0$, and we necessary get $\tau=+\infty$ since $\gamma$ and thus $\tilde{\gamma}$ is well defined for $t\geb 0$.

Since 
$$ \left[\begin{array}{cc} 0 & -\tilde{v}/2 \\ -2I_d &0
  \end{array}
  \right]^{2k}= \left[\begin{array}{cc} \tilde{v}^k & 0 \\ 0 & \tilde{v}^k
  \end{array}
  \right] \text{ and } \left[\begin{array}{cc} 0 & -\tilde{v}/2 \\ -2I_d &0
  \end{array}
  \right]^{2k+1}=\left[\begin{array}{cc} 0 & -\tilde{v}^{k+1}/2 \\ -2\tilde{v}^k &0
  \end{array}
  \right],$$ we get
$$M_1(t)=M_4(t)=\sum_{k=0}^\infty \frac{ t^{2k} \tilde{v}^k}{(2k)!}, \ M_2(t)=-\frac{1}{2} \sum_{k=0}^\infty \frac{t^{2k+1} \tilde{v}^{k+1}}{(2k+1)!}, \ M_3(t)=-2 \sum_{k=0}^\infty  \frac{t^{2k+1}\tilde{v}^{k}}{(2k+1)!}.$$
If  $\tilde{v}=v+b^2 \in \dpos$, $\sqrt{\tilde{v}}$ is well defined and we have $M_1(t)=M_4(t)=\cosh(t \sqrt{\tilde{v}})$, $M_2(t)=-\frac{1}{2} \sqrt{\tilde{v}}\sinh(t \sqrt{\tilde{v}})$ and $M_3(t)=-2 (\sqrt{\tilde{v}})^{-1}\sinh(t \sqrt{\tilde{v}})$.  Now, we define
$$V(t)=  M_4(t) -\frac{1}{2}M_3(t)\tilde{w}= \left( \sum_{k=0}^\infty t^{2k+1} \frac{\tilde{v}^{k}}{(2k+1)!} \right)\tilde{w}+\sum_{k=0}^\infty t^{2k} \frac{\tilde{v}^k}{(2k)!}.  $$ Since $V'(t)= - 2 M_2(t) +M_1(t)\tilde{w}  $, we obtain that
$$ \tilde{\gamma}(t)=  -\frac{1}{2} V'(t)V(t)^{-1} \text{ and thus }\gamma(t)=  -\frac{1}{2} \left( V'(t)V(t)^{-1} +b \right).$$
Last, we have $\beta'(t)= -\frac{1}{2} \alpha \Tr[V'(t)V(t)^{-1}]-\frac{1}{2} \alpha\Tr[b]$ and we obtain that
$$\beta(t)=-\frac{1}{2} \alpha \log(\det[V(t)])  -\frac{1}{2}\alpha  \Tr[b]t ,$$
since $\frac{d \det[V(t)]}{dt}=\det[V(t)] \Tr[V'(t)V(t)^{-1}]$.

It remains to show that we indeed have~\eqref{laplace_Wis_jointe} for $v$ and $w$ satisfying~\eqref{dom_lap}. We define $\mathcal{E}_t=\frac{\exp \left(\beta(T-t)+\Tr[\gamma(T-t)X_t]+\Tr[- \frac{v}{2}R_t] \right)}{\exp \left(\beta(T)+\Tr[\gamma(T)x] \right)}$. By Itô's formula, we have
$$ \frac{d \mathcal{E}_t }{ \mathcal{E}_t }= \Tr[\gamma(T-t)(\sqrt{X_t}dW_t +dW_t\tp \sqrt{X_t})].$$ 
This is a positive local martingale and thus a supermartingale which gives $\E[\mathcal{E}_T] \leqslant1$, and we want to prove that this is a martingale. To do so, we use the argument presented by Rydberg in~\cite{Rydberg}. For $L>0$, we define 
$$\tau_L= \inf \{t\geqslant 0, \Tr[X_t] \geqslant L \},$$
and $\pi_L(x)=x \mathbf{1}_{\Tr[x]\leqslant L} +  \frac{L} {\Tr[x]} x \mathbf{1}_{\Tr[x]> L} $ for $x \in \posm$. We consider $(\mathcal{E}^L_t,t\in[0,T])$ the solution of 
$$ d\mathcal{E}^L_t=\mathcal{E}^L_t \Tr[\gamma(T-t)(\sqrt{ \pi_L(X_t)}dW_t +dW_t\tp \sqrt{\pi_L(X_t)})], \  \mathcal{E}^L_0=1.$$ 
We clearly have $\E[\mathcal{E}^L_T]=1$. Besides, under $ \mathbb{P}^L$ given by  $\frac{d \mathbb{P}^L}{d \mathbb{P}}\big|_{\mathcal{F}_T}=\mathcal{E}^L_T$, the process 
$$dW^L_t=dW_t-2\sqrt{ \pi_L(X_t)}\gamma(T-t)dt,  \ t\in [0,T], $$
is a matrix Brownian motion. Since $\mathcal{E}_t=\mathcal{E}_t^L$ for $t\leqslant \tau_L$, we have $\E[\mathcal{E}_T]=\E[\mathcal{E}^L_T \mathbf{1}_{\tau_L>T}]+ \E[\mathcal{E}_T \mathbf{1}_{\tau_L \leqslant  T}]$. By Lebesgue's theorem, we get $\E[\mathcal{E}_T \mathbf{1}_{\tau_L \leqslant T}] \underset{L \rightarrow +\infty}{\rightarrow }0 $. On the other hand, $\E[\mathcal{E}^L_T \mathbf{1}_{\tau_L>T}]=\mathbb{P}^L(\tau_L>T)$. Let us consider the Wishart process $\tilde{X}$ starting from~$x$ such that
$$ d \tilde{X}_t= \left[\alpha I_d+ (b+2\gamma(T-t)) \tilde{X}_t+\tilde{X}_t(b+2 \gamma(T-t)) \right]dt+ \sqrt{\tilde{X}}_tdW_t+dW_t \tp \sqrt{\tilde{X}}_t.$$
We also define $\tilde{\tau}_L= \inf\{t \in [ 0, T], \Tr[\tilde{X}_t] \geqslant L \}$ with convention $\inf \emptyset= + \infty$. The process $\tilde{X}$ solves the same SDE on $[0,\tilde{\tau}_L \wedge T]$ under $\mathbb{P}$ as $X$ on $[0,\tau_L \wedge T]$ under $\mathbb{P}^L$. We therefore have 
$$ \mathbb{P}^L(\tau_L>T) = \mathbb{P}( \tilde{\tau_L}>T)  \underset{L \rightarrow +\infty}{\rightarrow }1 ,$$ 
which finally gives  $\E[\mathcal{E}_T] = 1$.
\end{proof}

\begin{coro}\label{cor_wis_gen}
Let $Y\sim WIS_d(y,\alpha,b,a)$ be a Wishart process with parameters $\alpha\geb d-1$, $y\in \posm$, $a,b\in \genm$ satisfying \begin{equation}\label{cond_a_b}
b a\tp  a= a\tp a b\tp  \text{ and } a \text{ invertible.} 
\end{equation}
Let $v,w\in \symm$ be such that 
\begin{equation}\label{cond_lapl_gen}
\exists m \in \symm, \ \frac{1}{2}awa\tp +m \in \posm \text{ and }\frac{ava\tp  }{2}-ab\tp a^{-1} m-m(a\tp )^{-1}ba\tp  -2m^2 \in \posm.
\end{equation}
Then, we have
\begin{align*}
&\E \left[\exp\left(-\frac{1}{2} \Tr\left[w Y_T + v \int_0^TY_s ds \right] \right)\right] \\
	&= \frac{\exp \big(- \frac{\alpha}{2} \Tr[b] t\big)}{\det[V_{v,w}(t)]^{ \frac{\alpha}{2}}} \exp \Big(- \frac{1}{2}\Tr \big[ (V'_{v,w}(t)V_{v,w}(t)^{-1}+ (a\tp )^{-1} b a\tp  ) (a\tp )^{-1} y a^{-1} \big] \Big),
\end{align*}
with $	V_{v,w}(t)= \left( \sum_{k=0}^\infty t^{2k+1} \frac{\tilde{v}^{k}}{(2k)!} \right)\tilde{w}+\sum_{k=0}^\infty t^{2k} \frac{\tilde{v}^k}{(2k)!} $ and 
\begin{align*}
	\tilde{v}=a v a\tp  +(a\tp )^{-1}b^2a\tp , \quad \mbox{and} \quad		\tilde{w}=a wa\tp  -(a\tp )^{-1}ba\tp .
\end{align*}
\end{coro}
\begin{proof}
  We know that $Y\underset{law}{=}a\tp  Xa $ with $x=(a\tp )^{-1}ya^{-1}$ and $X \sim WIS_d(x,\alpha,(a\tp )^{-1}ba\tp ,I_d)$, see e.g. equation~(13) in~\cite{AA2013}. We notice that $(a\tp )^{-1}ba\tp =ab\tp a^{-1} \iff b a\tp  a= a\tp a b\tp $ and thus $(a\tp )^{-1}ba\tp  \in \symm$.
We have $$\E \left[\exp\left(-\frac{1}{2} \Tr\left[w Y_T + v \int_0^TY_s ds \right] \right)\right]= \E \left[\exp\left(-\frac{1}{2} \Tr\left[awa\tp  X_T + ava\tp  \int_0^TX_s ds \right] \right)\right], $$
which gives the result by applying Proposition~\ref{prop:lap_Wish}.
\end{proof}
By setting $\tilde{m}=a^{-1}m(a\tp )^{-1}$, the condition~\eqref{cond_lapl_gen} is equivalent to the existence of $\tilde{m} \in \symm$, such that
\begin{equation}\label{cond_lap_finie}
 \frac{1}{2}w+\tilde{m} \in \posm  \text{ and }\frac{v }{2}- b\tp   \tilde{m} - \tilde{m} b  -2 \tilde{m} a\tp  a\tilde{m} \in \posm. 
\end{equation}
The case $m=0$ gives back the finiteness of the Laplace transform when $v,w \in \posm$. If we take $\tilde{m}=-w/2$, we get also the finiteness when 
\begin{equation}\label{cond_lapl_gen_utile}
v+b\tp w+wb-wa\tp aw\in \posm. 
\end{equation}
Another interesting choice is $m=-\frac{1}{2}(a\tp )^{-1}ba\tp $. We have $m\in\symm$ from~\eqref{cond_a_b}. This choice  gives the finiteness of the Laplace transform when $v+b\tp (a\tp a)^{-1}b \in \posm$ and $w-(a\tp a)^{-1}b \in \posm $. Let us note that $\tilde{v}=a(v+b\tp (a\tp a)^{-1}b)a\tp $ so that the first condition is the same as $\tilde{v} \in \posm$. Another interesting choice of~$m$ is given by the next remark. 
\begin{rmrk}\label{rk_setcv} Proposition~\ref{prop:lap_Wish} extends the result of  Gnoatto and Grasselli~\cite{GnoattoGrasselli} to $\alpha\geb d-1$, and the sufficient condition~\eqref{cond_lapl_gen} that ensures the finiteness of the Laplace transform is also less restrictive, which is crucial in our study especially in the nonergodic case. In particular, it does not assume a priori that $v+b\tp (a\tp a)^{-1}b \in \posm$. We can recover the result of~\cite{GnoattoGrasselli} as follows. Let us assume $v+b\tp (a\tp a)^{-1}b \in \posm$ and take $m=-\frac{ (a\tp )^{-1}ba\tp }{2}+\frac{1}{2}\sqrt{a(v+b\tp (a\tp a)^{-1}b)a\tp }$. We have $m\in \symm$ from~\eqref{cond_a_b} and it satisfies $\frac{ava\tp  }{2}-ab\tp a^{-1} m-m(a\tp )^{-1}ba\tp  -2m^2=0 \in \posm$. Therefore, \eqref{cond_lapl_gen} holds if
$$ w - (a\tp a)^{-1}b + a^{-1} \sqrt{a(v+b\tp (a\tp a)^{-1}b)a\tp } (a\tp )^{-1}\in \posm.  $$
This is precisely the condition stated in~\cite{GnoattoGrasselli}. 
\end{rmrk}
\begin{rmrk}\label{rk_ext_Aff} It is possible to get similarly the Laplace transform of $\left(Y_T,\int_0^TY_sds \right)$ when $Y$ solves 
$$  dY_t=\left[\overline{\alpha} +bY_t+Y_tb\tp  \right]dt+ \sqrt{Y}_tdW_t a+a\tp  dW_t \tp \sqrt{Y}_t , \ Y_0=y\in \posm,$$
with $a,b$ satisfying~\eqref{cond_a_b} and $\overline{\alpha}-(d-1)a\tp a \in \posm$. Again, equation~(13) in~\cite{AA2013} gives $Y\underset{law}{=}a\tp  Xa $, where 
$$  dX_t=\left[\hat{\alpha} +\hat{b}X_t+X_t\hat{b}\tp  \right]dt+ \sqrt{X}_tdW_t +  dW_t \tp \sqrt{X}_t , \ X_0=x,$$
with $x=(a\tp )^{-1}ya^{-1}\in \symm$, $\hat{b}=(a\tp )^{-1}ba\tp \in \symm$ and $\hat{\alpha}=(a\tp )^{-1}\overline{\alpha} a^{-1}\in \symm$. Repeating the proof of Proposition~\ref{prop:lap_Wish}, we observe that the Riccati equation~\eqref{eq_gamma} and equation~\eqref{eq_delta} remain unchanged while~\eqref{eq_beta} is replaced by 
$$\beta'=\Tr[\hat{\alpha} \gamma]=-\frac{1}{2}\Tr[\hat{\alpha} V'(t)V(t)^{-1}]-\frac{1}{2}\Tr[\hat{\alpha} \hat{b}].$$
Therefore, we deduce that under the same condition~\eqref{cond_lapl_gen}, we have
 \begin{align*}
&\E \left[\exp\left(-\frac{1}{2} \Tr\left[w Y_T + v \int_0^TY_s ds \right] \right)\right] \\
	&=  \exp (\beta(T)) \exp \Big(- \frac{1}{2}\Tr \big[ (V'_{v,w}(t)V_{v,w}(t)^{-1}+ (a\tp )^{-1} b a\tp  ) (a\tp )^{-1} y a^{-1} \big] \Big),
\end{align*}
with $\beta(t)=-\frac{1}{2} \int_0^t\Tr[(a\tp )^{-1}\overline{\alpha} a^{-1} V'_{v,w}(s)V_{v,w}(s)^{-1}]ds-\frac{t}{2}\Tr[\overline{\alpha}(a\tp a)^{-1}b]$ and $V_{v,w}(t)$ defined as in Corollary~\ref{cor_wis_gen}. Thus, the formula is no longer totally explicit. In Gnoatto and Grasselli~\cite{GnoattoGrasselli},  the result is stated with $\Tr[(a\tp )^{-1}\overline{\alpha} a^{-1}\log(V_{v,w}(t))]$ instead of the first integral. However, this replacement does not seem clear to us unless $V'_{v,w}(s)$ and $V_{v,w}(s)$ commute for all $s\geb 0$ (this happens when the matrices $\tilde{v}$ and $\tilde{w}$ in~$V_{v,w}$ commute) or $\overline{\alpha}=\alpha a\tp a$ by using the trace cyclic theorem.
\end{rmrk}

\begin{coro}\label{cor_girsanov}
Let $Y\sim WIS_d(y,\alpha,b,a)$ be a Wishart process with parameters such that $b a\tp  a= a\tp a b\tp $ and $a$  invertible. Then,
$$\forall u \in \symm, \ \E\left[ \exp \left( \int_0^T \Tr[u \sqrt{Y_s} dW_s a ] ds - \frac12\int_0^T \Tr[ a uY_s u a\tp  ] ds \right)\right]=1 .$$
\end{coro}
\begin{proof}
We have $2 \int_0^T \Tr[u \sqrt{Y_s} dW_s a ] ds= \Tr[u(Y_T-y)]-\alpha T \Tr[ua\tp a] - \Tr\left[ (ub+b\tp u)  \int_0^TY_sds\right]$. We apply Corollary~\ref{cor_wis_gen} with $w= -u$ and $v=ub+b\tp u+u a\tp au$. Therefore, \eqref{cond_lapl_gen_utile} holds. We then have $\tilde{w}=-(aua\tp +(a\tp )^{-1}ba\tp )$ and $\tilde{v}=\tilde{w}^2$ and the result follows by simple calculations. 
\end{proof}

\subsection{Study of the MLE of $b$ with the Laplace transform}\label{subsec_lapl_MLE}

We consider $\epsilon:\R_+ \rightarrow \R_+^*$ a (deterministic) decreasing function such that $\lim_{t\rightarrow  + \infty} \epsilon_t =0$. From the definition of the MLE of~$b$~\eqref{MLE_bseul}, we get that
$$\frac{1}{\epsilon_T}(\hat{b}_T-b)=\mathcal{L}^{-1}_{\epsilon_T^2 R_T}(\epsilon_T[X_T-x-\alpha T I_d-bR_T-R_Tb]).$$
Thus, we want to calculate the Laplace transform of $(\epsilon_T[X_T-x-\alpha T I_d-bR_T-R_Tb], \epsilon_T^2 R_T)$ in order to study the convergence of $\frac{1}{\epsilon_T}(\hat{b}_T-b)$. For $\lambda_1,\lambda_2 \in \symm$, we define
\begin{align}
\label{eq:Laplace_lim_UT_RT} 
&\mathbf {\mathcal{E}}(T,\lambda_1,\lambda_2) : =\E_\theta \Big[  \exp \Big(  -{ \epsilon_T  } \Tr [\lambda_2(X_T-x-\alpha T I_d-bR_T-R_Tb) ]- {\epsilon_T^2 } \Tr [\lambda_1 R_T]\Big) \Big] \\
&=  \exp \left(  { \epsilon_T  } \Tr[\lambda_2 (x+\alpha T I_d)] \right) \E_\theta \Big[  \exp \Big(  - \Tr [ \epsilon_T\lambda_2X_T]-  \Tr [(\epsilon_T^2\lambda_1 -\epsilon_T (\lambda_2b +b \lambda_2)) R_T]\Big) \Big].  
\end{align}
We now consider $\lambda_1,\lambda_2 \in \symm$ such that
\begin{align}\label{condition_sur_lambdas}
\lambda_1-2 \lambda_2^2 \in \dpos.
\end{align}
We define 
\begin{align}
        v_T= 2\lambda_1 \epsilon_T^2 - 2(b\lambda_2+\lambda_2b) \epsilon_T, \quad \tilde{v}_T=v_T+b^2, \quad w_T=2\lambda_2 \epsilon_T, \quad \tilde{w}_T=w_T-b,
\end{align}
and have $v_T+bw_T+w_Tb-w_T^2=\epsilon_T^2(2\lambda_1-4\lambda_2^2) \in \dpos$. Thus, by applying Proposition~\ref{prop:lap_Wish} with $m=-\epsilon_T \lambda_2$, we get that $\mathbf {\mathcal{E}}(T,\lambda_1,\lambda_2)$ is finite and given by
\begin{align}
 {\mathcal{E}}(T,\lambda_1,\lambda_2) 	=& \frac{\exp \big(- \frac{\alpha}{2} \Tr[b] T\big)}{\det[V_{v_T,w_T}(T)]^{ \frac{\alpha}{2}}} \exp \Big(- \frac{1}{2}\Tr \big[ (V'_{v_T,w_T}(T)V_{v_T,w_T}(T)^{-1}+b  )x  \big]  \Big) \nonumber \\
& \times \exp  \Big(  { \epsilon_T  } \Tr \big[\lambda_2(x+\alpha T I_d  ) \big] \Big)
\label{eq:loi_limite_estim_b}
\end{align}

with 
\begin{align*}
    	V_{v_T,w_T}(T)=& (\sqrt{\tilde{v}_T})^{-1}\sinh(\sqrt{\tilde{v}_T} T)\tilde{w}_T+\cosh(\sqrt{\tilde{v_T}} T)\\
    	 V'_{v_T,w_T}(T)=&  \cosh(\sqrt{\tilde{v}_T} T)\tilde{w}_T+\sinh(\sqrt{\tilde{v}_T} T)\sqrt{\tilde{v}_T}.
\end{align*}
 Besides, we have $\tilde{v}_T=(b-2 \epsilon_T \lambda_2)^2 +  \epsilon_T^2(2\lambda_1-4\lambda_2^2) \in \dpos$. 

When $-b \in \dpos$ and $\epsilon_T=1/\sqrt{T}$, we can make  explicit calculations and get $$\lim_{T\rightarrow + \infty} {\mathcal{E}}(T,\lambda_1,\lambda_2)=\exp(-\Tr[\lambda_1\rbi]-\Tr[2 \lambda_2^2 \rbi]),$$ which gives another mean to prove Theorem~\ref{thm_bseul_ergo}. Here, we prove Theorem~\ref{theo_b0_pos}.

\begin{proof}[Proof of Theorem~\ref{theo_b0_pos}] Here, we focus on the case $b=b_0 I_d$ with $b_0>0$ and set $\epsilon_T=e^{-b_0T}$.  Since the square root function is analytic on the set of positive definite matrices (see e.g. \cite{Rogers_Williams_2000_Vol2}, p.~134) we get that
$$\sqrt{\tilde{v}_T}=b_0 I_d - 2 \epsilon_T \lambda_2 + \frac{\epsilon_T^2}{b_0}(\lambda_1-2 \lambda_2^2) +O(\epsilon_T^3),$$
since the squares of each sides coincides up to a $O(\epsilon_T^3)$ term. We observe that $\tilde{w}_T=2 \epsilon_T \lambda_2-b_0I_d$, and thus $\sqrt{\tilde{v}_T}+\tilde{w}_T=\frac{\epsilon_T^2}{b_0}(\lambda_1-2 \lambda_2^2) +O(\epsilon_T^3).$

We now write
\begin{align*}
V_{v_T,w_T}(T)=& (\sqrt{\tilde{v}_T})^{-1}\left[ \frac{1}{2} \exp(\sqrt{\tilde{v}_T} T) ( \sqrt{\tilde{v}_T}+ \tilde{w}_T)  + \frac{1}{2} \exp(-\sqrt{\tilde{v}_T} T) ( \sqrt{\tilde{v}_T}- \tilde{w}_T) \right]\\
    	 V'_{v_T,w_T}(T)=&  \frac{1}{2} \exp(\sqrt{\tilde{v}_T} T) ( \sqrt{\tilde{v}_T}+ \tilde{w}_T)  + \frac{1}{2} \exp(-\sqrt{\tilde{v}_T} T) (  \tilde{w}_T - \sqrt{\tilde{v}_T}).
\end{align*}
Since $\epsilon_T \exp(\sqrt{\tilde{v}_T} T) \underset{T\rightarrow + \infty}{\rightarrow } I_d$, we get 
$ \frac{1}{\epsilon_T} V_{v_T,w_T}(T)\underset{T\rightarrow + \infty}{\rightarrow }  \frac{1}{ b_0}\left[ \frac{1}{2 b_0}(\lambda_1-2 \lambda_2^2)+b_0I_d \right]$ and $\frac{1}{\epsilon_T} V'_{v_T,w_T}(T) \underset{T\rightarrow + \infty}{\rightarrow }  \frac{1}{2 b_0}(\lambda_1-2 \lambda_2^2)-b_0I_d  $. This yields to
$$V'_{v_T,w_T}(T)V_{v_T,w_T}(T)^{-1}+b_0 I_d  \underset{T\rightarrow + \infty}{\longrightarrow }  (\lambda_1-2\lambda_2^2) \left(\frac{1}{2 b_0}(\lambda_1-2\lambda_2^2)+b_0I_d \right)^{-1}. $$
We also have $\frac{\exp \big(- \frac{\alpha}{2} \Tr[b_0 I_d] T\big)}{\det[V_{v_T,w_T}(T)]^{ \frac{\alpha}{2}}}=\frac{1}{\det[ \epsilon_T^{-1} V_{v_T,w_T}(T)]^{ \frac{\alpha}{2}}}\underset{T\rightarrow + \infty}{\rightarrow } \frac{1}{\det\left[ \frac{1}{ b_0}\left[ \frac{1}{2 b_0}(\lambda_1-2 \lambda_2^2)+b_0I_d \right] \right]}$, and therefore

\begin{equation}\label{lap_b0_pos}\lim_{T\rightarrow + \infty} {\mathcal{E}}(T,\lambda_1,\lambda_2)=\frac{\exp \left( -  \frac{1}{2b_0}\Tr\left[ (\lambda_1-2\lambda_2^2) \left(\frac{1}{2 b_0^2}(\lambda_1-2\lambda_2^2)+I_d \right)^{-1}x \right]\right)}{\det \left[  \frac{1}{2 b_0^2}(\lambda_1-2 \lambda_2^2)+I_d  \right]} .
\end{equation}

We now want to identify the limit.  We know that $X \sim WIS_d\left(\frac{x}{2b_0},\alpha,0,I_d;\frac{1}{4b_0^2} \right) $  has the following Laplace transform
$$u \in \posm, \E[\exp(-\Tr[uX])]= \frac{\exp\left(-\Tr\left[u (I_d+\frac{1}{2b_0^2}u)^{-1} \frac{x}{2b_0} \right]\right)}{\det[I_d+\frac{1}{2b_0^2}u]}.$$
Let $\tilde{\mathbf{G}}$ denote a $d$-square matrix independent from~$X$, whose entries are independent and follow a standard Normal distribution. By Lemma~\ref{lemme_matgauss}, we have
\begin{align*}
\E[\exp(-\Tr[\lambda_1 X + \lambda_2(\sqrt{X} \tilde{\mathbf{G}}+\tilde{\mathbf{G}} \sqrt{X})])]&=\E[\exp(-\Tr[(\lambda_1-2\lambda_2^2) X])].
\end{align*}
Thus, \eqref{lap_b0_pos} shows the convergence in law of $\left( \epsilon_T (X_T-x-\alpha T I_d-bR_T-R_Tb),  \epsilon_T^2  R_T \right)$ to $\left( X, \sqrt{X} \tilde{\mathbf{G}}+\tilde{\mathbf{G}} \sqrt{X} \right)$ under $\Px_\theta$, which gives the claim of  Theorem~\ref{theo_b0_pos}. 
\end{proof}

\section{Numerical Study}\label{Sec_Num_Study}

In this section, we test the convergence of the MLE given by \eqref{eq:estimator_couple} and \eqref{MLE_bseul}. To do so, we consider a given large value of~$T$ and simulate the Wishart process exactly on the regular time grid $t_i=\frac{iT}{N}$, $i=0, \cdots, N$. This can be done by using the method presented in Ahdida and Alfonsi~\cite{AA2013}, see also Alfonsi~\cite{Alfonsibook}. We take $N$ sufficiently large and approximate the integrals $R_T$ and $Q_T^{-1}$ applying the trapezoidal rule along this time grid. Thus, we will use the estimator with the exact value of~$X_T$ and these approximated values of $R_T$ and $Q_T^{-1}$.

This section has three goals. First, we check numerically the convergence results that we have obtained. Second, we investigate numerically the convergence of the MLE in some nonergodic cases, where no theoretical result of convergence have been found. Last, we test the estimation of the parameters of a full Wishart process~\eqref{SDE_Intro}. To do so, we estimate first $a$ with the quadratic variation and then the parameters $\alpha$ and $b$ by using the MLE~\eqref{eq:estimator_couple} on the process $(a\tp)^{-1}X a^{-1}$.

\newsavebox{\smlmat}
\newsavebox{\smlmatrix}
\newsavebox{\smlmatx}

\subsection{Numerical validation of the convergence results}

Using the method mentioned above, we have checked the convergence results obtained in this paper. Namely, we sample $M=10000$ independent paths of~$X$ in order to draw an histogram of the properly rescaled value of $\hat{b}_{i,j}-b_{i,j}$ or $\hat{\alpha}-\alpha$. We do not reproduce all these graphics here, and present for example in Figure~\ref{fig:b_lambdaI}  an illustration of the convergence given by Theorem~\ref{theo_b0_pos}.

\savebox{\smlmat}{$\left(\begin{smallmatrix} 0.5 & 0.1 \\
0.1 & 0.3 \end{smallmatrix}\right)$}

\savebox{\smlmatrix}{$\left(\begin{smallmatrix} 0.3 & 0.1 \\
0.1 & 0.2 \end{smallmatrix}\right)$}

\begin{figure}[H]
\begin{centering}
        \begin{subfigure}[b]{0.45\textwidth}
                \includegraphics[width=\textwidth]{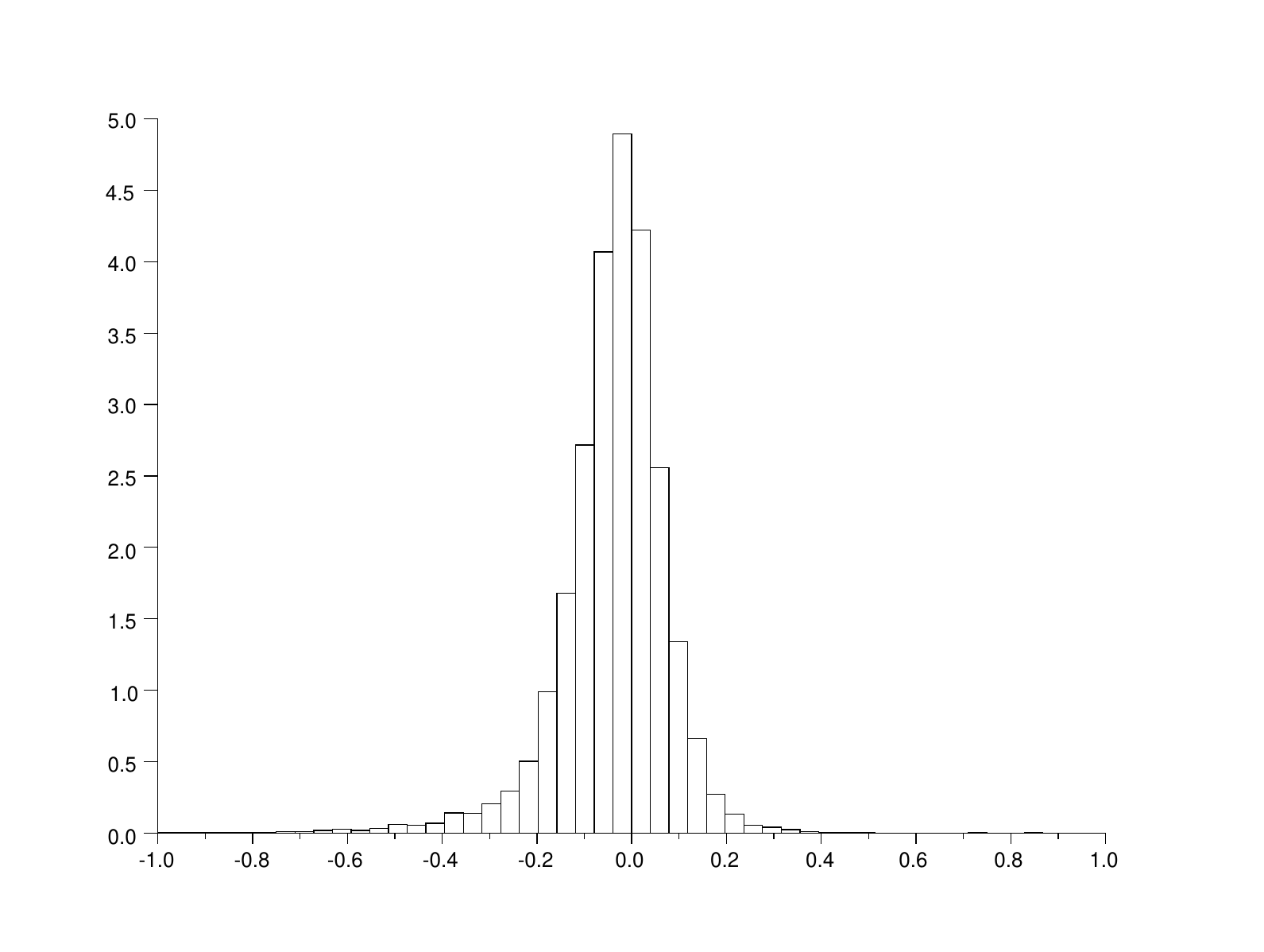}
                \caption{Limit law of $\exp(0.05T)(b-\hat{b}_T)_{1,1}$. }
                \label{fig:b_lambdaI_non_ergo_b_11}
        \end{subfigure}
        \begin{subfigure}[b]{0.45\textwidth}
                \includegraphics[width=\textwidth]{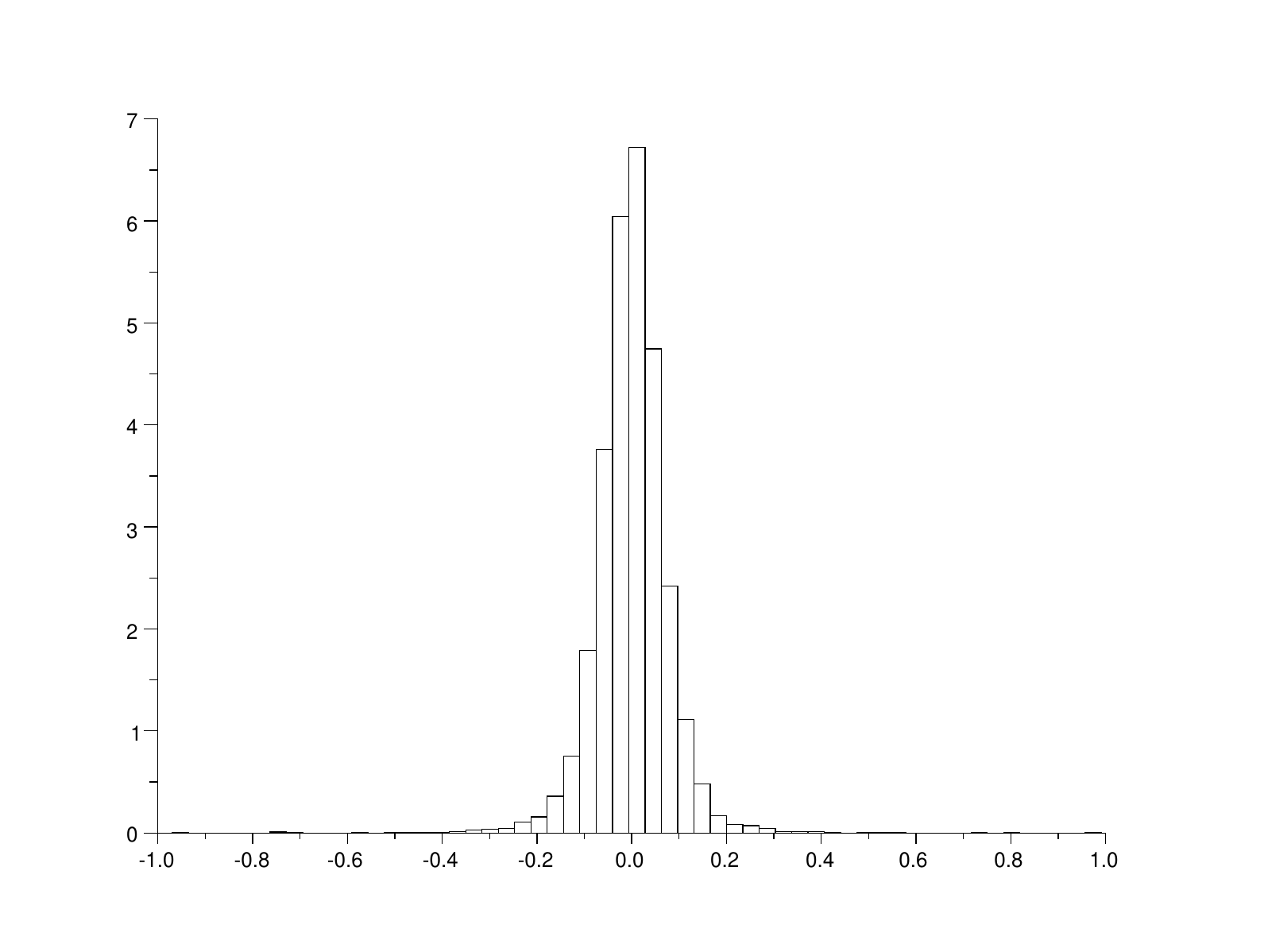}
                \caption{Limit law of $\exp(0.05T)(b-\hat{b}_T)_{1,2}$.}
                \label{fig:b_antidiag_lambdaI}
        \end{subfigure}     

        \caption{Asymptotic law of the error for the estimation of $\theta=b$ with for: $x=$~\usebox{\smlmat}, $T=100$, $N=10000$, $\alpha=4.5$ and $b=0.05I_d$.} \label{fig:b_lambdaI}
\end{centering}
\end{figure}

\subsection{Experimental convergence in a nonergodic case}

In this paragraph, we try to guess the asymptotic behavior of the MLE in a nonergodic case, where no theoretical convergence result is known. Namely,
 we observe in Figure~\ref{fig:b_lambdaI_presum} the asymptotic estimation error, when $b=\diag(0.1,0.005)$ is diagonal with positive and distinct terms on its diagonal and when we use the estimator~\eqref{MLE_bseul}. As one might have guess, the convergence of the diagonal terms seems to be with an exponential rate,  with the exponential speed corresponding to its value. Namely, $\hat{b}_{11}$ seems to converge to $b_{11}$ with a speed of $\exp(0.1T)$ while $\hat{b}_{22}$ seems to converge to $b_{22}$ with a speed of $\exp(0.005T)$.  More interesting is the antidiagonal term. One could have imagine that the convergence rate is the slowest of these two rates. Instead, on our experiment, the convergence of $\hat{b}_{12}$ towards $b_{12}$ seems to happen with the rate $\exp(0.1T)$. We have observed the same behaviour for other parameter values. Of course, it would be hasty to draw a global conclusion from few particular experiments. However, it is interesting to note that these numerical tests are a way to guess or check the convergence rate of the MLE.

\savebox{\smlmat}{$\left(\begin{smallmatrix} 0.5 & 0.1 \\
0.1 & 0.3 \end{smallmatrix}\right)$}

\savebox{\smlmatrix}{$\left(\begin{smallmatrix} 0.3 & 0.1 \\
0.1 & 0.2 \end{smallmatrix}\right)$}

\begin{figure}[H]
\begin{centering}   
        \begin{subfigure}[b]{0.45\textwidth}
                \includegraphics[width=\textwidth]{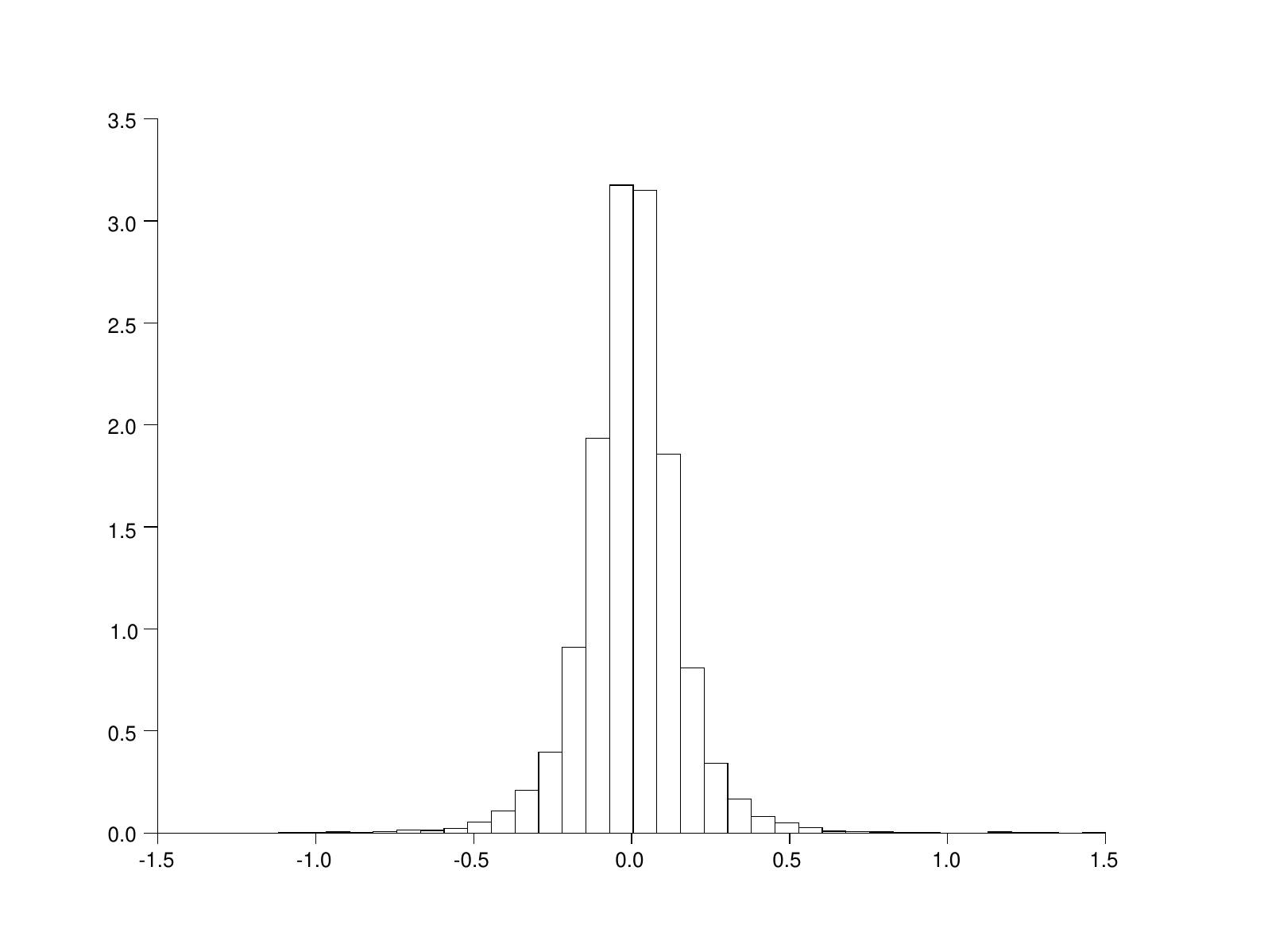}
                \caption{Limit law of $\exp(0.1T)(b-\hat{b}_T)_{1,1}$. }
        \end{subfigure}
         \begin{subfigure}[b]{0.45\textwidth}
                \includegraphics[width=\textwidth]{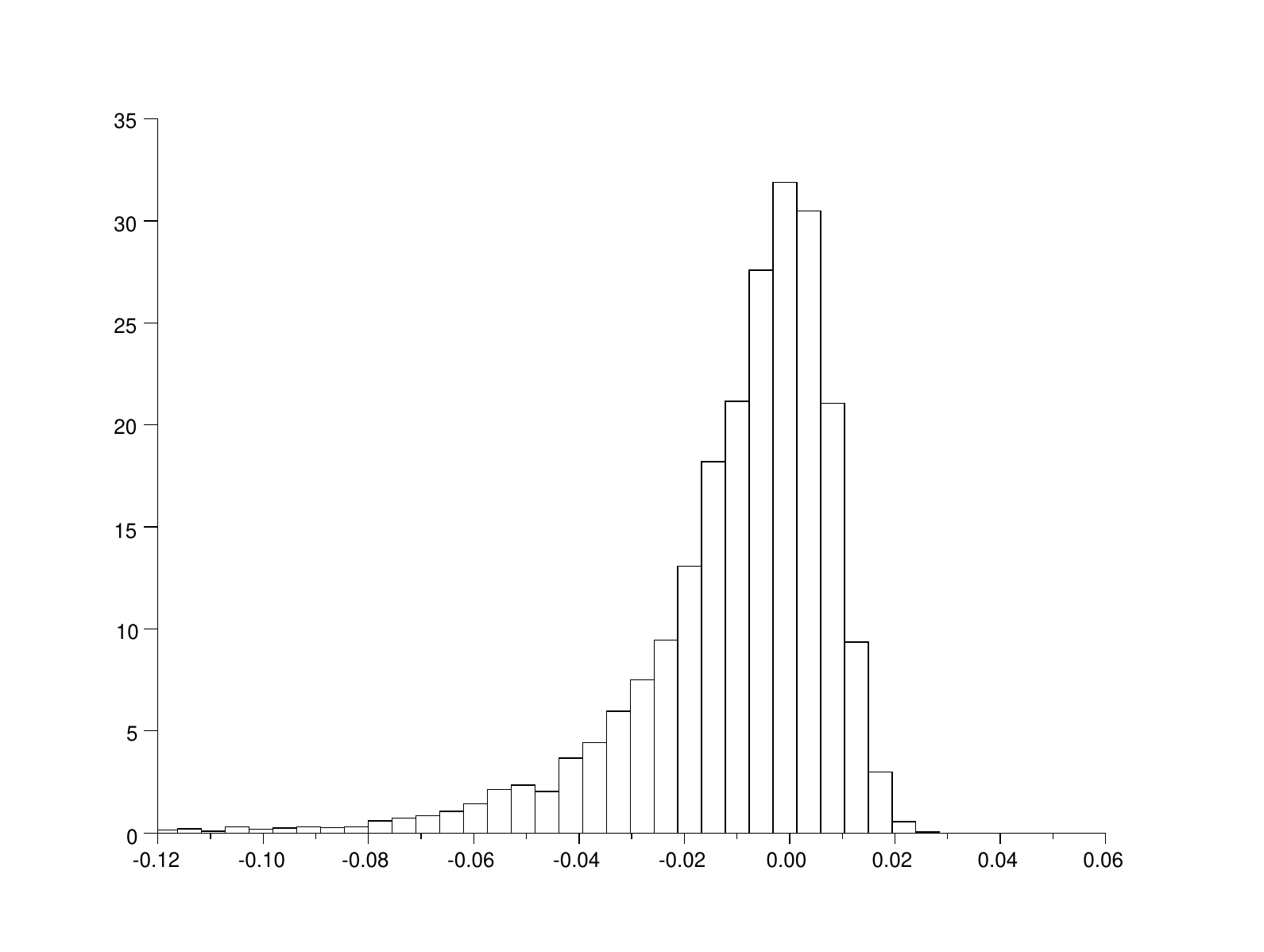}
                \caption{Limit law of $\exp(0.005T)(b-\hat{b}_T)_{2,2}$. }
        \end{subfigure}
        \begin{subfigure}[b]{0.5\textwidth}
                \includegraphics[width=\textwidth]{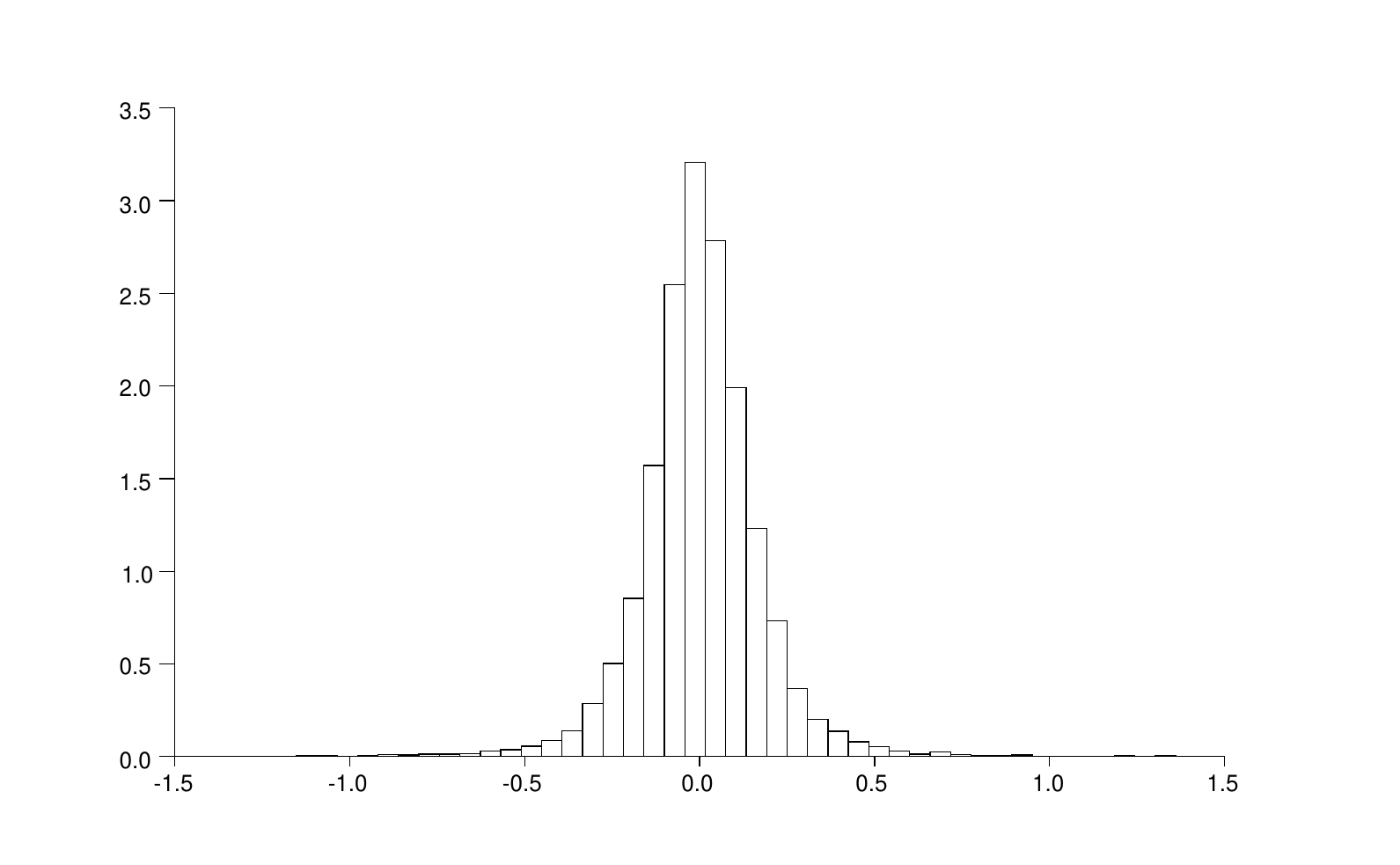}
                \caption{Limit law of $\exp(0.1T)(b-\hat{b}_T)_{1,2}$.}
        \end{subfigure}     

        \caption{Asymptotic law of the error for the estimation of $\theta=b$ with $x=$\usebox{\smlmatrix}, $T=100$, $N=10000$, $\alpha=3.5$ and $b=\diag(0.1,0.005)$.}\label{fig:b_lambdaI_presum}
\end{centering}
\end{figure}

\subsection{Estimation of the whole Wishart process}\label{est_allwish_num}
In this last part of the numerical study, we perform the estimation of all the parameters of the Wishart process~\eqref{SDE_Intro}. We consider a case where $a$ is upper triangular and $(a\tp)^{-1}b a\tp$ is symmetric. We proceed as follows. First, we sample exactly a discrete path $(X_{iT/N},0\leb i\leb N )$. Then, we estimate the matrix $a\tp a$ by using~\eqref{eq:estim_crochet_a}, where the quadratic variations are replaced by their classical approximations and the integrals are replaced by the trapezoidal rule. By a Cholesky decomposition we get then an estimator $\hat{a}$ of~$a$. Then, we use the MLE~\eqref{eq:estimator_couple} on the path $((\hat{a}\tp)^{-1} X_{iT/N}  \hat{a}\tp,0\leb i\leb N)$. This gives an estimator of $\alpha$ and $(a\tp)^{-1}b a\tp$, and therefore an estimator of $b$. As a comparison, we also calculate similarly the estimator of $\alpha$ and $b$ when $a$ is known and has not to be estimated. To draw histograms or calculate empirical expectations, we run $M=10000$ independent paths of~$X$.

We consider a sufficiently large value of~$T$ and are interested in looking at the convergence with respect to~$N$.  First, we plot the the error on the estimator of $a$ with respect to the number of time step in Log-Log scale. We observe that the convergence to zero takes place with experimental rate close to~$1/2$. This is in line with the general results on the estimation of the diffusion coefficient, see Dohnal~\cite{Dohnal} and Genon-Catalot and Jacod~\cite{GCJ}.  Then, we focus on the influence of the discretization and the unknown parameter $a$ on the convergence of the MLE of $b$ and $\alpha$. In Table~\ref{table:biais_a}, we give in function of $N$ the Mean Squared Error $\MSE(\widehat{\theta}^N \vert \theta)= \mathbb{E}[ \vert \widehat{\theta}^N- \theta \vert^2 ]$ of the estimator $\widehat{\theta}^N$, with $\theta=(b,\alpha)$. It is estimated with the empirical expectation. First, we observe that the convergence of the estimator of~$\alpha$ is roughly the same whether we know $a$ or not. This is expected since the estimation of $\alpha$ does not depend on the estimation of $a$.  Instead, the bias on $b$ is much higher when $a$ is estimated than when $a$ is known. However, it decreases also faster at an experimental order of $0.7$ while the bias when $a$ is known decreases at an experimental order of $0.45$. This latter rate is in line with the rate of $1/2$ obtained in dimension~1 by Ben Alaya and Kebaier~\cite{Alaya_Kebaier_2013}. In our case, it seems that the influence of the estimation of $a$ vanishes around $N=5000$. Last, we have plotted in Figure \ref{fig:b_pos_a_inc} the limit law of the estimator $\sqrt{T}(\widehat{\theta}^N-\theta)$ with $N=10000$. 

This short numerical study shows that the estimator obtained by discretizing the continuous time estimator is efficient in practice. Of course, it would be nice to obtain general convergence results in function of $T$ and $N$, but we leave this for  further research. 

\savebox{\smlmat}{$\left(\begin{smallmatrix} -1 & 0.2 \\
2 & -2 \end{smallmatrix}\right)$}

\savebox{\smlmatrix}{$\left(\begin{smallmatrix} 1 & 1 \\
0 & 2 \end{smallmatrix}\right)$}

\savebox{\smlmatx}{$\left(\begin{smallmatrix} 0.8 & 0.5 \\
0.5 & 1 \end{smallmatrix}\right)$}

\begin{figure}[H]
\begin{centering}  
                \includegraphics[width=0.5\textwidth]{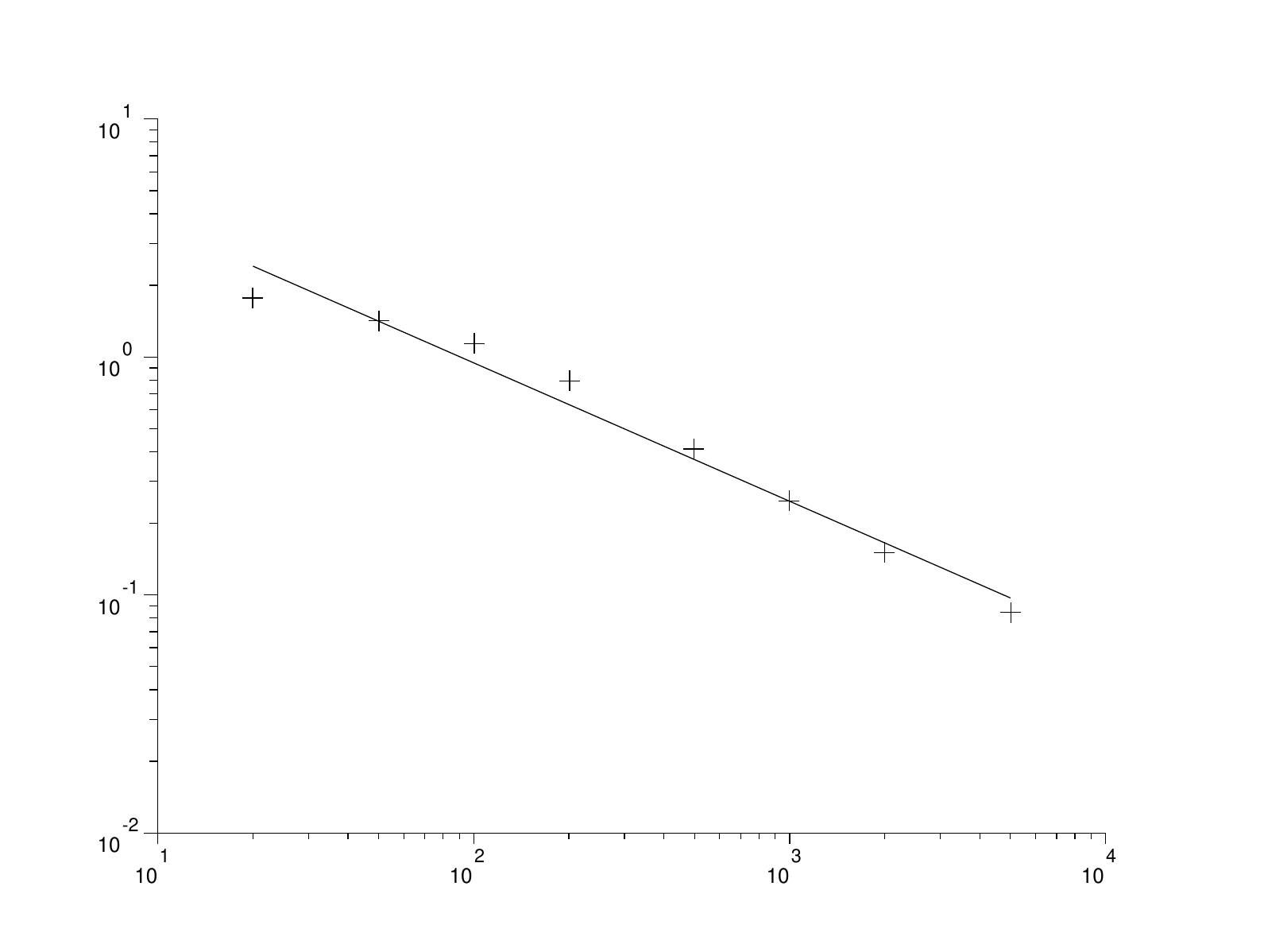}

        \caption{Log-Log representation of the empirical expectation of $\mathbb{E}[\Tr[(a-\widehat{a}^N)^2]  ]^{1/2}$ for $x=$\usebox{\smlmatx}, $T=100$, $a= $\usebox{\smlmatrix}, $\alpha=4.5$, $b=$\usebox{\smlmat}, where the line is the simple linear regression i.e. $\log(\mathbb{E}[\Tr[(a-\widehat{a}^N)^2]  ]^{1/2}) \approx 2.62-0.58 \log(N)$. }\label{fig:Log_Log_a}
\end{centering}
\end{figure}

\begin{table}
\begin{centering}
\begin{tabular}{||l|c||c|c|c|c|c|c|c|r||} 
   \hline
    \hline
    \multicolumn{2}{||c||}{Number of time steps} & 20 & 50 & 100 & 200 & 500 & 1000 & 2000 & 5000 \\
    \hline
    \hline
    \multicolumn{2}{||c||}{$\mathbb{E}[\Tr[(a-\widehat{a}^N)^2]  ]^{1/2}$} & 1.7671 &  1.4311 & 1.1487 & 0.7913 & 0.4107 & 0.2472 & 0.1514 & 0.0846 \\
    \hline
  &  $ \widehat{a}=a$  & 0.0745 & 0.0338 & 0.0181 & 0.0115 & 0.0082 & 0.0069 & 0.0061 & 0.0058 \\
    \cline{2-10} 
 $ \MSE (\widehat{b}^N_{1,1} \vert b_{1,1})$        & $ \widehat{a}=\widehat{a}^N$  & 0.7636 & 0.5266 & 0.3489 & 0.1891 & 0.0624 & 0.0273 & 0.0142 & 0.0085 \\
    \hline
    \hline
  &  $ \widehat{a}=a$  & 0.2554 & 0.1310 & 0.0664 & 0.0372 & 0.0231 & 0.0176 & 0.0153 & 0.0139\\
    \cline{2-10} 
 $\MSE (\widehat{b}^N_{2,2} \vert b_{2,2})$        & $ \widehat{a}=\widehat{a}^N$  & 3.4085 & 2.8722 & 2.1159 & 1.1995 & 0.3600 & 0.1264 & 0.0480 & 0.0201 \\
    \hline
    \hline
  &  $ \widehat{a}=a$  & 0.0075 & 0.0033 & 0.0017 & 0.0011 & 0.0008 & 0.0008 & 0.0007 & 0.0007 \\
    \cline{2-10} 
$\MSE (\widehat{b}^N_{1,2} \vert b_{1,2})$       & $ \widehat{a}=\widehat{a}^N$  & 0.0442 & 0.0568 & 0.0596 & 0.0352 & 0.0148 & 0.0075 & 0.0039 &  0.0019\\
    \hline
    \hline
  &  $ \widehat{a}=a$  & 0.8448 & 0.3579 & 0.1993 & 0.1151 & 0.0614 & 0.0416 & 0.0308 & 0.0230 \\
    \cline{2-10} 
$\MSE (\widehat{\alpha}^N \vert \alpha)$      & $ \widehat{a}=\widehat{a}^N$  & 0.8267 & 0.3496 & 0.1895 & 0.1095 & 0.0617 & 0.0410 & 0.0311 & 0.0234 \\
    \hline
    \hline
\end{tabular}
\caption{Mean Squared Error for the estimation of $\theta=(b,\alpha)$ with respect to $N$. Same parameters as Figure~\ref{fig:Log_Log_a}. 
}
\label{table:biais_a}
\end{centering}
\end{table}

\begin{figure}[H]
\begin{centering}
        \begin{subfigure}[b]{0.45\textwidth}
                \includegraphics[width=\textwidth]{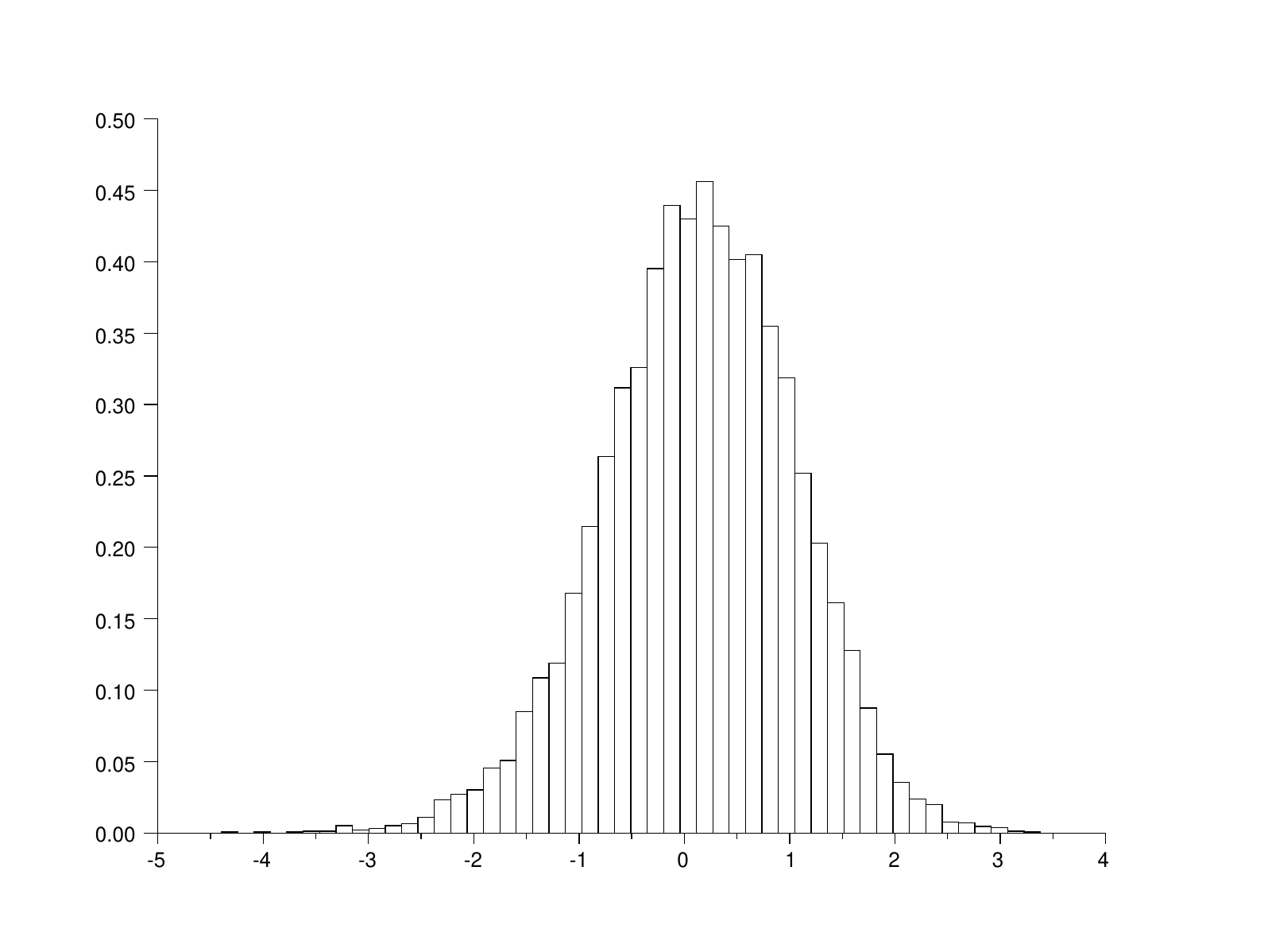}
                \caption{Limit law of $\sqrt{T}(b-\hat{b}^N_T)_{1,1}$.}
	\end{subfigure}
        \begin{subfigure}[b]{0.5\textwidth}
                \includegraphics[width=\textwidth]{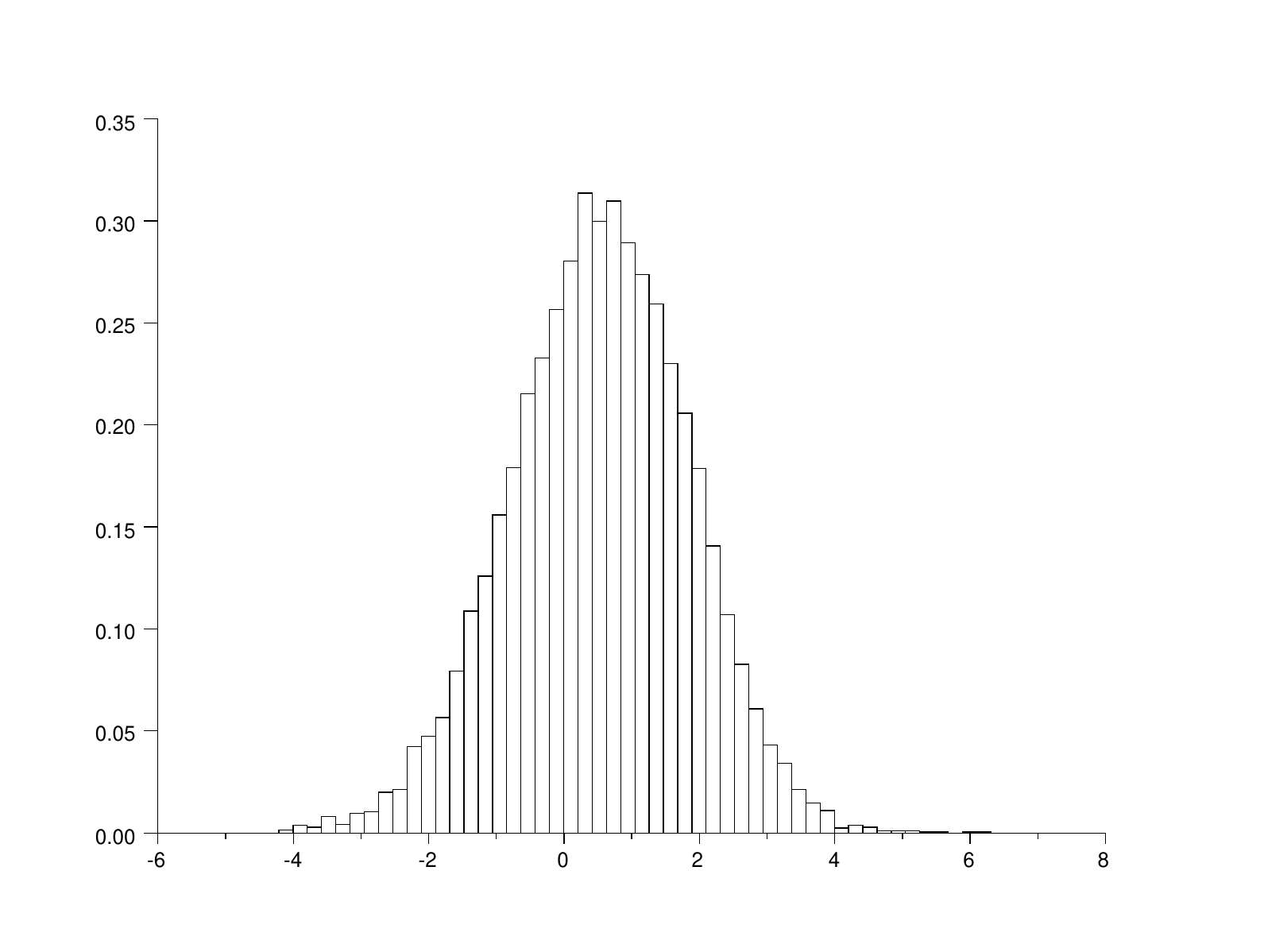}
                \caption{Limit law of $\sqrt{T}(b-\hat{b}^N_T)_{2,2}$.}
        \end{subfigure}
	\begin{subfigure}[b]{0.45\textwidth}
                \includegraphics[width=\textwidth]{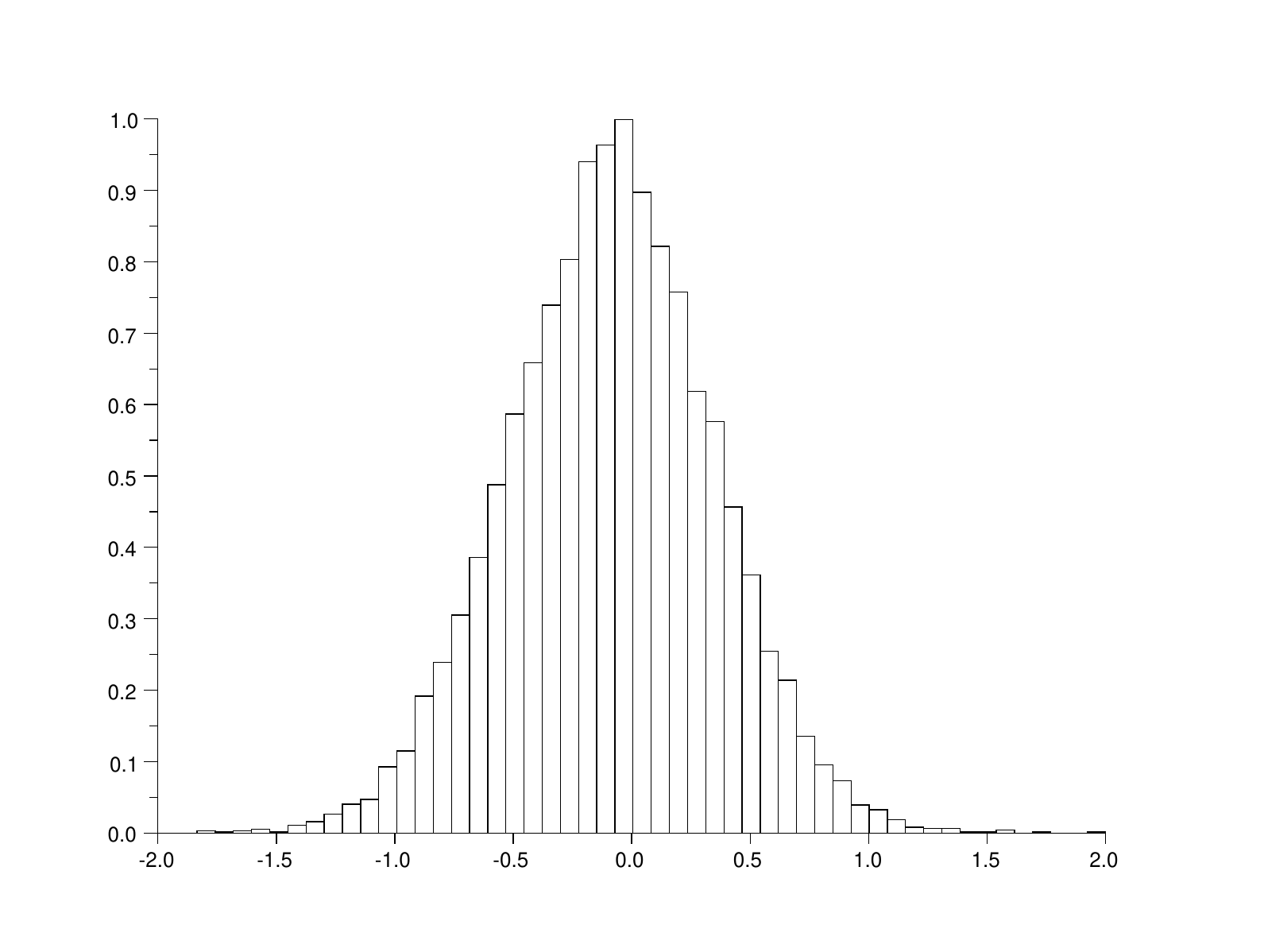}
                \caption{Limit law of $\sqrt{T}(b-\hat{b}^N_T)_{1,2}$.}
        \end{subfigure}
        \begin{subfigure}[b]{0.5\textwidth}
                \includegraphics[width=\textwidth]{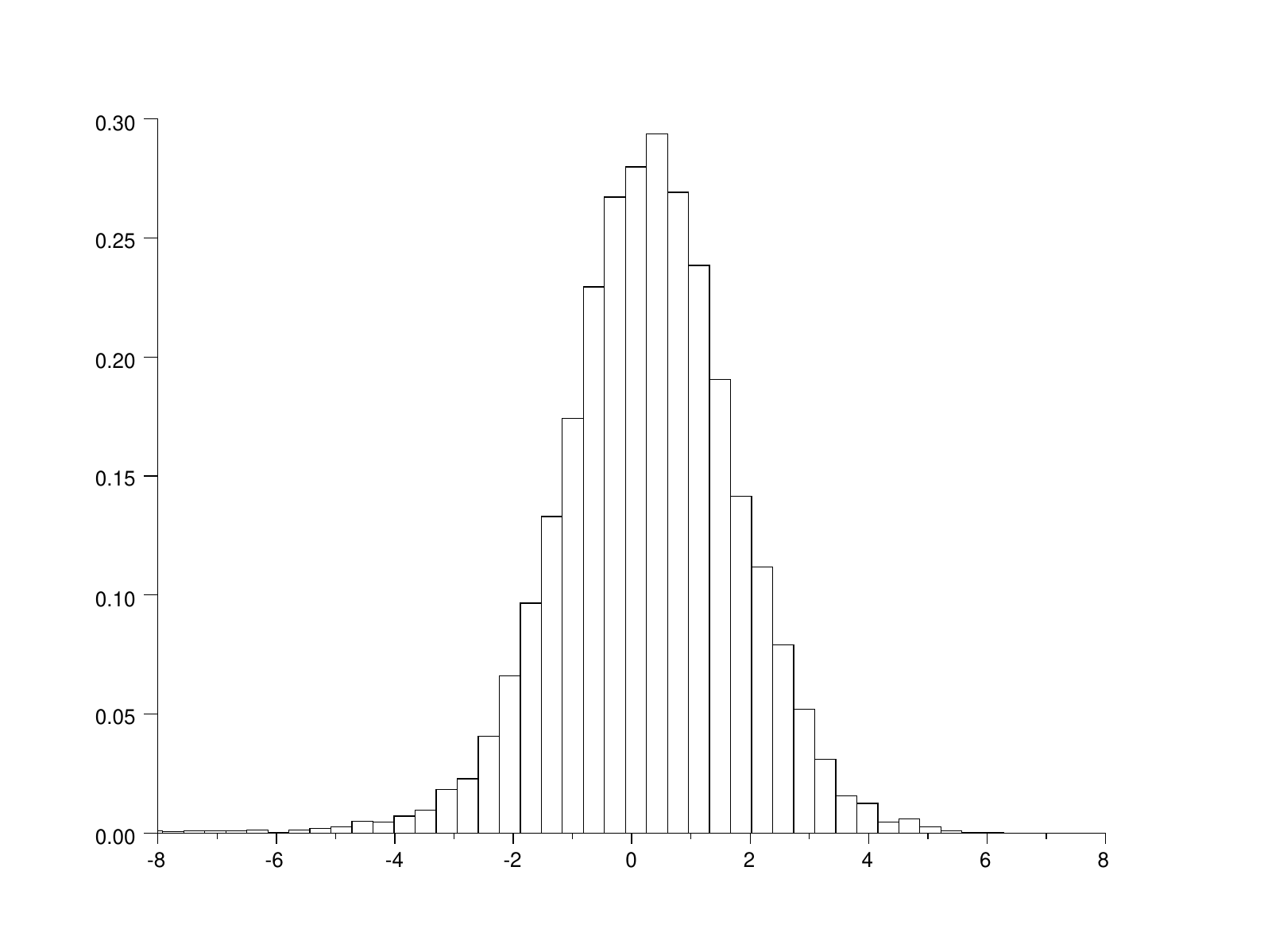}
                \caption{Limit law of $\sqrt{T}(\alpha-\hat{\alpha}^N_T)$.}
        \end{subfigure}  
        \caption{Asymptotic laws of the error for the estimation of $\theta=(b,\alpha)$ for $\widehat{a}=\widehat{a}^N$, $N=10000$, same parameters as Figure~\ref{fig:Log_Log_a}. 
        }
\label{fig:b_pos_a_inc} 
\end{centering}
\end{figure}

\newpage

\appendix 
\section{Proof of Proposition~\ref{prop_MLE_bgen}} \label{App_proof_mle_bgen}

We denote $b^s=(b+b\tp )/2$ (resp. $b^a=(b-b\tp )/2$) the symmetric (resp. antisymmetric) part of $b$. We have
\begin{align*}
\int_0^T \Tr [(\sqrt{X_s})^{-1}d W_s ] &=
\frac{1}{2}\log\left(\frac{\det[X_T]}{\det[x]} \right) - \Tr[b]T -\frac{1}{2}
\int_0^T(\alpha-1-d)\Tr[X_s^{-1}] ds,\\
\int_0^T \Tr[b^s\sqrt{X_s}dW_s]
&=    \frac{1}{2}\int_0^T \Tr[b^s(\sqrt{X_s}dW_s+dW_s\tp  \sqrt{X_s})] 
 \\
&=\frac{\Tr[b^sX_T]-{\Tr[b^sx]}}{2}-\frac{\alpha T}{2}{\Tr[b^s]}-\frac{1}{2}\int_0^T \Tr[b^s(bX_t+X_tb\tp)]dt.
\end{align*}
Thus, the only part to calculate is $\E\left[\exp(\int_0^T \Tr[-b^a\sqrt{X_s}dW_s])\bigg| \cF^X_T \right]$, and we set $M^a_t=\int_0^t\Tr[b^a\sqrt{X_s}dW_s]$. We now observe that $\langle \Tr[A_sdW_s],\Tr[B_sdW_s] \rangle=\Tr[A_sB_s\tp]ds$ and are looking for the process $\Gamma$ that takes values in~$\symm$ and minimizes
 $$\langle dM^a_t-\Tr[\Gamma_t(dX_t-(\alpha I_d+bX_t+X_t b\tp)dt)]  \rangle = \left\{ -\Tr[b^aX_tb^a]+2\Tr[\Gamma_t(X_tb^a-b^aX_t)]  +4\Tr[\Gamma_t^2X_t] \right\}dt.$$
We obtain that $2(X_tb^a-b^aX_t)+ 4(X_t \Gamma_t+ \Gamma_t X_t)=0$ and thus
$$\Gamma_t=\mathcal{L}_{X_t}^{-1}\left(\frac{1}{2} (b^aX_t-X_tb^a) \right). $$
It satisfies $\Tr[\Gamma_t(X_tb^a-b^aX_t)]=-2\Tr[\Gamma_t(\Gamma_t X_t+X_t\Gamma_t)]=-4\Tr[\Gamma_t^2 X_t]$. 
By construction, we have $\langle dM^a_t-\Tr[\Gamma_t(dX_t-(\alpha I_d+bX_t+X_t b\tp)dt)], \Tr[\tilde{\Gamma}(dX_t-(\alpha I_d+bX_t+X_t b\tp)dt)]  \rangle=0$ for any $\tilde{\Gamma} \in \symm$. Thus, there exists a Brownian motion $\beta$ independent of~$X$ such that $dM^a_t-\Tr[\Gamma_t(dX_t-(\alpha I_d+bX_t+X_t b\tp)dt)]=\sqrt{ -\Tr[b^aX_tb^a]-\Tr[\Gamma_t(b^aX_t-X_tb^a)]}d \beta_t$. In fact, both processes $(X_t,\int_0^t\sqrt{ -\Tr[b^aX_sb^a]-\Tr[\Gamma_t(b^aX_s-X_sb^a)]}d \beta_s)$ and   $$(X_t,M^a_t-\int_0^t\Tr[\Gamma_s(\sqrt{X_s} dW_s +dW_s \tp \sqrt{X_s})])$$ solve the same martingale problem for which uniqueness holds. Therefore, we have 
\begin{align*}
&\E\left[\exp\left(-\int_0^T \Tr[b^a\sqrt{X_t}dW_t]\right)\bigg| \cF^X_T \right]\\
=& \exp\left( \int_0^T\Tr\left[\Gamma_t((\alpha I_d+bX_t+X_t b\tp)dt-dX_t)\right] -  \frac{1}{2} \int_0^T\Tr[b^aX_tb^a] + \Tr[\Gamma_t(b^aX_t-X_tb^a)]dt\right) \\
=& \exp\left( - \int_0^T\Tr\left[\mathcal{L}_{X_t}^{-1}\left(\frac{1}{2} (b^aX_t-X_tb^a) \right) dX_t\right] - \int_0^T \frac{1}{2}\Tr[b^aX_tb^a]dt \right.\\
&\left. + \frac{1}{2} \int_0^T \Tr\left[\mathcal{L}_{X_t}^{-1}\left(\frac{1}{2} (b^aX_t-X_tb^a) \right) (b^aX_t-X_tb^a)\right]dt  +  \int_0^T\Tr[b^aX_tb^s]dt \right),
\end{align*}
since $\Tr\left[\mathcal{L}_{X_t}^{-1}\left(\frac{1}{2} (b^aX_t-X_tb^a) \right)\right]=\frac{1}{2}\Tr[X_t^{-1}(b^aX_t-X_tb^a)]=0$ by Lemma~\ref{lemme-linear} and $$\Tr[\Gamma_t(b^sX_t+X_tb^s)]=\Tr[b^s(\Gamma_tX_t+X_t\Gamma_t)]=\frac{1}{2}\Tr\left[b^s(b^aX_t-X_tb^a)\right]=\Tr[b^aX_tb^s].$$ 
Using~\eqref{def_likelihood} and the previous calculations, we obtain
\begin{multline*}
L_T^{\theta,\theta_0}=\exp\Bigl(\frac{\alpha-\alpha_0}{4}\log\left(\frac{\det[X_T]}{\det[x]}\right)+\frac{\Tr[b^sX_T]-{\Tr[b^sx]}}{2}-\frac{1}{2}\int_0^T \Tr[(b^s)^2X_s]ds
\\
- \int_0^T \Tr[b^aX_sb^s]ds-\frac{\alpha-\alpha_0}{4}\big( \frac{\alpha+\alpha_0}{2}-1-d \big)\int_0^T \Tr[X_s^{-1}]ds-\frac{\alpha T}{2}\Tr[b] \\ +\frac{1}{2} \int_0^T \Tr\left[\mathcal{L}_{X_t}^{-1}\left(b^aX_t-X_tb^a \right) dX_t \right]-\frac{1}{4}  \int_0^T \Tr\left[\mathcal{L}_{X_t}^{-1}\left(b^aX_t-X_tb^a \right)  (b^aX_t-X_tb^a) \right]dt  \Bigr). 
\end{multline*}
Last, we use $\mathcal{L}_{X_t}^{-1}\left(b^sX_t+X_tb^s \right)=b^s$ and $\Tr[\mathcal{L}_{X_t}^{-1}\left(b^aX_t-X_tb^a \right)(b^sX_t+X_tb^s)]=2\Tr[b^aX_tb^s]$ to obtain~\eqref{vs_gen_bgen}.

\section{Technical lemmas}

\begin{lemme}\label{lemme-linear}
For $X\in \dpos$ and $a \geqslant 0$, let ${\mathcal{L}}_{X,a}$ and $\mathcal{L}_X={\mathcal{L}}_{X,0}$ be the linear
applications defined by~\eqref{def_LXa} on $\mathcal{S}_d$. If $a\Tr[X^{-1}] \neq 1$, then ${\mathcal{L}}_{X,a}$ is invertible and we have $\Tr[\mathcal{L}^{-1}_{X,a}(Y)]=\frac{\Tr[X^{-1} Y]}{2(1-a\Tr[X^{-1}])}$. Besides,   the map $(X,Y,a)  \mapsto \mathcal{L}^{-1}_{X,a}(Y)$ is continuous on $\{(X,Y,a) \in \dpos \times \symm \times \R_+, \ a\Tr[X^{-1}] \neq 1 \}$.

\end{lemme}
\begin{proof} 
The invertibility of ${\mathcal{L}}_{X,a}$ is equivalent to its one-to-one property.
Since $X \in \mathcal{S}_d^{+,*}$, there exists an orthogonal matrix $O_X$ and a diagonal matrix $D_X$  with positive elements such that $X=
O_X D_X O_X\tp $. We get
\begin{align}
Y\in \ker({\mathcal{L}}_{X,a}) \iff& O_X D_X O_X\tp Y + YO_X D_X O_X\tp  =2a \Tr[Y] I_d \nonumber\\
 \iff  &  D_X(O\tp _X Y O_X)=2a \Tr[Y]I_d -(O_X\tp  Y O_X)D_X . \label{cond_invers}
\end{align} 
Since $D_X$ is diagonal, we obtain  for $1\leb i,k\leb d$, $\left(  (O\tp _X Y O_X)D_X \right)_{i,k} =(O\tp _X Y O_X)_{i,k} (D_X)_{k,k} $ and   $\left(  D_X(O\tp _X Y O_X) \right)_{i,k}= (D_X)_{i,i}(O\tp _X Y O_X)_{i,k} $. 
For $k\not= i$, \eqref{cond_invers} gives $(O\tp _X Y O_X)_{i,k}=0$.
For $k=i$, we get $(O\tp _X Y O_X)_{i,i} (D_X)_{i,i}=a\Tr[Y]$ and therefore 
\begin{align*}
\Tr[Y]=\Tr[O\tp _X Y O_X]=\Tr[Y]a\sum_{i=1}^d \frac{1}{(D_X)_{i,i}}=\Tr[Y]a\Tr[X^{-1}].
\end{align*}
Since $a\Tr[X^{-1}] \not= 1$, we obtain $\Tr[Y]=0$ and then $(O\tp _X Y O_X)_{i,i}=0$, which gives $Y=0$ and the invertibility of ${\mathcal{L}}_{X,a}$. 
 Let $c=\mathcal{L}_{X,a}^{-1}(Y)$. We have  $c+X^{-1}cX-2a \Tr[c]X^{-1}=X^{-1}Y$, which gives $2(1-a\Tr[X^{-1}])\Tr[c] =\Tr[X^{-1}Y].$ Last, the continuity property is obvious since $(X,a)\mapsto \mathcal{L}_{X,a}$ is continuous and $\mathcal{L} \mapsto \mathcal{L}^{-1}$ is continuous on $\{ \mathcal{L}:\symm \rightarrow \symm \text{ linear and invertible}\}$.  
\end{proof}

\begin{lemme}\label{lemme-autoadjoint}
For $X\in \dpos$, $\mathcal{L}_X$ is self-adjoint and positive definite: $$\Tr[\mathcal{L}_X(Y)Y] \ge 2\underline{\lambda}(X)\Tr[Y^2],$$ where $\underline{\lambda}(X)>0$ is the lowest eigenvalue of~$X$. Besides, for $ a < 1/\Tr[X^{-1}]$, $\mathcal{L}_{X,a}$ is self-adjoint and positive definite.
\end{lemme}
\begin{proof}
For $Y,Z \in\symm$, we have $\Tr[\mathcal{L}_X(Y)Z]=\Tr[(XY+YX)Z]=\Tr[Y(XZ+ZX)]=\Tr[Y\mathcal{L}_X(Z)]$ and $\Tr[\mathcal{L}_X(Y)Y]=2\Tr[XY^2]\ge 2 \underline{\lambda}(X)\Tr[Y^2]$ since $X-\underline{\lambda}(X)I_d\in \posm$. The self-adjoint property is then clear for $\mathcal{L}_{X,a}$, and the positive definiteness comes from Lemma~\ref{lemme-linear} and the continuity of the eigenvalues of $\mathcal{L}_{X,a}$ with respect to~$a$. 
\end{proof}

\begin{lemme}\label{lemme-autoadjoint2}
For $X\in \dpos$, $Y\in\genm$, $\bar{\cL}_X(Y)=\cL_X^{-1}(YX+XY\tp)X$ is self-adjoint and positive. The linear application  $\bar{\cL}_{X,a}(Y)=\cL_X^{-1}(YX+XY\tp)X-a\Tr[Y]I_d$ is also positive for $a<1/\Tr[X^{-1}]$, and there is a positive $c_{X,a}>0$ such that 
$$\Tr[\bar{\cL}_X(Y)\tp Y]\ge c_{X,a} \Tr[(\cL_X^{-1}(YX+XY\tp))^2].$$ 
\end{lemme}
\begin{proof}
Since $\cL_X^{-1}$ is self-adjoint, we have for $Z\in \genm$
\begin{align*}
\Tr[\bar{\cL}_{X}(Y)\tp Z]&=\Tr[ \cL_X^{-1}(YX+XY\tp) Z X ]=\frac{1}{2}\Tr[ \cL_X^{-1}(YX+XY\tp) (Z X+XZ\tp) ] \\
&=\frac{1}{2}\Tr[ (YX+XY\tp) \cL_X^{-1}(Z X+XZ\tp) ]=\Tr[Y \tp \bar{\cL}_X(Z)].
\end{align*}
Similarly, $\Tr[\bar{\cL}_{X,a}(Y)\tp Z]=\Tr[\bar{\cL}_{X}(Y)\tp Z]-a\Tr[Y]\Tr[Z]=\Tr[Y \tp \bar{\cL}_{X,a}(Z)]$. Besides, we notice that
\begin{align*}
\Tr[\bar{\cL}_{X,a}(Y)\tp Y]&=\frac{1}{2}\left(\Tr[ (YX+XY\tp) \cL_X^{-1}(Y X+XY\tp) ]-2a \Tr[Y]^2 \right) \\
&=\frac{1}{2}\left(\Tr[ \cL_{X,a}( \cL_X^{-1}(Y X+XY\tp) ) \cL_X^{-1}(Y X+XY\tp) ] \right) 
\end{align*}
by Lemma~\ref{lemme-linear}. This gives the claim since $\cL_{X,a}$ is positive definite by Lemma~\ref{lemme-autoadjoint}.
\end{proof}

The following lemma gives the Laplace transform of the matrix Normal distribution.

\begin{lemme}\label{lemme_matgauss}
Let $C \in \mathcal{S}_d^{+,*}$ and $\mathscr{C} [C] \in (\mathbb{R}^d)^{\otimes 4}$ defined by

\begin{align}
\label{eq:cov_matrix_normal}
\mathscr{C}[C]_{i,j,k,l}=\delta_{ik} C_{j,l}+\delta_{il} C_{j,k}+\delta_{jk} C_{i,l}+\delta_{jl} C_{i,k}  .
\end{align}

We introduce the $\mathcal{M}_d$-valued random variables $\tilde{\mathbf{G}}$ and $\mathbf{G} \sim \mathcal{N}(0, \mathscr{C}[C])$ of which components are Normal random variables with mean $0$ such that

\begin{align}
\forall i, j,k,l \in \{1, \dots ,d \}, \quad \mathbb{E}[\tilde{\mathbf{G}}_{i,j} \tilde{\mathbf{G}}_{k,l}] =\delta_{ik} \delta_{jl}, \quad \mathbb{E}[\mathbf{G}_{i,j} \mathbf{G}_{k,l}] =\mathscr{C}[C]_{i,j,k,l}.
\end{align} 
We have the following results.

\begin{enumerate}
\item For all $c \in \mathcal{S}_d$,  $\mathbb{E}\big[ \exp(-\Tr[c \mathbf{G} ]) \big]=\exp(2 \Tr[c^2 C])$.

\item For $\tilde{C} \in \genm$ such that $\tilde{C}\tilde{C}\tp =C$, $\tilde{C}\tilde{\mathbf{G}}+\tilde{\mathbf{G}}\tp \tilde{C}\tp $ and $\mathbf{G}$ have the same law.

\item Let $X\in \dpos$. For $c\in \symm$,  $\E[\exp(- \Tr[c \mathcal{L}_X^{-1}(\sqrt{X} \tilde{\mathbf{G}}+\tilde{\mathbf{G}}\tp \sqrt{X})])]=\E[\exp(\Tr[c\mathcal{L}_X^{-1}(c)])$
\end{enumerate}

\label{lemme:Lap_Normal_R2}
\end{lemme}

\begin{proof}
We focus on the first point. For all $c \in \mathcal{S}_d$, we have

\begin{align*}
\mathbb{E}\big[ \exp(-\Tr[c \mathbf{G} ]) \big]=\mathbb{E}\big[ \exp( -\sum_{1 \leqslant i,j \leqslant d}c_{i,j}\mathbf{G}_{i,j}) \big].
\end{align*}
Moreover, $\sum_{1 \leqslant i,j \leqslant d}c_{i,j}\mathbf{G}_{i,j}$ is a Normal random variable and its variance is given by
\begin{align*}
\mathbb{E}\big[ \big( \sum_{1 \leqslant i,j \leqslant d}c_{i,j}\mathbf{G}_{i,j} \big)^2 \big] =& \sum_{1 \leqslant i,j,k,l \leqslant d}c_{i,j}c_{k,l}\mathscr{C}[C]_{i,j,k,l}  = 4\Tr[c^2 C].
\end{align*}
It follows from the moment generating function of the Normal distribution that
\begin{align*}
\mathbb{E}\big[ \exp(-\Tr[c \mathbf{G} ]) \big]=\exp(2\Tr[c^2 C]).
\end{align*}

To prove the second point it is sufficient to notice that $\Tr[c( \tilde{C}\tilde{\mathbf{G}}+\tilde{\mathbf{G}}\tp \tilde{C}\tp )]= \Tr[2c \tilde{C}\tilde{\mathbf{G}}]$ and 
\begin{align*}
\mathbb{E}\big[ \big( \sum_{1 \leqslant i,j \leqslant d}(c \tilde{C})_{i,j}\tilde{\mathbf{G}}_{i,j} \big)^2 \big] =& \sum_{1 \leqslant i,j,k,l \leqslant d}(c \tilde{C})_{i,j}(c \tilde{C} )_{k,l}\delta_{ik}\delta_{jl} =\sum_{1 \leqslant i,j \leqslant d}(c \tilde{C})_{i,j}^2=\Tr[c \tilde{C} \tilde{C}\tp c] .
\end{align*}

For the third point, we set $Z=\mathcal{L}_X^{-1}(\sqrt{X} \tilde{\mathbf{G}}+\tilde{\mathbf{G}}\tp \sqrt{X} )$ and have $XZ+ZX=\sqrt{X} \tilde{\mathbf{G}}+\tilde{\mathbf{G}} \tp \sqrt{X}$. We also introduce $\tilde{c}=\mathcal{L}_X^{-1}(c)$ and have $\tilde{c}X+X \tilde{c}=c$. Thus, we obtain
\begin{align*}
\Tr[cZ]&=\Tr[(\tilde{c}X+X \tilde{c})Z] =\Tr[\tilde{c}(\sqrt{X} \tilde{\mathbf{G}}+\tilde{\mathbf{G}} \tp \sqrt{X})]
\end{align*}
and therefore $\E[\exp(-\Tr[cZ])]=\exp(2\Tr[\tilde{c}^2X])=\exp(\Tr[\tilde{c}(\tilde{c}X+X\tilde{c})])=\exp(\Tr[\tilde{c}c])$.

\end{proof}

\section{Some asymptotic behaviour of Wishart processes}

\begin{lemme}\label{lem_ergo} Let $X\sim WIS_d(x,\alpha,b,I_d)$ with $b\in \symm$, $x\in \posm$ and $\alpha \geb d-1$. Then $X_T$ converges in law when $T\rightarrow +\infty$ if and only if $-b \in \dpos$. In this case, $X_T$ converges in law to $WIS_d(0,\alpha,0,\sqrt{-b^{-1}};1/2)$.

 Let $X\sim WIS_d(x,\alpha,b,I_d)$ with $b\in \genm$, $x\in \posm$ and $\alpha \geb d-1$. If $-(b+b\tp)\in \dpos$, $q_\infty:=\int_0^\infty e^{sb}e^{sb\tp}ds$ is well defined and $X_T$ converges in law to $WIS_d(0,\alpha,0,\sqrt{2q_\infty};1/2)$.
\end{lemme}
\begin{proof}
Let us first consider the case $-b \in \dpos$. From  Proposition~4 in~\cite{AAP}, we have for $v\in \posm$,
\begin{align*}
  \E[\exp(-\Tr[vX_T])]&=\frac{\exp\left(\Tr\left[-v \left(I_d+2\left(\int_0^T e^{2bs}ds \right) v \right)^{-1}  e^{Tb}x e^{Tb}\right]\right)}{\det \left[I_d+2\left(\int_0^T e^{2bs}ds \right)v \right]^{\alpha/2}}\\
&\underset{T\rightarrow + \infty} \rightarrow\frac{1}{\det \left[I_d-b ^{-1} v  \right]^{\alpha/2}},
\end{align*}
which is the Laplace transform of $WIS_d(0,\alpha,0,\sqrt{-b^{-1}};1/2)$. Now, let us consider $-b \not \in \dpos$. Then, there exists an eigenvector  $v\in \R^d\setminus \{0\}$ such that $bv=\lambda v$ with $\lambda\geb 0$. Then, we have $\frac{d}{dt} \E[v\tp X_tv]=\alpha v\tp v +2\lambda  \E[v\tp X_tv] $, and therefore $  \E[v\tp X_Tv]\underset{T\rightarrow + \infty} \rightarrow + \infty$. 

In the case $b\in \genm$ with $-(b+b\tp)\in \dpos$, we know that the norm of $e^{bs}$ decays exponentially to $0$ as $s\rightarrow + \infty$, see e.g. Problem 11.3.6 in Golub and Van Loan~\cite{Golub}. Using again Proposition~4 in~\cite{AAP}, we get that $\E[\exp(-\Tr[vX_T])]\underset{T\rightarrow + \infty} \rightarrow\frac{1}{\det \left[I_d+ 2q_\infty v  \right]^{\alpha/2}}$.
\end{proof}

\begin{lemme}\label{lem_yor} $\bullet$ Assume $\alpha>d+1$ and $b=0$. Then, $\frac{Q_T^{-1}}{d \log(T)} \underset{T\rightarrow +\infty}{\rightarrow} \frac{1}{\alpha -(d+1)}$ a.s. Besides,  $\frac{Z_T}{\log(T)}$ converges almost surely to $d$, and we have
\begin{equation} \forall \mu>0,  \sup_{T\geb 2} \E\left[ \exp\left(  \frac{\mu }{ \sqrt{\log(T)}} N_T \right) \right] < \infty.\label{ui1}
\end{equation}

$\bullet$ Assume $\alpha=d+1$ and $b=0$. Then, as $T\rightarrow + \infty$, $\left(\frac{2}{d \log(T)} \right)^2 Q_T^{-1}$ converges in law to $\tau_1=\inf\{ t \geb 0, B_t=1 \}$, where $B$ is a Brownian motion. Besides, $\frac{Z_T}{\log(T)}=\frac{2N_T}{\log(T)}$ converges in probability to $d$, and we have \begin{equation} \forall \mu>0, \sup_{T\geb 2} \E\left[ \exp\left(  \frac{\mu }{ {\log(T)}} N_T \right) \right] < \infty.\label{ui2}
\end{equation}
\end{lemme}
We mention that the results on the convergence for $Q_T$ are given in Donati-Martin et al.~\cite{DDMY2003}. However, their proofs is in a working paper by the same authors that we have not been able to find. For this reason, we present here an autonomous proof. 
\begin{proof}
We first consider the case $\alpha>d+1$. We have $dX_t=\alpha I_d dt + \sqrt{X_t}dW_t + dW_t\tp \sqrt{X_t}$ and thus
$$d(e^{-t}X_{e^{t}-1})=[\alpha I_d -e^{-t}X_{e^{t}-1}]dt+\sqrt{e^{-t}X_{e^{t}-1}}d\tilde{W}_t+d\tilde{W}_t\tp  \sqrt{e^{-t}X_{e^{t}-1}},$$
with $d\tilde{W}_t= e^{-t/2}d(W_{e^{t}-1})$. 
We observe that $\tilde{W}$ is a matrix Brownian motion, which gives $Y \sim WIS_d(x,\alpha,-I_d/2,I_d)$, where $Y_t=e^{-t}X_{e^{t}-1}$ for $t\geb 0$. Using equation~\eqref{eq_Z} to the process $Y$, we get
\begin{equation}\label{log_detY} 
\frac{1}{t} \log\left( \frac{\det[Y_t] }{\det[Y_0]}\right) = (\alpha-1-d)\frac{1}{t}\int_0^t \Tr[Y^{-1}_s] ds -d +\frac{2}{t} \int_0^t\Tr[\sqrt{Y^{-1}_s}d\tilde{W}_s].
\end{equation}
Since~$Y$ is ergodic and $\langle \int_0^t\Tr[\sqrt{Y^{-1}_s}d\tilde{W}_s] \rangle=\int_0^t \Tr[Y^{-1}_s] ds$,
 we get that the left hand side converges in probability to zero and the right hand side converges a.s. to $(\alpha-1-d) \E[\Tr[Y^{-1}_\infty]]  -d$, where $Y_\infty\sim WIS_d(0,\alpha,0,\sqrt{2}I_d;1/2)$ is the stationary law of $Y$. Therefore, $\frac{1}{t} \log\left( \frac{\det[Y_t] }{\det[Y_0]}\right)$ converges a.s. to zero. Since $\frac{1}{t} \log\left( \frac{\det[Y_t] }{\det[Y_0]}\right)=\frac{1}{t} \log\left( \frac{\det[e^{-t}X_{e^{t}-1}] }{\det[x]}\right)=\frac{1}{t} \log\left( \frac{\det[X_{e^{t}-1}] }{\det[x]}\right)-d$, we get that $\frac{Z_T}{\log(T)}=\frac{1}{\log(T)} \log\left( \frac{\det[X_{T}] }{\det[x]} \right)$ converges a.s. to $d$ when $T\rightarrow +\infty$. 

Now, we use~\eqref{eq_Z} taken at time $T=e^t-1$ and Dubins-Schwarz theorem: there is a Brownian motion $\beta$ such that for all $t\geb 0$,
$$\frac{\alpha-(1+d)}{Q_{e^t-1}t}+\frac{2\beta_{Q^{-1}_{e^t-1}}}{t}=\frac{1}{t}\log\left( \frac{\det[X_{e^{t}-1}]}{\det[x]} \right).$$
This gives that $\frac{\alpha-(1+d)}{Q_{e^t-1}t}\underset{t\rightarrow +\infty}{\rightarrow}  d$ a.s., and therefore $\frac{Q_T^{-1}}{d \log(T)} \underset{T\rightarrow +\infty}{\rightarrow} \frac{1}{\alpha -(d+1)}$, a.s. 

It remains to prove~\eqref{ui1}. From~\eqref{eq_Z}, we have $N_T=\frac{Z_T}{2}-\frac{\alpha-1-d}{2}Q_T^{-1}\le\frac{Z_T}{2}$ and thus  $\E\left[ \exp\left(  \frac{\mu }{ \sqrt{\log(T)}} N_T \right) \right]\leb \E\left[\left( \frac{\det[X_{T}] }{\det[x]} \right)^{\frac{\mu }{ 2 \sqrt{\log(T)}}} \right]< \infty$, since the moments of~$X$ are bounded. Again we set $t=\log(T+1)$, and for $\Lambda \in [0,1]$, we have from~\eqref{log_detY}
\begin{align*}
N_T&= \int_0^T\Tr[\sqrt{X^{-1}_s}dW_s] = \int_0^t\Tr[\sqrt{Y^{-1}_s}d\tilde{W}_s]\\
&=\Lambda \int_0^t\Tr[\sqrt{Y^{-1}_s}d\tilde{W}_s]  +(1-\Lambda)\left(\frac{1}{2}\log\left( \frac{\det[Y_{t}] }{\det[x]} \right) +\frac{d}{2} t -\frac{\alpha-1-d}{2} \int_0^t \Tr[Y_s^{-1}]ds \right).
\end{align*}
By Cauchy-Schwarz inequality, we get
\begin{align*}
& \E\left[ \exp\left(  \frac{\mu }{ \sqrt{\log(T+1)}} N_T \right) \right] \\
\le& 
e^{\frac{\mu d (1-\Lambda)}{2} \sqrt{t}}
\E^{\frac{1}{2}}\left[ \left( \frac{\det[Y_{t}] }{\det[x]}  \right)^{(1-\Lambda)\frac{\mu}{\sqrt{t}}} \right]
 \\
&\times \E^{\frac{1}{2}}\left[ \exp\left( \frac{ 2 \mu \Lambda}{\sqrt{t}} \int_0^t\Tr[\sqrt{Y^{-1}_s}d\tilde{W}_s] - \mu(1-\Lambda)\frac{\alpha-1-d}{\sqrt{t}} \int_0^t \Tr[Y_s^{-1}]ds \right) \right].
\end{align*}
We now take $\Lambda=\Lambda_t=\frac{1}{2 \epsilon_t}\left(-1+\sqrt{1+4\epsilon_t}\right)$ with $\epsilon_t=\frac{2\mu}{(\alpha-1-d)\sqrt{t}}$ in order to obtain\linebreak $\frac12 \left(\frac{ 2 \mu \Lambda_t}{\sqrt{t}}  \right)^2=\mu(1-\Lambda_t)\frac{\alpha-1-d}{\sqrt{t}}.$ We note that for $t$ large enough, $\Lambda_t\in [0,1]$. Besides, we have $\Lambda_t\underset{t\rightarrow +\infty}{=}1-\epsilon_t + o(1/t)$, so that $\sqrt{t}(1-\Lambda_t)$ converges to  $\frac{2 \mu}{\alpha-1-d}$. From Theorem~4.1 in~\cite{MayerhoferHammamet}, the second expectation is then equal to~$1$, while the first one is bounded since $Y$ is ergodic. This yields to~\eqref{ui1}.

We now consider the case $\alpha=d+1$. We set again $t=\log(1+T)$ and have $T=e^t-1$. Thus, 
$$Z_T =\log\left( \frac{\det[X_T]}{\det[x]} \right)=\log\left( \frac{\det[e^tY_t]}{\det[x]} \right)=\log\left( \frac{\det[Y_t]}{\det[x]} \right)+dt.$$
Again,  $Y_t$ converges in law towards $WIS_d(0,\alpha, 0, \sqrt{2}I_d;1/2)$. Therefore, the ergodic theorem gives that $\frac{1}{t}\log\left( \frac{\det[Y_t]}{\det[x]} \right)$ converges in probability to $0$, which yields to the convergence in probability of $\frac{Z_T}{\log(T)}$ to $d$. 
We now turn to the convergence of $\left(\frac{2}{d \log(T)} \right)^2 Q_T^{-1}$. We know from Theorem 4.1 in Mayerhofer~\cite{MayerhoferHammamet} that for $T>0$ and $\lambda\geb 0$,
$$\E\left[ \exp\left( \frac{2 \lambda}{d \log(1+T)}N_T- \frac{(2\lambda)^2}{2 d^2\log(1+T)^2} Q_T^{-1} \right) \right]=1.$$
From~\eqref{eq_Z}, we have $N_T=Z_T/2$ and we write
\begin{align*}
1=&\E\left[ \exp\left( \lambda - \frac{(2\lambda)^2}{2 d^2\log(1+T)^2} Q_T^{-1} \right) \right] \\&+ \E\left[ \exp\left(- \frac{(2\lambda)^2}{2 d^2\log(1+T)^2} Q_T^{-1}\right) \left( \exp\left( \frac{2 \lambda}{d \log(1+T)}N_T \right)- \exp(\lambda) \right) \right]
\end{align*}
We now observe that $\exp\left(- \frac{(2\lambda)^2}{2 d^2\log(1+T)^2} Q_T^{-1}\right) \leb 1$ and that 
$$ \E\left[   \exp\left( \frac{2 \lambda}{d \log(1+T)}N_T \right)  \right] = \E\left[  \left(  \frac{\det[X_T]}{\det[x]}\right)^{ \frac{ \lambda}{d \log(1+T)}}  \right]= e^{ \lambda} \E\left[  \left(  \frac{\det[Y_t]}{\det[x]}\right)^{ \frac{ \lambda}{d t}}  \right]. $$
Since $Y$ has bounded moments and is stationary, $\sup_{t\geb 1}\E\left[  \left(  \frac{\det[Y_t]}{\det[x]}\right)^{ \frac{ \lambda}{d t}}  \right]< \infty $. This gives the uniform integrability~\eqref{ui2} and that $$\E\left[ \exp\left(- \frac{(2\lambda)^2}{2 d^2\log(1+T)^2} Q_T^{-1}\right) \left( \exp\left( \frac{2 \lambda}{d \log(1+T)}N_T \right)- \exp(\lambda) \right) \right] \underset{T\rightarrow +\infty}\rightarrow 0.$$

Therefore, $\lim_{T\rightarrow +\infty} \E\left[ \exp\left( \lambda - \frac{(2\lambda)^2}{2 d^2\log(1+T)^2} Q_T^{-1} \right) \right]=1$, which gives the desired convergence in law.
\end{proof}

\noindent {\bf Acknowledgements.} The authors would like to thank Arnaud Gloter  (University of Evry) and Marina Kleptsyna (University of Le Mans) for helpful discussions, and the two anonymous referees for their fruitful comments. 

\bibliography{ref_MLE_Wishart}

\def\cprime{$'$} \def\cprime{$'$}
\begin{thebibliography}{10}

\bibitem{AA2013}
A.~Ahdida and A.~Alfonsi.
\newblock Exact and high-order discretization schemes for {W}ishart processes
  and their affine extensions.
\newblock {\em Ann. Appl. Probab.}, 23(3):1025--1073, 2013.

\bibitem{AAP}
A.~Ahdida, A.~Alfonsi, and E.~Palidda.
\newblock Smile with the {G}aussian term structure model.
\newblock {\em Working paper}, 2014.

\bibitem{Alfonsibook}
A.~Alfonsi.
\newblock {\em Affine diffusions and related processes: simulation, theory and
  applications}, volume~6 of {\em Bocconi \& Springer Series}.
\newblock Springer, Cham; Bocconi University Press, Milan, 2015.

\bibitem{BK2011}
M.~Ben~Alaya and A.~Kebaier.
\newblock Parameter estimation for the square-root diffusions: ergodic and
  nonergodic cases.
\newblock {\em Stoch. Models}, 28(4):609--634, 2012.

\bibitem{Alaya_Kebaier_2013}
M.~Ben~Alaya and A.~Kebaier.
\newblock Asymptotic behavior of the maximum likelihood estimator for ergodic
  and nonergodic square-root diffusions.
\newblock {\em Stoch. Anal. Appl.}, 31(4):552--573, 2013.

\bibitem{Bruthesis}
M.~Bru.
\newblock {\em Th\`ese $3^{\text{\`eme}}$ cycle. R\'esistance d'Escherichia
  coli aux antibiotiques: Sensibilit\'es des analyses en composantes
  principales aux perturbations Browniennes et simulation.}
\newblock PhD thesis, Universit\'e Paris Nord, 1987.

\bibitem{Bru}
M.~Bru.
\newblock Wishart processes.
\newblock {\em J. Theoret. Probab.}, 4(4):725--751, 1991.

\bibitem{CFMT}
C.~Cuchiero, D.~Filipovi{\'c}, E.~Mayerhofer, and J.~Teichmann.
\newblock Affine processes on positive semidefinite matrices.
\newblock {\em Ann. Appl. Probab.}, 21(2):397--463, 2011.

\bibitem{DFGI}
J.~Da~Fonseca, M.~Grasselli, and F.~Ielpo.
\newblock Estimating the {W}ishart affine stochastic correlation model using
  the empirical characteristic function.
\newblock {\em Stud. Nonlinear Dyn. Econom.}, 18(3):253--289, 2014.

\bibitem{Dafonseca}
J.~Da~Fonseca, M.~Grasselli, and C.~Tebaldi.
\newblock Option pricing when correlations are stochastic: an analytical
  framework.
\newblock {\em Review of Derivatives Research}, 10:151--180, 2008.

\bibitem{DieciEirola}
L.~Dieci and T.~Eirola.
\newblock Positive definiteness in the numerical solution of {R}iccati
  differential equations.
\newblock {\em Numer. Math.}, 67(3):303--313, 1994.

\bibitem{Dohnal}
G.~Dohnal.
\newblock On estimating the diffusion coefficient.
\newblock {\em J. Appl. Probab.}, 24(1):105--114, 1987.

\bibitem{DDMY2003}
C.~Donati-Martin, Y.~Doumerc, H.~Matsumoto, and M.~Yor.
\newblock Some properties of the {W}ishart processes and a matrix extension of
  the {H}artman-{W}atson laws.
\newblock {\em Publ. Res. Inst. Math. Sci.}, 40(4):1385--1412, 2004.

\bibitem{FournieTalay}
E.~Fourni\'e and D.~Talay.
\newblock Application de la statistique des diffusions \`a un mod\`ele de taux
  d'inter\^et.
\newblock {\em Finance}, 12(2):79--111, 1991.

\bibitem{GCJ}
V.~Genon-Catalot and J.~Jacod.
\newblock On the estimation of the diffusion coefficient for multi-dimensional
  diffusion processes.
\newblock {\em Ann. Inst. H. Poincar\'e Probab. Statist.}, 29(1):119--151,
  1993.

\bibitem{Gnoatto}
A.~Gnoatto.
\newblock The {W}ishart short rate model.
\newblock {\em Int. J. Theor. Appl. Finance}, 15(8):1250056, 24, 2012.

\bibitem{GnoattoGrasselli}
A.~Gnoatto and M.~Grasselli.
\newblock The explicit {L}aplace transform for the {W}ishart process.
\newblock {\em J. Appl. Probab.}, 51(3):640--656, 2014.

\bibitem{Golub}
G.~H. Golub and C.~F. Van~Loan.
\newblock {\em Matrix computations}.
\newblock Johns Hopkins Studies in the Mathematical Sciences. Johns Hopkins
  University Press, Baltimore, MD, third edition, 1996.

\bibitem{GourierouxSufana}
C.~Gourieroux and R.~Sufana.
\newblock Derivative pricing with {W}ishart multivariate stochastic volatility.
\newblock {\em J. Bus. Econom. Statist.}, 28(3):438--451, 2010.

\bibitem{GourierouxSufana2}
C.~Gourieroux and R.~Sufana.
\newblock Discrete time {W}ishart term structure models.
\newblock {\em J. Econom. Dynam. Control}, 35(6):815--824, 2011.

\bibitem{Heston}
S.~Heston.
\newblock A closed-form solution for options with stochastic volatility with
  applications to bond and currency options.
\newblock {\em The Review of Financial Studies}, 6(2):327--343, 1993.

\bibitem{Jeg}
P.~Jeganathan.
\newblock Some aspects of asymptotic theory with applications to time series
  models.
\newblock {\em Econometric Theory}, 11(5):818--887, 1995.
\newblock Trending multiple time series (New Haven, CT, 1993).

\bibitem{KutoyantsBook}
Y.~A. Kutoyants.
\newblock {\em Statistical inference for ergodic diffusion processes}.
\newblock Springer Series in Statistics. Springer-Verlag London, Ltd., London,
  2004.

\bibitem{Lecam}
L.~Le~Cam.
\newblock Locally asymptotically normal families of distributions. {C}ertain
  approximations to families of distributions and their use in the theory of
  estimation and testing hypotheses.
\newblock {\em Univ. california Publ. Statist.}, 3:37--98, 1960.

\bibitem{LecYan}
L.~Le~Cam and G.~L. Yang.
\newblock {\em Asymptotics in statistics}.
\newblock Springer Series in Statistics. Springer-Verlag, New York, second
  edition, 2000.
\newblock Some basic concepts.

\bibitem{Levin}
J.~J. Levin.
\newblock On the matrix {R}iccati equation.
\newblock {\em Proc. Amer. Math. Soc.}, 10:519--524, 1959.

\bibitem{LS_Book}
R.~S. Liptser and A.~N. Shiryaev.
\newblock {\em Statistics of random processes. {I},{II}}.
\newblock Applications of Mathematics (New York). Springer-Verlag, Berlin,
  expanded edition, 2001.
\newblock Applications, Translated from the 1974 Russian original by A. B.
  Aries, Stochastic Modelling and Applied Probability.

\bibitem{Luc92}
H.~Luschgy.
\newblock Local asymptotic mixed normality for semimartingale experiments.
\newblock {\em Probab. Theory Related Fields}, 92(2):151--176, 1992.

\bibitem{MayerhoferHammamet}
E.~{Mayerhofer}.
\newblock {Wishart Processes and Wishart Distributions: An Affine Processes
  Point of View}.
\newblock {\em ArXiv e-prints}, Jan. 2012.

\bibitem{MondPecaric}
B.~Mond and J.~E. Pe{\v{c}}ari{\'c}.
\newblock On matrix convexity of the {M}oore-{P}enrose inverse.
\newblock {\em Internat. J. Math. Math. Sci.}, 19(4):707--710, 1996.

\bibitem{Overbeck}
L.~Overbeck.
\newblock Estimation for continuous branching processes.
\newblock {\em Scand. J. Statist.}, 25(1):111--126, 1998.

\bibitem{Pages_ergo}
G.~Pag{\`e}s.
\newblock Sur quelques algorithmes r\'ecursifs pour les probabilit\'es
  num\'eriques.
\newblock {\em ESAIM Probab. Statist.}, 5:141--170 (electronic), 2001.

\bibitem{Rogers_Williams_2000_Vol2}
L.~C.~G. Rogers and D.~Williams.
\newblock {\em Diffusions, {M}arkov processes, and martingales. {V}ol. 2}.
\newblock Cambridge Mathematical Library. Cambridge University Press,
  Cambridge, 2000.
\newblock It{\^o} calculus, Reprint of the second (1994) edition.

\bibitem{Rydberg}
T.~H. Rydberg.
\newblock A note on the existence of unique equivalent martingale measures in a
  markovian setting.
\newblock {\em Finance and Stochastics}, 1(3):251--257, 1997.

\bibitem{Stroock}
D.~W. Stroock.
\newblock {\em Probability theory, an analytic view}.
\newblock Cambridge University Press, Cambridge, 1993.

\end{thebibliography}
\bibliographystyle{abbrv} 

\end{document}